\pdfoutput=1
\documentclass[reqno, 12pt]{amsart}
\def\ssign{\textsection\nobreak\hspace{1pt plus 0.3pt}}
\makeatletter
\newif\if@sectsign \@sectsigntrue
\def\@seccntformat#1{\protect\textup{\protect\@secnumfont
		\ifnum\pdfstrcmp{#1}{section}=\z@
			\if@sectsign \protect\ssign\csname the#1\endcsname\protect\enspace
			\else \csname the#1\endcsname\protect\@secnumpunct\fi
		\else
			\csname the#1\endcsname\protect\@secnumpunct
		\fi
	}}
\g@addto@macro\appendix{\@sectsignfalse}\makeatother

\usepackage{amsmath,amssymb,amsthm}
\usepackage{mathrsfs}
\usepackage{mathabx}\changenotsign
\usepackage{dsfont}

\usepackage[T1]{fontenc}
\usepackage{lmodern}
\usepackage[british]{babel}
\usepackage[babel]{microtype}

\usepackage{geometry}
\usepackage{float}
\usepackage{tikz}
\usepackage{subcaption}
\usetikzlibrary{quotes}
\usetikzlibrary {graphs}
\usetikzlibrary{backgrounds}
\usetikzlibrary{fadings}
\pgfdeclarelayer{foreground}
\pgfsetlayers{background,main,foreground}
\tikzgraphsset{
	empty nodes,
nodes={circle, fill, draw, inner sep=0pt, minimum size=2.5pt},
}

\usepackage{xcolor}
\usepackage[backref]{hyperref}
\hypersetup{
	colorlinks,
	linkcolor={red!60!black},
	citecolor={green!60!black},
	urlcolor={blue!60!black}
}

\usepackage[open,openlevel=2,atend]{bookmark}

\usepackage[abbrev,msc-links,backrefs]{amsrefs}
\renewcommand{\MR}[1]{}

\usepackage{doi}

\renewcommand{\PrintDOI}[1]{\doi{#1}}

\numberwithin{equation}{section}
\numberwithin{figure}{section}

\usepackage{calc}\usepackage{enumitem}
\def\rmlabel{\upshape({\itshape \roman*\,})}

\def\alabel{\upshape(\makebox[\widthof{\itshape a}][c]{\itshape \alph*}\,)}

\def\nlabel{\upshape({\itshape \arabic*\,})}

\theoremstyle{plain}
\newtheorem{thm}{Theorem}[section]
\newtheorem{fact}[thm]{Fact}
\newtheorem{prop}[thm]{Proposition}
\newtheorem{cor}[thm]{Corollary}
\newtheorem{lemma}[thm]{Lemma}

\theoremstyle{definition}
\newtheorem{dfn}[thm]{Definition}

\newtheorem*{nquest}{Question}

\theoremstyle{remark}
\newtheorem{rem}[thm]{Remark}
\newtheorem{clm}[thm]{Claim}

\let\theta=\vartheta
\let\rho=\varrho
\let\phi=\varphi

\let\Ups=\varUpsilon

\def\NN{\mathds N}
\def\ZZ{\mathds Z}

\makeatletter
\def\@defcal#1{\expandafter\def\csname c#1\endcsname{{\mathcal{#1}}}}
\def\@defscr#1{\expandafter\def\csname cc#1\endcsname{{\mathscr{#1}}}}
\def\@deffrak#1{\expandafter\def\csname f#1\endcsname{{\mathfrak{#1}}}}
\@tfor\next:=ABCDEFGHIJKLMNOPQRSTUVWXYZ\do{\expandafter\@defcal\next}
\@tfor\next:=ABCDEFGHIJKLMNOPQRSTUVWXYZ\do{\expandafter\@defscr\next}
\@tfor\next:=ABCDEFGHIJKLMNOPQRSTUVWXYZ\do{\expandafter\@deffrak\next}
\@tfor\next:=abcdefghjklmnopqrstuvwxyz\do{\expandafter\@deffrak\next}\makeatother

\def\Fs{F_{\star}}

\let\polishlcross=\l
\DeclareRobustCommand{\l}{\ifmmode\ell\else\polishlcross\fi}
\ifdefined{\def\l{l}}\fi

\makeatletter
\def\moverlay{\mathpalette\mov@rlay}
\def\mov@rlay#1#2{\leavevmode\vtop{   \baselineskip\z@skip
		\lineskiplimit-\maxdimen
\ialign{\hfil$\m@th#1##$\hfil\cr#2\crcr}}}
\newcommand{\charfusion}[3][\mathord]{
	#1{\ifx#1\mathop\vphantom{#2}\fi
		\mathpalette\mov@rlay{#2\cr#3}
	}
\ifx#1\mathop\expandafter\displaylimits\fi}
\makeatother

\newcommand{\dcup}{\charfusion[\mathbin]{\cup}{\cdot}}
\newcommand{\bigdcup}{\charfusion[\mathop]{\bigcup}{\cdot}}

\def\tand{\ \text{and}\ }

\def\qqand{\qquad\text{and}\qquad}

\usepackage{scalerel}
\usepackage{graphicx}
\usepackage{trimclip}
\makeatletter
\newcommand{\hpsc}{0.6}    \newcommand{\hpgapn}{3}    \newcommand{\hpinL}{0.1}   \newcommand{\hpinR}{0.0}   \newcommand{\hpbot}{0.68}  \newcommand{\hpbarb}{0.55} \newcommand{\hpoh}{0.1}    \newcommand{\hpminf}{0.5}  \newcommand{\hpov}{0.04}   \newcommand{\hpdx}{0}      \newsavebox\hpA\newsavebox\hpG\newsavebox\hpB\newsavebox\hpF\newsavebox\hpL\newsavebox\hpSl
\newdimen\hp@li\newdimen\hp@L\newdimen\hp@gap\newdimen\hp@dx\newdimen\hp@bs\newdimen\hp@oh\newdimen\hp@rw\newdimen\hp@ov\newdimen\hp@ri
\def\hp@nil{}
\def\@hpfirst#1#2\hp@nil{\sbox\hpF{$\m@th\SavedStyle#1$}}
\def\@hplast#1{\ifx#1\hp@nil\else\def\@hplastv{#1}\expandafter\@hplast\fi}
\newcommand{\hp@common}[3]{\sbox\hpA{$\m@th\SavedStyle#3$}\@hpfirst#3\hp@nil \@hplast#3\hp@nil \sbox\hpL{$\m@th\SavedStyle\@hplastv$}\hp@li=#1\wd\hpF \hp@ri=#2\wd\hpL
  \hp@L=\wd\hpA \advance\hp@L-\hp@li \advance\hp@L-\hp@ri
  \@tempdimb=\hpminf\wd\hpG \ifdim\hp@L<\@tempdimb \hp@L=\@tempdimb \fi
  \hp@bs=\hpbarb\wd\hpG \hp@oh=\hpoh\ht\hpG
  \hp@rw=\hp@L \advance\hp@rw-\wd\hpG \advance\hp@rw\hp@bs
\dimen@=\fontdimen8\textfont3
  \if S\m@switch\dimen@=\fontdimen8\scriptfont3\fi \if s\m@switch\dimen@=\fontdimen8\scriptscriptfont3\fi
  \hp@gap=\hpgapn\dimen@ \@tempdima=\hpbot\ht\hpG \advance\hp@gap-\@tempdima
  \hp@dx=\hpdx\wd\hpA \hp@ov=\hpov\wd\hpG}
\newcommand{\hpdrawR}{\ifdim\hp@L<\wd\hpG
    \clipbox{\dimexpr\wd\hpG-\hp@L\relax{} 0pt 0pt \dimexpr-\hp@oh\relax{}}{\usebox\hpG}\else
    \hbox to\dimexpr\hp@rw+\hp@ov\relax{\cleaders\hbox{\usebox\hpSl}\hfil}\kern-\hp@ov\usebox\hpB \fi}
\newcommand{\hp@bR}[3]{\sbox\hpG{\scalebox{\hpsc}{$\m@th\SavedStyle\rightharpoonup$}}\hp@common{#1}{#2}{#3}\sbox\hpB{\clipbox{\dimexpr\hp@bs\relax{} 0pt 0pt \dimexpr-\hp@oh\relax{}}{\usebox\hpG}}\sbox\hpSl{\clipbox{\dimexpr\wd\hpG*1/5\relax{} 0pt \dimexpr\wd\hpG*39/50\relax{} 0pt}{\usebox\hpG}}\mathord{\vbox{\offinterlineskip\ialign{##\cr
    \hbox to\wd\hpA{\kern\dimexpr\wd\hpA-\hp@ri-\hp@L+\hp@dx\relax\hpdrawR\hfil}\cr
    \noalign{\kern\hp@gap}\hbox{\usebox\hpA}\cr}}}}
\NewDocumentCommand\rharp{O{\hpinL}O{\hpinR}m}{\ThisStyle{\hp@bR{#1}{#2}{#3}}}
\makeatother

\def\vF{\rharp[0.24][0.0]{F}}

\def\bl{\bigl(}
\def\br{\bigr)}

\let\emptyset=\varnothing
\let\vn=\varnothing

\let\lra=\longrightarrow
\let\to=\lra

\newlength\tarrowwd \newlength\tarrowmin \newlength\tarrowdp
\DeclareRobustCommand{\tarrow}[2][-1.6pt]{\settowidth\tarrowwd{$\scriptstyle#2$}\addtolength\tarrowwd{10pt}\settowidth\tarrowmin{$\longrightarrow$}\settodepth{\tarrowdp}{$\scriptstyle#2$}\mathrel{\mathop{\ifdim\tarrowwd>\tarrowmin \hbox to \tarrowwd{\rightarrowfill}\else \longrightarrow \fi}\limits^{\vphantom{f}\raisebox{\dimexpr#1+\tarrowdp\relax}[0pt][0pt]{$\scriptstyle#2$}}}}

\DeclareMathSymbol{\cstar}{\mathbin}{symbols}{"03}
\DeclareRobustCommand{\clra}{\tarrow[-2.1pt]{\cstar}}

\DeclareSymbolFont{stmry}{U}{stmry}{m}{n}
\DeclareMathSymbol\arrownot\mathrel{stmry}{"58}
\DeclareMathSymbol\Arrownot\mathrel{stmry}{"59}

\makeatletter
\newcommand{\pushright}[1]{\ifmeasuring@#1\else\omit\hfill$\displaystyle#1$\fi\ignorespaces}
\newcommand{\pushleft}[1]{\ifmeasuring@#1\else\omit$\displaystyle#1$\hfill\fi\ignorespaces}
\makeatother

\DeclareFontFamily{U}  {MnSymbolC}{}
\DeclareSymbolFont{MnSyC}         {U}  {MnSymbolC}{m}{n}
\DeclareFontShape{U}{MnSymbolC}{m}{n}{
	<-6>  MnSymbolC5
	<6-7>  MnSymbolC6
	<7-8>  MnSymbolC7
	<8-9>  MnSymbolC8
	<9-10> MnSymbolC9
	<10-12> MnSymbolC10
<12->   MnSymbolC12}{}
\DeclareMathSymbol{\powerset}{\mathord}{MnSyC}{180}

\DeclareMathOperator{\diam}{diam}
\DeclareMathOperator{\dist}{dist}

\newsavebox\separatedbox
\savebox\separatedbox{\tikz{
		\draw[black,fill=black] (180:2.25) circle (.35);
		\draw[black,fill=black] (180:0.75) circle (.35);
		\draw[black,fill=black] (0:0.75) circle (.35);
		\draw[black,fill=black] (0:2.25) circle (.35);
		\draw[black,line width=0.2cm] (180:2.25)--(180:0.75);
		\draw[black,line width=0.2cm] (0:2.25)--(0:0.75);
		\draw[opacity=0] (270:0.6) circle (0.1);
		\draw[opacity=0] (90:1.25) circle (0.1);
}}
\newcommand{\separated}{\protect\mathord{\scaleobj{1.2}{\scalerel*{\usebox{\separatedbox}}{x}}}}
\DeclareRobustCommand{\rseparated}{\separated}

\newsavebox\overtopbox
\savebox\overtopbox{\tikz{
		\coordinate (x) at (180:2.25);
		\coordinate (y) at (180:0.75);
		\coordinate (z) at (0:0.75);
		\coordinate (w) at (0:2.25);
		\draw[black,fill=black] (180:2.25) circle (.35);
		\draw[black,fill=black] (180:0.75) circle (.35);
		\draw[black,fill=black] (0:0.75) circle (.35);
		\draw[black,fill=black] (0:2.25) circle (.35);
		\draw[black,line width=0.2cm] (180:0.75)--(0:0.75);
		\draw[black,line width=0.2cm] (x) to[out=+45,in=+135] (w);
		\draw[opacity=0] (270:0.6) circle (0.1);
		\draw[opacity=0] (90:1.25) circle (0.1);
}}
\newcommand{\overtop}{\mathord{\scaleobj{1.2}{\scalerel*{\usebox{\overtopbox}}{x}}}}
\DeclareRobustCommand{\rovertop}{\overtop}

\newsavebox\crossingbox
\savebox\crossingbox{\tikz{
		\coordinate (x) at (180:2.25);
		\coordinate (y) at (180:0.75);
		\coordinate (z) at (0:0.75);
		\coordinate (w) at (0:2.25);
		\draw[black,fill=black] (180:2.25) circle (.35);
		\draw[black,fill=black] (180:0.75) circle (.35);
		\draw[black,fill=black] (0:0.75) circle (.35);
		\draw[black,fill=black] (0:2.25) circle (.35);
		\draw[black,line width=0.2cm] (x) to[out=+80,in=+100] (z);
		\draw[black,line width=0.2cm] (y) to[out=+80,in=+100] (w);
		\draw[opacity=0] (270:0.6) circle (0.1);
		\draw[opacity=0] (90:1.25) circle (0.1);
}}
\newcommand{\crossing}{\mathord{\scaleobj{1.2}{\scalerel*{\usebox{\crossingbox}}{x}}}}
\DeclareRobustCommand{\rcrossing}{\crossing}

\newsavebox\ptwobox
\savebox\ptwobox{\tikz{
		\coordinate (x) at (180:1.5);
		\coordinate (y) at (0,0);
		\coordinate (z) at (0:1.5);
		\draw[black,fill=black] (x) circle (.35);
		\draw[black,fill=black] (y) circle (.35);
		\draw[black,fill=black] (z) circle (.35);
		\draw[black,line width=0.2cm] (x) to (y);
		\draw[black,line width=0.2cm] (y) to (z);
		\draw[opacity=0] (270:0.6) circle (0.1);
		\draw[opacity=0] (90:1.25) circle (0.1);
}}
\newcommand{\ptwo}{\mathord{\scaleobj{1.2}{\scalerel*{\usebox{\ptwobox}}{x}}}}
\DeclareRobustCommand{\rptwo}{\ptwo}

\newsavebox\lcherrybox
\savebox\lcherrybox{\tikz{
		\coordinate (x) at (180:1.5);
		\coordinate (y) at (0,0);
		\coordinate (z) at (0:1.5);
		\draw[black,fill=black] (x) circle (.35);
		\draw[black,fill=black] (y) circle (.35);
		\draw[black,fill=black] (z) circle (.35);
		\draw[black,line width=0.2cm] (x) to (y);
		\draw[black,line width=0.2cm] (x) to[out=+80,in=+100] (z);
		\draw[opacity=0] (270:0.6) circle (0.1);
		\draw[opacity=0] (90:1.25) circle (0.1);
}}
\newcommand{\lcherry}{\mathord{\scaleobj{1.2}{\scalerel*{\usebox{\lcherrybox}}{x}}}}
\DeclareRobustCommand{\rlcherry}{\lcherry}

\newsavebox\rcherrybox
\savebox\rcherrybox{\tikz{
		\coordinate (x) at (180:1.5);
		\coordinate (y) at (0,0);
		\coordinate (z) at (0:1.5);
		\draw[black,fill=black] (x) circle (.35);
		\draw[black,fill=black] (y) circle (.35);
		\draw[black,fill=black] (z) circle (.35);
		\draw[black,line width=0.2cm] (x) to[out=+80,in=+100] (z);
		\draw[black,line width=0.2cm] (y) to (z);
		\draw[opacity=0] (270:0.6) circle (0.1);
		\draw[opacity=0] (90:1.25) circle (0.1);
}}
\newcommand{\rcherry}{\mathord{\scaleobj{1.2}{\scalerel*{\usebox{\rcherrybox}}{x}}}}
\DeclareRobustCommand{\rrcherry}{\rcherry}

\newsavebox\lcherryredbox
\savebox\lcherryredbox{\tikz{
		\coordinate (x) at (180:1.5);
		\coordinate (y) at (0,0);
		\coordinate (z) at (0:1.5);
		\draw[red,line width=0.2cm] (x) to (y);
		\draw[red,line width=0.2cm] (x) to[out=+80,in=+100] (z);
		\draw[black,fill=black] (x) circle (.35);
		\draw[black,fill=black] (y) circle (.35);
		\draw[black,fill=black] (z) circle (.35);
		\draw[opacity=0] (270:0.6) circle (0.1);
		\draw[opacity=0] (90:1.25) circle (0.1);
}}
\newcommand{\lcherryred}{\mathord{\scaleobj{1.2}{\scalerel*{\usebox{\lcherryredbox}}{x}}}}
\newsavebox\lcherryvioletbox
\savebox\lcherryvioletbox{\tikz{
		\coordinate (x) at (180:1.5);
		\coordinate (y) at (0,0);
		\coordinate (z) at (0:1.5);
		\draw[violet,line width=0.2cm] (x) to (y);
		\draw[violet,line width=0.2cm] (x) to[out=+80,in=+100] (z);
		\draw[black,fill=black] (x) circle (.35);
		\draw[black,fill=black] (y) circle (.35);
		\draw[black,fill=black] (z) circle (.35);
		\draw[opacity=0] (270:0.6) circle (0.1);
		\draw[opacity=0] (90:1.25) circle (0.1);
}}
\newcommand{\lcherryviolet}{\mathord{\scaleobj{1.2}{\scalerel*{\usebox{\lcherryvioletbox}}{x}}}}
\newsavebox\rcherrymixbox
\savebox\rcherrymixbox{\tikz{
		\coordinate (x) at (180:1.5);
		\coordinate (y) at (0,0);
		\coordinate (z) at (0:1.5);
		\draw[red,line width=0.2cm] (x) to[out=+80,in=+100] (z);
		\draw[green!45!black,line width=0.2cm] (y) to (z);
		\draw[black,fill=black] (x) circle (.35);
		\draw[black,fill=black] (y) circle (.35);
		\draw[black,fill=black] (z) circle (.35);
		\draw[opacity=0] (270:0.6) circle (0.1);
		\draw[opacity=0] (90:1.25) circle (0.1);
}}
\newcommand{\rcherrymix}{\mathord{\scaleobj{1.2}{\scalerel*{\usebox{\rcherrymixbox}}{x}}}}

\newsavebox\crossinggreenredbox
\savebox\crossinggreenredbox{\tikz{
		\coordinate (x) at (180:2.25);
		\coordinate (y) at (180:0.75);
		\coordinate (z) at (0:0.75);
		\coordinate (w) at (0:2.25);
		\draw[green!45!black,line width=0.2cm] (x) to[out=+80,in=+100] (z);
		\draw[red,line width=0.2cm] (y) to[out=+80,in=+100] (w);
		\draw[black,fill=black] (x) circle (.35);
		\draw[black,fill=black] (y) circle (.35);
		\draw[black,fill=black] (z) circle (.35);
		\draw[black,fill=black] (w) circle (.35);
		\draw[opacity=0] (270:0.6) circle (0.1);
		\draw[opacity=0] (90:1.25) circle (0.1);
}}
\newcommand{\crossinggreenred}{\mathord{\scaleobj{1.2}{\scalerel*{\usebox{\crossinggreenredbox}}{x}}}}
\newsavebox\crossingredbox
\savebox\crossingredbox{\tikz{
		\coordinate (x) at (180:2.25);
		\coordinate (y) at (180:0.75);
		\coordinate (z) at (0:0.75);
		\coordinate (w) at (0:2.25);
		\draw[red,line width=0.2cm] (x) to[out=+80,in=+100] (z);
		\draw[red,line width=0.2cm] (y) to[out=+80,in=+100] (w);
		\draw[black,fill=black] (x) circle (.35);
		\draw[black,fill=black] (y) circle (.35);
		\draw[black,fill=black] (z) circle (.35);
		\draw[black,fill=black] (w) circle (.35);
		\draw[opacity=0] (270:0.6) circle (0.1);
		\draw[opacity=0] (90:1.25) circle (0.1);
}}
\newcommand{\crossingred}{\mathord{\scaleobj{1.2}{\scalerel*{\usebox{\crossingredbox}}{x}}}}
\newsavebox\crossingvioletbox
\savebox\crossingvioletbox{\tikz{
		\coordinate (x) at (180:2.25);
		\coordinate (y) at (180:0.75);
		\coordinate (z) at (0:0.75);
		\coordinate (w) at (0:2.25);
		\draw[violet,line width=0.2cm] (x) to[out=+80,in=+100] (z);
		\draw[violet,line width=0.2cm] (y) to[out=+80,in=+100] (w);
		\draw[black,fill=black] (x) circle (.35);
		\draw[black,fill=black] (y) circle (.35);
		\draw[black,fill=black] (z) circle (.35);
		\draw[black,fill=black] (w) circle (.35);
		\draw[opacity=0] (270:0.6) circle (0.1);
		\draw[opacity=0] (90:1.25) circle (0.1);
}}
\newcommand{\crossingviolet}{\mathord{\scaleobj{1.2}{\scalerel*{\usebox{\crossingvioletbox}}{x}}}}

\newsavebox\separatedgreenredbox
\savebox\separatedgreenredbox{\tikz{
		\draw[green!45!black,line width=0.2cm] (180:2.25)--(180:0.75);
		\draw[red,line width=0.2cm] (0:2.25)--(0:0.75);
		\draw[black,fill=black] (180:2.25) circle (.35);
		\draw[black,fill=black] (180:0.75) circle (.35);
		\draw[black,fill=black] (0:0.75) circle (.35);
		\draw[black,fill=black] (0:2.25) circle (.35);
		\draw[opacity=0] (270:0.6) circle (0.1);
		\draw[opacity=0] (90:1.25) circle (0.1);
}}
\newcommand{\separatedgreenred}{\mathord{\scaleobj{1.2}{\scalerel*{\usebox{\separatedgreenredbox}}{x}}}}
\newsavebox\separatedredbox
\savebox\separatedredbox{\tikz{
		\draw[red,line width=0.2cm] (180:2.25)--(180:0.75);
		\draw[red,line width=0.2cm] (0:2.25)--(0:0.75);
		\draw[black,fill=black] (180:2.25) circle (.35);
		\draw[black,fill=black] (180:0.75) circle (.35);
		\draw[black,fill=black] (0:0.75) circle (.35);
		\draw[black,fill=black] (0:2.25) circle (.35);
		\draw[opacity=0] (270:0.6) circle (0.1);
		\draw[opacity=0] (90:1.25) circle (0.1);
}}
\newcommand{\separatedred}{\mathord{\scaleobj{1.2}{\scalerel*{\usebox{\separatedredbox}}{x}}}}
\newsavebox\separatedvioletredbox
\savebox\separatedvioletredbox{\tikz{
		\draw[violet,line width=0.2cm] (180:2.25)--(180:0.75);
		\draw[red,line width=0.2cm] (0:2.25)--(0:0.75);
		\draw[black,fill=black] (180:2.25) circle (.35);
		\draw[black,fill=black] (180:0.75) circle (.35);
		\draw[black,fill=black] (0:0.75) circle (.35);
		\draw[black,fill=black] (0:2.25) circle (.35);
		\draw[opacity=0] (270:0.6) circle (0.1);
		\draw[opacity=0] (90:1.25) circle (0.1);
}}
\newcommand{\separatedvioletred}{\mathord{\scaleobj{1.2}{\scalerel*{\usebox{\separatedvioletredbox}}{x}}}}

\newsavebox\ptwogreenredbox
\savebox\ptwogreenredbox{\tikz{
		\coordinate (x) at (180:1.5);
		\coordinate (y) at (0,0);
		\coordinate (z) at (0:1.5);
		\draw[green!45!black,line width=0.2cm] (x) to (y);
		\draw[red,line width=0.2cm] (y) to (z);
		\draw[black,fill=black] (x) circle (.35);
		\draw[black,fill=black] (y) circle (.35);
		\draw[black,fill=black] (z) circle (.35);
		\draw[opacity=0] (270:0.6) circle (0.1);
		\draw[opacity=0] (90:1.25) circle (0.1);
}}
\newcommand{\ptwogreenred}{\mathord{\scaleobj{1.2}{\scalerel*{\usebox{\ptwogreenredbox}}{x}}}}
\newsavebox\ptworedbox
\savebox\ptworedbox{\tikz{
		\coordinate (x) at (180:1.5);
		\coordinate (y) at (0,0);
		\coordinate (z) at (0:1.5);
		\draw[red,line width=0.2cm] (x) to (y);
		\draw[red,line width=0.2cm] (y) to (z);
		\draw[black,fill=black] (x) circle (.35);
		\draw[black,fill=black] (y) circle (.35);
		\draw[black,fill=black] (z) circle (.35);
		\draw[opacity=0] (270:0.6) circle (0.1);
		\draw[opacity=0] (90:1.25) circle (0.1);
}}
\newcommand{\ptwored}{\mathord{\scaleobj{1.2}{\scalerel*{\usebox{\ptworedbox}}{x}}}}
\newsavebox\ptwovioletbox
\savebox\ptwovioletbox{\tikz{
		\coordinate (x) at (180:1.5);
		\coordinate (y) at (0,0);
		\coordinate (z) at (0:1.5);
		\draw[violet,line width=0.2cm] (x) to (y);
		\draw[violet,line width=0.2cm] (y) to (z);
		\draw[black,fill=black] (x) circle (.35);
		\draw[black,fill=black] (y) circle (.35);
		\draw[black,fill=black] (z) circle (.35);
		\draw[opacity=0] (270:0.6) circle (0.1);
		\draw[opacity=0] (90:1.25) circle (0.1);
}}
\newcommand{\ptwoviolet}{\mathord{\scaleobj{1.2}{\scalerel*{\usebox{\ptwovioletbox}}{x}}}}
\newsavebox\ptwovioletredbox
\savebox\ptwovioletredbox{\tikz{
		\coordinate (x) at (180:1.5);
		\coordinate (y) at (0,0);
		\coordinate (z) at (0:1.5);
		\draw[violet,line width=0.2cm] (x) to (y);
		\draw[red,line width=0.2cm] (y) to (z);
		\draw[black,fill=black] (x) circle (.35);
		\draw[black,fill=black] (y) circle (.35);
		\draw[black,fill=black] (z) circle (.35);
		\draw[opacity=0] (270:0.6) circle (0.1);
		\draw[opacity=0] (90:1.25) circle (0.1);
}}
\newcommand{\ptwovioletred}{\mathord{\scaleobj{1.2}{\scalerel*{\usebox{\ptwovioletredbox}}{x}}}}

\def\seplabel{($\rseparated$)$_{\arabic*}$}
\def\otoplabel{($\rovertop$)$_{\arabic*}$}
\def\crosslabel{($\rcrossing$)$_{\arabic*}$}
\def\ptwolabel{($\rptwo$)$_{\arabic*}$}
\def\lchlabel{($\rlcherry$)$_{\arabic*}$}

\def\bchlabel{$\genfrac{(}{)}{0pt}{2}{\raisebox{-1.5pt}[0pt]{$\textstyle\rlcherry$}}{\raisebox{-0.5pt}[0pt]{$\textstyle\rrcherry$}}_{\arabic*}$}
\def\og{\mathrm{og}}
\let\ol=\overline

\def\pt{\mathrm{pt}}
\def\HJ{\mathrm{HJ}}
\def\CPL{\mathrm{CPL}}
\newcommand{\conc}{\charfusion[\mathbin]{+}{\times}}
\let\sm=\smallsetminus
\let\wt=\widetilde
\def\Fbu{{F_\bullet}}
\def\sstar{{\star\star}}
\def\ccirc{{\circ\circ}}
\def\dip{\mathrm{dp}}

\usepackage{tikz-cd}
\tikzcdset{arrow style=tikz,
 diagrams={>={Stealth[round,length=4pt,width=5pt,inset=2.75pt]}}}
\usetikzlibrary{arrows.meta, positioning}

\makeatletter
\def\greek#1{\expandafter\@greek\csname c@#1\endcsname}
\def\Greek#1{\expandafter\@Greek\csname c@#1\endcsname}
\def\@greek#1{\ifcase#1
	\or $\alpha$\or $\beta$\or $\gamma$\or $\delta$\or
	$\epsilon$\or $\zeta$\or $\eta$\or
	$\theta$\or $\iota$\or $\kappa$\or $\lambda$	\or $\mu$\or
	$\nu$\or $\xi$\or $o$\or
	$\pi$\or $\rho$\or $\sigma$\or $\tau$\or $\upsilon$\or
	$\phi$\or $\chi$\or $\psi$\or
$\omega$\fi}
\def\@Greek#1{\ifcase#1
	\or $\mathrm{A}$\or $\mathrm{B}$\or $\Gamma$\or $\Delta$\or
	$\mathrm{E}$\or $\mathrm{Z}$\or
	$\mathrm{H}$\or $\Theta$\or $\mathrm{I}$\or $\mathrm{K}$\or
	$\Lambda$\or $\mathrm{M}$\or
	$\mathrm{N}$\or $\Xi$\or $\mathrm{O}$\or $\Pi$\or
	$\mathrm{P}$\or $\Sigma$\or
	$\mathrm{T}$\or $\mathrm{Y}$\or $\Phi$\or $\mathrm{X}$\or
$\Psi$\or $\Omega$\fi}

\AddEnumerateCounter{\greek}{\@greek}{24}
\AddEnumerateCounter{\Greek}{\@Greek}{12}

\makeatother

\makeatletter
\def\@setaddresses{\par
  \nobreak \begingroup
\footnotesize
  \def\author##1{\nobreak\addvspace\medskipamount}\def\\{\unskip, \ignorespaces}\interlinepenalty\@M
  \def\address##1##2{\begingroup
    \par\addvspace\medskipamount\indent
    \@ifnotempty{##1}{(\ignorespaces##1\unskip) }{\scshape\ignorespaces##2}\par\endgroup}\def\curraddr##1##2{\begingroup
    \@ifnotempty{##2}{\nobreak\indent\curraddrname
      \@ifnotempty{##1}{, \ignorespaces##1\unskip}\/:\space
      ##2\par}\endgroup}\def\email##1##2{\begingroup
    \@ifnotempty{##2}{\nobreak\indent\emailaddrname
      \@ifnotempty{##1}{, \ignorespaces##1\unskip}\/:\space
      \ttfamily##2\par}\endgroup}\def\urladdr##1##2{\begingroup
    \def~{\char`\~}\@ifnotempty{##2}{\nobreak\indent\urladdrname
      \@ifnotempty{##1}{, \ignorespaces##1\unskip}\/:\space
      \ttfamily##2\par}\endgroup}\addresses
  \endgroup
}
\makeatother

\begin{document}
\title[Clean canonical Ramsey theorem]{Canonical Ramsey
theorem for graphs with clean intersections}

\author[M. Az\'ocar]{Mat\'ias Az\'ocar Carvajal}
\address{Fachbereich Mathematik, Universit\"at Hamburg,
Hamburg, Germany}
\email{matias.azocar.carvajal@uni-hamburg.de}

\author[A. Basu]{Ayush Basu}
\address{Department of Mathematics, Emory University, Atlanta, GA, USA}
\email{ayush.basu@emory.edu}

\author[Chr. Reiher]{Christian Reiher}
\address{Fachbereich Mathematik, Universit\"at Hamburg,
Hamburg, Germany}
\email{christian.reiher@math.uni-hamburg.de}

\author[V. R\"odl]{Vojt\v ech R\"odl}
\address{Department of Mathematics, Emory University, Atlanta, GA, USA}
\email{vrodl@emory.edu}

\author[G. Santos]{Giovanne Santos}
\address{Departamento de Ingeniería Matemática, Universidad
de Chile, Santiago, Chile}
\email{gsantos@dim.uchile.cl}

\author[M. Schacht]{Mathias Schacht}
\address{Fachbereich Mathematik, Universit\"at Hamburg,
Hamburg, Germany}
\email{schacht@math.uni-hamburg.de}

\thanks{The first author was supported by ANID and DAAD under
	ANID-PFCHA/Doctorado Acuerdo Bilateral DAAD Becas Chile/2023-62230021.
	The fourth author was supported by NSF grant DMS 2300347.
	The fifth author was supported by ANID Becas/Doctorado
Nacional~21221049.}


\begin{abstract}
	Extending earlier results of
	Ne\v set\v ril and R\"odl
	[\emph{Selective graphs and hypergraphs}, Ann. Discrete
	Math. \textbf{3} (1978), 181--189],
	we show that for every ordered graph~$F$ there exist an
	ordered graph~$H$ and a system~$\ccH_F$ of induced copies of~$F$
	such that every colouring of the edges of~$H$ yields a
	canonically coloured copy of~$F$ from~$\ccH_F$
	and any two copies from~$\ccH_F$ intersect
	either in a vertex or an edge or not at all.

	As a consequence, this allows us to construct, for any
	given ordered graph~$F$, canonical Ramsey graphs~$H$
	enjoying additional structural properties.
	In particular, $H$ can have the same clique number as~$F$
	and, provided~$F$ is not bipartite, the same odd girth.
	Moreover,
	if~$F$ is connected, then 
	the copies of~$F$ from~$\ccH_F$ are not only induced, but
	their pairs of vertices
	also have the same distances in~$H$ as in~$F$.
\end{abstract}

\maketitle
\setcounter{footnote}{1}

\section{Introduction}
\label{sec:introduction}

\subsection{Structural Ramsey theory}
Following Erd\H os
and Rado~\cites{ER53,ER56} we use the shorthand notation
\begin{equation}\label{eq:Ramsey}
	H\lra(F)_r
\end{equation}
for graphs~$F$ and~$H$
and some number of colours~$r$ to signify the Ramsey property
that every~$r$-colouring of $E(H)$ yields a monochromatic and
\emph{induced} copy of~$F$ in~$H$. Without requiring the monochromatic copy
of~$F$ to be induced, the existence of~$H$ would be a consequence of
Ramsey's theorem~\cite{R30} for (finite) graphs.
Henson~\cite{He73} asked whether the induced version holds as well.
This was answered affirmatively independently by Deuber~\cite{Deu75},
by Erd\H{o}s, Hajnal, and P\'osa~\cite{EHP75},
and by R\"odl~\cites{R73,R76}. In particular, those authors established the
\emph{induced Ramsey theorem} asserting the existence of
a graph~$H$ satisfying~\eqref{eq:Ramsey}
for any given~$F$ and~$r$.

Structural Ramsey theory aims at
a better understanding of unavoidable properties
that such a Ramsey graph~$H$ must display.
A first general result in this direction is due to
Folkman~\cite{F70}, who addressed
a question of Erd\H{o}s and Hajnal~\cite{EH67} and showed that
for two colours there exist Ramsey graphs~$H$ for the
clique~$K_k$ with the minimal possible clique number,
i.e., $\omega(H)=k$.

Around the same time Galvin (see~\cite{EH70}) formulated the
more general
question of the extent to which restricted classes of graphs
are closed under the Ramsey
property. He specifically asked this question
for the class of triangle-free graphs. This problem was addressed
by Ne\v{s}et\v{r}il and R\"odl~\cite{NR75} by showing
that for every triangle-free graph~$F$ and every number of colours~$r$
there exists a triangle-free graph~$H$ satisfying the Ramsey
property~\eqref{eq:Ramsey}.
This result was further extended in two directions: firstly, for
graphs with given clique number~\cite{NR-ColorFolkman},
which generalises Folkman's theorem,
and secondly, for graphs without
odd cycles up to a given length~\cite{NR-OddGirth}.

Since then quite a few results in structural Ramsey theory have emerged
and we point the interested reader to the surveys of
Ne\v{s}et\v{r}il~\cite{N95} and of
Hubi\v{c}ka and Kone\v{c}n\'{y}~\cite{HK26} and the
references therein. Many of the proofs in the area
are based on the \emph{partite construction method}
introduced by Ne\v{s}et\v{r}il and R\"odl~\cite{NR-SimpleProofPartCons}.
In more advanced applications of that method a careful analysis of
the distribution of the copies of~$F$ in the constructed
Ramsey graph~$H$ is essential. Putting trivial cases aside
(e.g., graphs~$F$ with at most one edge),
it is easy to see that the copies of~$F$ in any Ramsey graph~$H$ cannot
be edge disjoint. Erd\H os, Ne\v{s}et\v{r}il, and R\"odl (see~\cite{Er75}) asked
whether all other pairwise intersections can be avoided in carefully
constructed Ramsey graphs for~$F$ being a clique; this was verified
in~\cite{NR-SimpleProofPartCons}. For general graphs~$F$ this leads to the
following concept of clean intersections, which
appeared in~\cite{NR84}*{Proposition~2.3}.

We say two graphs~$F$ and $F'$ have a \emph{clean
intersection} (or they \emph{intersect cleanly}) if
\[
	|V(F)\cap V(F')|\leq 1
	\qquad\text{or}\qquad
	V(F)\cap V(F')\in E(F)\cap E(F')\,.
\]
Moreover, a family of graphs has clean
intersections if any two distinct members
intersect cleanly.
Note that for disconnected~$F$ no Ramsey graph~$H$ can have
the property that all copies of~$F$ in~$H$ pairwise intersect
cleanly, since copies of the components of~$F$ may combine
into copies of~$F$ violating this property. For that reason
we consider \emph{Ramsey systems}, i.e., distinguished
families of copies~$\ccH_F$ of~$F$ in~$H$ witnessing the Ramsey
property. Ramsey systems with clean intersections
are not only interesting in their own right, as they exhibit
the minimal possible pairwise intersections---they also turned out to be useful
for obtaining further results in structural Ramsey theory.

In this work we employ the partite construction method
to obtain systems of copies~$\ccH_F$ of any given graph~$F$
with clean intersections that possess the \emph{canonical}
Ramsey property.
These constructions of
systems of copies with clean intersections will allow us
to establish analogues of the structural Ramsey results
discussed above in the canonical setting (see
Theorems~\ref{thm:main-simple} and~\ref{thm:main} below).

\subsection{Canonical colourings}
\label{sec:canonical}
The canonical Ramsey theorem of Erd\H os and Rado~\cite{ER50}
generalises Ramsey's theorem to an unbounded number of
colours. Obviously,
with this relaxation we can no longer ensure the appearance of
monochromatic copies of the target graph~$F$. However,
Erd\H os and Rado characterised all canonical colour patterns
that are unavoidable in such colourings of edge sets.
As it turns out, the ordering of the underlying vertex set
plays a r\^ole in that context and, hence, we
consider \emph{ordered graphs} with
totally ordered vertex sets.
We recall that a colouring $\phi\colon E(F)\to \NN$ of an
ordered graph~$F$ is \emph{canonical}
if one of the following four patterns arises:
\begin{enumerate}[label=\rmlabel]
	\item $\phi$ is \emph{monochromatic}, i.e., all edges of
		$F$ have the same colour;
	\item\label{it:min-colouring} $\phi$ is
		\emph{min-coloured}, i.e., for~$e$, $e'\in E(F)$ we have
		$\phi(e)=\phi(e')$ $\Longleftrightarrow$
		$\min e=\min e'$;
	\item $\phi$ is \emph{max-coloured}, i.e., for~$e$, $e'\in
		E(F)$ we have $\phi(e)=\phi(e')$ $\Longleftrightarrow$
		$\max e=\max e'$;
	\item $\phi$ is \emph{rainbow}, i.e., all edges of~$F$ have
		distinct colours.
\end{enumerate}
For graphs~$F$ and~$H$ we denote by $\binom{H}{F}$ the set of all
\emph{induced} copies of~$F$ in~$H$. If~$H$ and~$F$ are
ordered, then
we naturally restrict $\binom{H}{F}$ to the ordered induced copies
of~$F$ in~$H$.

For ordered graphs~$F$ and~$H$ we write
\[
	H\clra(F)
\]
to signify that~$H$ has the \emph{canonical Ramsey property} for~$F$,
i.e., for every colouring $\phi\colon E(H)\to\NN$ there exists a copy
$F_\star\in\binom{H}{F}$ such that~$\phi$ restricted to
$E(F_\star)$ is canonical.
Moreover, we write
\[
	\ccH_F\clra(F)
\]
for a system of copies $\ccH_F\subseteq \binom{H}{F}$ to
signify that the guaranteed
canonically coloured copy is always an element of~$\ccH_F$.

An induced version of the canonical Ramsey theorem with
preserved clique number
was obtained by Ne\v set\v ril and R\"odl~\cite{NR-78}.
We strengthen this result by showing
that we can further restrict to
a system of copies $\ccH_F\subseteq \binom{H}{F}$ with clean
intersections.
When~$F$ is a clique, such a result was recently obtained
by Kam\v cev and Schacht~\cite{KS25} by considering
(appropriate subgraphs of) sparse random graphs.
The proof presented here is constructive and in addition we
also maintain the same odd girth.
We recall that~$\og(F)$, the \emph{odd girth} of a graph~$F$,
is defined as the length of a
shortest odd cycle in~$F$ if such a cycle exists and we set
$\og(F)=\infty$ for bipartite graphs~$F$.
With this notation at hand our main result can be stated as follows.

\begin{thm}\label{thm:main-simple}
	For every ordered
	graph~$F$ there exist an ordered graph~$H$ and a system of
	copies $\ccH_F\subseteq \binom{H}{F}$
	with clean intersections such that $\ccH_F\clra(F)$.
	Moreover, the graph~$H$ has the same clique number as~$F$
	and, in case~$F$ is not bipartite, it also has the same
	odd girth. Finally, if \(F\) is connected, then any two
	vertices of any copy \(F_\star\in\ccH_F\) have the same
	distance in \(H\) as in \(F_\star\).
\end{thm}

We remark that in general for ordered bipartite graphs~$F$
we cannot insist that~$H$ is bipartite as well.
For example, the path~$P$ of length three with vertices
$u<v<w<w'$ and edges $uv$, $vw$, $ww'$ is clearly bipartite, but
for any given ordered bipartite graph~$H$ with vertex
partition $X\dcup Y$ we may colour the edges $xy$ with $x\in
X$ and $y\in Y$
by the colour~$x$. Since every
copy of~$P$ in~$H$ has two vertices in~$X$, this colouring
induces exactly two colours on any copy of~$P$, while every
canonical colouring of~$P$ is either monochromatic or uses exactly
three colours.
However, our proof shows that we can require the odd girth
of~$H$ to be arbitrarily large, if~$F$ is bipartite (see
Theorem~\ref{thm:main} below).

Our proof of Theorem~\ref{thm:main-simple} falls short of maintaining
the girth itself.
Such a result likely requires more delicate control
over the systems of copies arising throughout the construction.
For example, such copies would need to expand and not create
short cycles themselves. In the non-canonical context, this level
of control was recently established by
Reiher and R\"odl~\cite{RR-girth} and we believe that similar
results also hold for the canonical Ramsey property.

Furthermore, Theorem~\ref{thm:main-simple} provides a
system $\ccH_F$ of copies of~$F$ and leaves open when
this system can be taken to consist of all copies, that is,
when $\ccH_F=\binom{H}{F}$ can be achieved.
For cliques such a result follows from the
probabilistic proof of Kam\v cev and Schacht~\cite{KS25} and is
also delivered by the constructive proof presented here.
In general, it is plausible that such a strengthening of
Theorem~\ref{thm:main-simple} holds for most sufficiently
connected graphs~$F$.

\section{Overview of the proof}\label{sec:summary}
In this section we decompose the proof of our main result
into two steps---a preliminary
construction and a statement from structural Ramsey theory,
where instead of edges we
colour subgraphs of a fixed isomorphism type.

For a fixed graph~$F$, by an \emph{$F$-system} we mean a pair
$(B, \ccB_F)$ consisting of
a graph~$B$ and a set~$\ccB_F$ of induced copies of~$F$ in
$B$. Similarly, for an ordered
graph~$F$ an ordered~$F$-system $(B, \ccB_F)$ has an ordered
underlying graph~$B$ and
the copies in~$\ccB_F$ need to respect the vertex ordering.

In the first step of our proof we construct for every ordered
graph~$F$
an auxiliary ordered~$F$-system $(B, \ccB_F)$ with several
desirable properties. We begin with a property of edge
colourings, which connects
the two parts of our argument.

\begin{dfn}
	Given an ordered graph~$F$ and an ordered~$F$-system $(B, \ccB_F)$,
	an edge colouring~${\phi\colon E(B)\lra\NN}$ is
	\emph{$\ccB_F$-homogeneous}
	if for any two copies $F_1, F_2\in \ccB_F$ the strictly increasing
	map $\eta\colon V(F_1)\lra V(F_2)$ (which is an
	isomorphism between~$F_1$ and~$F_2$)
	preserves the relation of having the same colour, i.e.,
	for all edges $e, f\in E(F_1)$ the following holds:
	\[
		\phi(e)=\phi(f) \quad \Longleftrightarrow \quad
		(\phi\circ\eta)(e)=(\phi\circ\eta)(f)\,.
	\]
\end{dfn}

So, intuitively speaking, $\phi$ is~$\ccB_F$-homogeneous if
$\phi$ induces the same
colour pattern on any two copies of~$F$ in~$\ccB_F$.

In order to preserve the odd girth in our main result, we
sometimes need to work with
copies that are not only induced but also inherit bounded
distances from their
host graph. This is made precise by the following concepts.

\begin{dfn}
	Let~$F$ be a subgraph of another graph~$B$.
	\begin{enumerate}[label=\rmlabel]
		\item We call~$F$ a \emph{distance preserving} subgraph
			of~$B$ if any two
			vertices of~$F$ have the same distance in~$B$ as
			in~$F$; the set of
			distance preserving copies of~$F$ in~$B$ is denoted
			by~$\binom BF_\dip$.

		\item For a natural number~$n$, we say that~$F$ is
			\emph{$n$-induced} in~$B$
			provided that the following holds: If for two distinct
			vertices $x, y\in V(F)$
			there is an~$x$-$y$-path~$P$ in~$B$ which is not
			contained in~$F$ and
			whose length is at most~$n$, then the distance of~$x$
			and~$y$ in~$F$
			is strictly smaller than the length of~$P$. Moreover,
			we write $\binom BF_n$
			for the set of~$n$-induced copies of~$F$ in~$B$.
	\end{enumerate}
\end{dfn}

Clearly, distance preserving copies must be induced and
we have the monotone inclusions
\[
	\binom BF
	=\binom BF_1
	\supseteq \dots
	\supseteq \binom BF_n
	\supseteq \binom BF_{n+1}
	\supseteq \cdots \,.
\]
Moreover,
\[
	\bigcap_{n\in\NN}\binom BF_n\subseteq\binom BF_\dip\subseteq \binom BF
\]
and, if~$F$ is connected, then we also have
\begin{equation}\label{eq:dipincl}
	\binom BF_{\diam(F)}\subseteq\binom BF_\dip\,.
\end{equation}

The following definition encapsulates the desired structural
properties for the
ordered~$F$-system $(B, \ccB_F)$ we intend to construct.

\begin{dfn}\label{def:conf}
	For an ordered graph~$F$ and an integer $n\geq 1$,
	we say that an ordered~$F$-system $(B, \ccB_F)$ is
	\emph{$n$-conform} if
	\begin{enumerate}[label=\rmlabel]
		\item\label{def:conf-clean} the copies in~$\ccB_F$ have
			clean intersections,
		\item\label{def:conf-omega} $\omega(B)=\omega(F)$,
		\item\label{def:conf-oddg} $\og(B)\geq\min\{\og(F),2n+1\}$,
		\item\label{def:conf-induced} and $\ccB_F\subseteq \binom{B}{F}_n$.
	\end{enumerate}
\end{dfn}
For preserving the odd girth one may expect that
part~\ref{def:conf-oddg} in the definition above should simply
be $\og(B)=\og(F)$. However, as discussed before, such a
strong version of
Theorem~\ref{thm:main-simple} is false for bipartite graphs,
which have infinite odd girth.
In this case we can still guarantee canonical Ramsey graphs
of arbitrarily large odd girth and
this leads to the formulation chosen in
part~\ref{def:conf-oddg} of the definition.
The strong version of Theorem~\ref{thm:main-simple} can now
be stated as follows.
\begin{thm}\label{thm:main}
	For every ordered
	graph~$F$ and every integer $n\geq 1$
	there exists an~$n$-conform~$F$-system $(H,\ccH_F)$ such
	that $\ccH_F\clra(F)$.
\end{thm}
In view of the inclusion~\eqref{eq:dipincl}, Theorem~\ref{thm:main}
applied for connected graphs~$F$ and with $n=\diam(F)$ yields
a system~$\ccH_F$ of distance preserving copies of~$F$.
In the context of Ramsey's theorem
such a structural result was obtained by Dellamonica and
R\"odl~\cite{DR12}.

Clearly, Theorem~\ref{thm:main} implies Theorem~\ref{thm:main-simple}:
it suffices to apply it with
\[
	n=\max\bigl\{\diam(F),\tfrac{\og(F)-1}{2},1\bigr\}\,,
\]
where the
terms $\diam(F)$ and $\tfrac{\og(F)-1}{2}$ are only taken into
account if~$F$ is connected or not bipartite, respectively.
Moreover, $\ccH_F\clra(F)$ entails $\ccH_F\ne\vn$, so~$H$
contains an induced copy of~$F$ and, hence, also
$\og(H)\le\og(F)$.
The rest of this work is devoted to the proof of
Theorem~\ref{thm:main}.
For that we first establish
the following result in \ssign\ref{sec:conform}.

\begin{prop}\label{prop:hom2canon}
	For every ordered graph~$F$ and every integer $n\geq 1$ there exists
	an ordered~$n$-conform~$F$-system $(B, \ccB_F)$ with \(\ccB_F\ne\vn\)
	such that for every~$\ccB_F$-homogeneous colouring
	\mbox{$\phi\colon E(B)\to\NN$} all copies in~$\ccB_F$ are canonically
	coloured.
\end{prop}

Roughly speaking, our intended way of using this result is
that if for any ordered
graph~$H$ we have an arbitrary edge colouring $\phi\colon
E(H)\lra \NN$, then each copy
in $\binom HF$ receives one of boundedly many colour
patterns. Indeed, the number of
possible colour patterns agrees with the number of
equivalence relations on $E(F)$, which
is a function of $e(F)$.
Thus when~$H$ is constructed in such a way that
$\binom{H}{B}$ is ``sufficiently rich'', then there exists
a copy $B^\star\in\binom HB$
such that the corresponding~$F$-system $(B^\star,
\ccB^\star_F)$ isomorphic to~$(B, \ccB_F)$
receives a $\ccB^\star_F$-homogeneous colouring, meaning that
its copies of~$F$
are canonically coloured.

We shall now introduce some language that will allow us to
present this plan in a more
precise form. An~$F$-system $(B, \ccB_F)$ is said to be an
\emph{induced~$F$-subsystem} of
another~$F$-system $(H, \ccH_F)$ if~$B$ is an induced
subgraph of~$H$ and,
furthermore,
\[
	\ccB_F=\ccH_F\cap \binom BF\,.
\]
We write $\binom{(H, \ccH_F)}{(B, \ccB_F)}$ for the set of
all induced~$F$-subsystems
of $(H, \ccH_F)$ which are, in the obvious sense, isomorphic
to $(B, \ccB_F)$.
By a \emph{$(B, \ccB_F)$-conglomerate} we mean a triple $(H,
\ccH_B, \ccH_F)$ such
that $(H, \ccH_F)$ is an~$F$-system and $\ccH_B\subseteq
\binom{(H, \ccH_F)}{(B, \ccB_F)}$. Given such a conglomerate
and a number of colours~$r$
the partition relation
\[\ccH_B\lra(B, \ccB_F)^F_r
\]indicates that for every colouring $\phi\colon \ccH_F\lra
[r]$ there exists a copy
$(B^\star, \ccB_F^\star)\in \ccH_B$ such that $\ccB_F^\star$
is monochromatic with
respect to~$\phi$.

These concepts extend to ordered graphs and systems in the expected way.
Most importantly, if~$F$ is an ordered graph and $(B,
\ccB_F)$, $(H, \ccH_F)$
are ordered~$F$-systems, then~$\binom{(H, \ccH_F)}{(B, \ccB_F)}$
only contains copies of $(B, \ccB_F)$ that are ordered
correctly; similarly,
for an {ordered $(B, \ccB_F)$-conglomerate} $(H, \ccH_B,
\ccH_F)$ the copies
in~$\ccH_F$ and~$\ccH_B$ need to respect the vertex orderings
of~$F$, $B$, and~$H$.

\begin{prop}\label{prop:final}
	For every ordered graph~$F$, every $n\geq 1$,
	every ordered $n$-conform $F$-system $(B, \ccB_F)$, and
	every number of colours~$r$,
	there exists an ordered $(B, \ccB_F)$-conglomerate
	$(H, \ccH_B, \ccH_F)$ such that
	\begin{enumerate}[label=\alabel]
		\item\label{it:final-a} $\ccH_B\lra (B, \ccB_F)^F_r$
		\item\label{it:final-b} and $(H,\ccH_F)$ is an~$n$-conform~$F$-system.
	\end{enumerate}
\end{prop}
The proof of this proposition employs the partite construction
method and is deferred to
\S\ssign\ref{sec:Hales}\,--\,\ref{sec:conform-cong}.
We conclude this section with the proof of Theorem~\ref{thm:main}
assuming Propositions~\ref{prop:hom2canon} and~\ref{prop:final}.

\begin{proof}[Proof of Theorem~\ref{thm:main}]
	Fix an ordered graph~$F$ and an integer $n\geq 1$.
	By Proposition~\ref{prop:hom2canon}, there exists an~$n$-conform
	ordered~$F$-system $(B, \ccB_F)$ with $\ccB_F\ne\vn$ such that every
	$\ccB_F$-homogeneous colouring of $E(B)$ colours
	all copies of~$F$ in~$\ccB_F$ canonically.

	Let~$r$ be the number of equivalence relations on~$E(F)$,
	which equals the number of colour patterns that any edge colouring
	can induce on a copy of~$F$.
	We apply Proposition~\ref{prop:final} with~$(B, \ccB_F)$ and~$r$ to
	obtain an ordered $(B, \ccB_F)$-conglomerate $(H, \ccH_B, \ccH_F)$
	satisfying properties~\ref{it:final-a} and~\ref{it:final-b}
	of the proposition.
	In particular, $(H, \ccH_F)$ is an~$n$-conform~$F$-system
	by assertion~\ref{it:final-b}. We now argue that
	$\ccH_F\clra(F)$ holds.

	Given any edge colouring $\phi\colon E(H)\to\NN$, each copy
	$F_\star\in\ccH_F$ inherits from~$\phi$ an equivalence relation
	on~$E(F_\star)$, determining one of the~$r$ colour patterns.
	This defines a colouring $\psi\colon \ccH_F\to[r]$ of the
	copies of~$F$.
	By assertion~\ref{it:final-a} of Proposition~\ref{prop:final},
	there exists a copy $(B^\star,\ccB_F^\star)\in\ccH_B$
	such that $\ccB_F^\star$ is monochromatic
	under~$\psi$. Consequently, $\phi$ induces the same colour pattern on
	every $F_\star\in\ccB_F^\star$, i.e., $\phi$ is
	$\ccB_F^\star$-homogeneous.
	Since $(B^\star, \ccB_F^\star)$ is isomorphic to $(B, \ccB_F)$,
	Proposition~\ref{prop:hom2canon} implies that every copy of~$F$
	in~$\ccB_F^\star\subseteq\ccH_F$ is canonically coloured by~$\phi$.
	In particular, as $\ccB_F\ne\vn$, the set $\ccB_F^\star$ is
	non-empty and, hence,~$\ccH_F$ indeed contains a canonically
	coloured copy of~$F$, i.e., $\ccH_F\clra(F)$.
\end{proof}

\section{Conform systems being canonical on homogeneous
colourings}\label{sec:conform}
This section is devoted to a constructive proof of
Proposition~\ref{prop:hom2canon}.
For that we first study in \ssign\ref{sec:simpleext} simple
operations on systems of copies that preserve conformity.
In \ssign\ref{sec:Scycle} and~\ssign\ref{sec:pf-prophom2canon} we
utilise these operations
to obtain the desired conform~$F$-system $(B,\ccB_F)$, for
which homogeneous colourings induce
canonical copies of~$F$.

\subsection{Extensions of conform systems}\label{sec:simpleext}
We study simple operations on conform systems that keep
conformity intact.
For example, it is easy to check that removing a copy $F'$
of~$F$ from an
$n$-conform~$F$-system (consisting of at least two copies), where
we only remove vertices and edges that are only contained in
$F'$ and no other copy of~$F$,
results in an~$n$-conform system. Below we analyse the situation when
an~$n$-conform system is enlarged.

The first observation shows
that the union of conform~$F$-systems intersecting in exactly
one copy of~$F$ is again conform.

\begin{lemma}\label{lem:union-conform}
	Suppose $(B,\ccB_F)$ and $(B',\ccB'_F)$ are two~$n$-conform
	$F$-systems
	satisfying $B\cap B'\in\ccB_F\cap\ccB'_F$.
	Then the union
$(B\cup B',\ccB_F\cup \ccB'_F)$
is again an~$n$-conform~$F$-system.
\end{lemma}
\begin{proof}
	Obviously, $(B\cup B',\ccB_F\cup \ccB'_F)$ is still an
	$F$-system  with
	\[
		\omega(B\cup B')=\max\big\{\omega(B),\omega(B')\big\}=\omega(F)
	\]
	and the copies of~$F$ in $\ccB_F\cup \ccB'_F$ have clean
	intersections.

	Concerning the odd girth, suppose~$C$ is a shortest odd
	cycle in $B\cup B'$. We may assume that~$C$ is neither
	contained in~$B$ nor in~$B'$. Let $v_1,\dots,v_k$ be the
	vertices of the unique shared copy
	$\Fs\in\ccB_F\cap\ccB'_F$ appearing on the cycle
	$C$ in cyclic order. For $i\in\ZZ/k\ZZ$
	let~$P_i$ be the corresponding~$v_i$-$v_{i+1}$-subpath in
	$C$. By our choice of~$C$ there are $i\neq j$
	such that~$P_i$ contains a vertex from $V(B)\sm
	V(B')$ and~$P_j$ contains a vertex from $V(B')\sm V(B)$.
	Let~$\l_i$ and~$\l_j$ denote the lengths of these two paths
	and suppose $\l_i\leq \l_j$.

	If $\l_i\geq n$, then the length of~$C$ is at least $2n+1$ and
	we are done. Assuming $\l_i<n$ allows us to appeal
	to the~$n$-inducedness of $(B,\ccB_F)$, which yields a
	shorter~$v_i$-$v_{i+1}$-path~$Q_i$ within~$\Fs$.
	If~$Q_i$ and~$P_i$ have the same parity, then replacing
	$P_i$ by~$Q_i$ in~$C$ leads to
	a shorter closed odd walk in $B\cup B'$ contradicting the
	minimal choice of~$C$. Otherwise
	the union of~$P_i$ and~$Q_i$ forms a closed odd walk~$W$ with
	\[
		\text{length of $W$}
		=\text{length of $Q_i$}+\text{length of $P_i$}
		<2\l_i
		\leq \l_i+\l_j
		\leq\text{length of $C$}\,,
	\]
	which leads to the same contradiction.

	For the~$n$-inducedness we consider an~$x$-$y$-path~$P$ of
	length~$m\leq n$
	in $B\cup B'$ connecting two vertices in a copy $F_0\in
	\ccB_F$, but such that \(P\not\subseteq F_0\). If~$P$ is
	contained in~$B$, then the~$n$-inducedness of $(B,\ccB_F)$
	implies that the
	distance in~$F_0$ is less than~$m$. Otherwise, any maximal
	segment $aPb$ with all inner vertices from
	$V(B')\sm V(B)$ has the property that~$a$ and~$b$ are
	both contained in
	the unique shared copy $\Fs\in\ccB_F\cap\ccB'_F$.
	Consequently, the~$n$-inducedness of $(B',\ccB'_F)$ asserts
	that the segment $aPb$ can be replaced by a shorter path
	contained in $\Fs\subseteq B$.
\end{proof}

Other simple extension operations concern adding a single
copy of~$F$ to a conform system with restricted
intersections. For example, it is easy to see that adding a
single copy of~$F$,
which only intersects the given system in a single vertex or
a single edge, has no effect on the conformity, which we
summarise in the following fact.
\begin{fact}\label{fact:simple-ext}
	Suppose $(B,\ccB_F)$ is an~$n$-conform~$F$-system and $F'$
	is a copy of~$F$ such that
	\[
		\big|V(F')\cap V(B)\big|\leq 1
		\qquad\text{or}\qquad
		V(F')\cap V(B)\in E(F')\cap E(B)\,.
	\]
	Then the extension $(B\cup F',\ccB_F\cup \{F'\})$
	is also an~$n$-conform~$F$-system.\qed
\end{fact}

The following lemmata analyse the situation when the added
copy intersects a
conform system in three or four vertices. For that we recall
the following standard graph-theoretic concepts.
For a graph $B=(V,E)$ and~$x$, $y\in V$
we denote by $\dist_B(x,y)$ the length of a shortest
$x$-$y$-path in~$B$ if such a path exists and~$\infty$ otherwise.
More generally, for sets of vertices
$X$, $Y\subseteq V$ we write $\dist_B(X,Y)$ for the minimum
of $\dist_B(x,y)$ over all
$x\in X$ and $y\in Y$. In particular, $\dist_B(X,Y)=\infty$
if one of the sets~$X$ or~$Y$ is empty.
For vertices $v\in V$, edges $xy\in E$, and subgraphs $\Fs\subseteq B$
we also simply write $\dist_B(v,\Fs)$ and $\dist_B(xy,\Fs)$
instead of $\dist_B(\{v\},V(\Fs))$ and
$\dist_B(\{x,y\},V(\Fs))$, respectively.

The first lemma addresses the case when the added copy of~$F$
intersects the given conform system
in a vertex and an edge.

\begin{lemma}\label{lem:extend-family-ve}
	Suppose $(B,\ccB_F)$ is an~$n$-conform~$F$-system
	with an edge $ab\in E(B)$ and a vertex $y\in
	V(B)\sm\{a,b\}$ satisfying $B=\bigcup\ccB_F$ and
	\begin{equation}\label{eq:distance-ve}
		\dist_B(ab,\Fs)+\dist_B(y,\Fs)
		\geq
		\max\big\{3,\,\min\{\og(F)-1,2n\},\,n+1\big\}
	\end{equation}
	for every $\Fs\in\ccB_F$.

	Then for every copy $F'$ of~$F$ satisfying $V(F')\cap
	V(B)=\{a,b,y\}$ and
	$ab\in E(F')$ the extension $(B\cup F',\ccB_F\cup \{F'\})$
	is an~$n$-conform~$F$-system.
\end{lemma}

The next lemma deals with a very similar case, when the
single vertex~$y$ is replaced by
an edge $xy$ disjoint from $ab$. Below we shall only
prove that version of the lemma, since the same reasoning
yields Lemma~\ref{lem:extend-family-ve}.

\begin{lemma}\label{lem:extend-family-2m}
	Suppose $(B,\ccB_F)$ is an~$n$-conform~$F$-system
	with disjoint edges $ab$, $xy\in E(B)$ satisfying
	$B=\bigcup\ccB_F$ and
	\[
		\dist_B(ab,\Fs) + \dist_B(xy,\Fs)
		\geq
		\max\big\{3,\,\min\{\og(F)-1,2n\},\,n+1\big\}
	\]
	for every $\Fs\in\ccB_F$.

	Then for every copy $F'$ of~$F$ satisfying $V(F')\cap
	V(B)=\{a,b,x,y\}$ and
	$ab$, $xy\in E(F')$ the extension $(B\cup F',\ccB_F\cup
	\{F'\})$ is an~$n$-conform~$F$-system.
\end{lemma}
\begin{proof} Property~\ref{def:conf-clean} of
	Definition~\ref{def:conf} concerns clean intersections and for
	the extended system~$\ccB_F\cup\{F'\}$ this is a direct
	consequence of the distance assumption, since it prevents
	any copy $\Fs\in\ccB_F$ from containing a vertex from $\{a,b\}$
	and $\{x,y\}$ at the same time. In particular,
	this also shows that all copies of~$F$ in the extended
	system $\ccB_F\cup \{F'\}$ are still
	induced in $B\cup F'$, regardless of whether $F'$ contains an edge
	with one vertex in $\{a,b\}$ and one in $\{x,y\}$.

	Next we address property~\ref{def:conf-omega} of
	Definition~\ref{def:conf}. For that we
	suppose for a contradiction that there is a clique
	$K_m\subseteq B\cup F'$ with $m>\omega(B)=\omega(F)$.

	If $V(K_m)\subseteq V(B)$,
	then~$K_m$ contains an edge $uv\in E(F')$ with
	$u\in\{a,b\}$ and $v\in\{x,y\}$ and another vertex
	$w\in V(B)\sm\{a,b,x,y\}$. Since $B=\bigcup\ccB_F$
	there is some copy $F_w\in\ccB_F$ containing~$w$ and owing
	to the distance assumption we have
	\[
		\dist_B(u,w)+\dist_B(v,w)
		\geq \dist_B(ab,F_w)+\dist_B(xy,F_w)
		\geq 3\,.
	\]
	Hence, one of the edges $uw$ or $vw$ must be missing in~$B$, as
	well as in $B\cup F'$, ruling out this case.

	In the other case~$K_m$ must contain vertices $w\in
	V(B)\sm V(F')$ and
	$w'\in V(F')\sm V(B)$, which contradicts the fact
	that $ww'\not\in E(B\cup F')$.

	For property~\ref{def:conf-oddg} of Definition~\ref{def:conf}
	we consider a shortest odd cycle~$C$ in $B\cup F'$ not
	contained in~$B$ and not contained in $F'$.
	Such a cycle~$C$ must contain at least two distinct
	vertices from the set $\{a,b,x,y\}$. In fact, $C$ contains
	a~$u$-$v$-path~$P$ with~$u$, $v\in\{a,b,x,y\}$ and all
	internal vertices from
	$V(B)\sm V(F')$. If $uv$ equals the edge $ab$, then
	either $C=P+ab$ contradicting
	that~$C$ is not contained in~$B$ or~$ab$ is a chord in~$C$.
	However, such a chord~$ab$ either contradicts the minimal
	choice of~$C$ or it reveals that the length of~$C$ is at
	least $\min\{\og(B),\og(F)\}$ and we are done.

	Consequently, it remains to address the case $u\in\{a,b\}$
	and $v\in\{x,y\}$. In this case~$P$ contains some internal
	vertex~$w$ contained in some copy $F_w\in\ccB_F$ and we arrive at
	\begin{align*}
		|E(C)|
		>
		|E(P)|
		&\geq
		\dist_B(u,w)+\dist_B(v,w)\\
		&\geq
		\dist_B(ab,F_w)+\dist_B(xy,F_w)
		\geq
		\min\{\og(F)-1,2n\}\,,
	\end{align*}
	which concludes the verification of
	property~\ref{def:conf-oddg} of Definition~\ref{def:conf}.

	Finally we turn our attention to the inducedness properties
	of the extended system.
	For that we first consider a copy $\Fs\in\ccB_F$ and two
	vertices~$u$, $v\in V(\Fs)$ for which there
	exists a shortest~$u$-$v$-path~$P$ in $B\cup F'$ of length
	at most~$n$. Moreover, we can assume
	that~$P$ contains some edge from $E(F')\sm E(B)$. In
	particular, $P$ must contain two vertices
	from $\{a,b,x,y\}$ and appealing to the minimality of~$P$
	we infer that~$P$ contains a vertex from~$\{a,b\}$
	and a vertex from $\{x,y\}$, since otherwise we may use one
	of the edges $ab$ or $xy$ to locate a shorter path~$Q$,
	which either contradicts the choice of~$P$ or it is already
	contained in~$B$ allowing us to appeal to the assumption
	that $(B,\ccB_F)$ is~$n$-induced.
	However, with~$P$ containing a vertex from $\{a,b\}$ and
	from $\{x,y\}$, we can invoke the distance assumption
	\[
		|E(P)|
		> \dist_B(ab,\Fs)+\dist_B(xy,\Fs)
		\ge n+1\,,
	\]
	which contradicts the choice of~$P$.

	It remains to consider the case when~$u$, $v$ are vertices
	of the added copy $F'$. However, in this case the reasoning
	above applies again: a shortest~$u$-$v$-path of length at
	most~$n$ not contained in~$F'$ either admits a shortcut
	through one of the edges $ab$, $xy\in E(F')$, or it meets
	both $\{a,b\}$ and $\{x,y\}$ and has length at least~$n+1$.
\end{proof}

The last lemma addresses the situation
when the extended copy intersects the given system in two
edges sharing a vertex~$y$.
For the statement we consider distances in the given graph
with~$y$ removed. For a graph $B=(V,E)$,
distinct vertices~$x$, $y\in V$, and a subgraph $F\subseteq B$ we set
\[
	\dist_{B-y}(x,F)=\dist_{B'}(x,V(F)\sm\{y\})
\]
for the induced subgraph $B'$ of~$B$ induced on
$V\sm\{y\}$. Note that we make no assumption
on whether the considered subgraph $F\subseteq B$ contains
the vertex~$y$ or not.

\begin{lemma}\label{lem:extend-family}
	Suppose $(B,\ccB_F)$ is an~$n$-conform~$F$-system
	with edges $xy$, $yz\in E(B)$ satisfying $B=\bigcup\ccB_F$ and
	\begin{equation}\label{eq:distance-P2}
		\dist_{B-y}(x,\Fs) + \dist_{B-y}(z,\Fs)
		\geq
		\max\{3,\,n+1\}
	\end{equation}
	for every $\Fs\in\ccB_F$.

	Then for every copy $F'$ of~$F$ satisfying $V(F')\cap
	V(B)=\{x,y,z\}$ and
	$xy$, $yz\in E(F')$ the extension $(B\cup F',\ccB_F\cup
	\{F'\})$ is an~$n$-conform~$F$-system.
\end{lemma}
The proof of Lemma~\ref{lem:extend-family} parallels the
proof of Lemma~\ref{lem:extend-family-2m}
and the main difference concerns proving
property~\ref{def:conf-oddg} of Definition~\ref{def:conf}.
\begin{proof}
	Again the distance assumption ensures that no copy of~$F$
	in~$\ccB_F$ can contain both vertices~$x$ and~$z$.
	As a result all copies of~$F$ in the extended system
	$\ccB_F\cup \{F'\}$ are induced in $B\cup F'$.
	Moreover, since no copy of~$F$ in~$\ccB_F$ contains~$x$ and
	$z$, the extended system still enjoys clean
	intersections.

	For property~\ref{def:conf-omega} of
	Definition~\ref{def:conf} we again have only to consider cliques~$K_m$
	contained in $V(B)$ that appeared by adding the edge $xz$
	to~$B$. However, the distance assumption
	combined with $B=\bigcup\ccB_F$ ensures that at least one
	of the edges $wx$ or $wz$ must be missing in~$B$
	for every $w\in V(B)\sm \{x,y,z\}$.

	Regarding the odd girth in form of
	property~\ref{def:conf-oddg} of Definition~\ref{def:conf},
	we may assume
	that $xz\not\in E(F')$, since otherwise $\og(F)=3$ and
	property~\ref{def:conf-oddg} is trivial.
	We consider a shortest odd cycle~$C$ in $B\cup F'$ and we
	may assume that it is neither contained in~$B$ nor in~$F'$.
	The vertices of $\{x,y,z\}$ traversed by~$C$ split it into
	at least two segments whose internal vertices avoid
	$\{x,y,z\}$. Consequently, every such segment lies
	completely in~$B$ or completely in~$F'$ and its endvertices
	form one of the pairs $\{x,y\}$, $\{y,z\}$, or $\{x,z\}$.
	Now we fix one segment and replace every other segment by
	the edge~$xy$, the edge~$yz$, or the~$2$-edge path
	$x$-$y$-$z$, respectively, all of which are available in
	both~$B$ and~$F'$. This results in a closed walk, which is
	not longer than~$C$ and which is completely contained
	in~$F'$ or in~$B$, and a parity check shows that the kept
	segment can be chosen such that this walk is odd.

	It is left to show that $(B\cup F',\ccB_F\cup \{F'\})$ is~$n$-induced.
	For that we first consider a copy $\Fs\in\ccB_F$ and two
	vertices~$u$, $v\in V(\Fs)$ for which there
	exists a shortest~$u$-$v$-path~$P$ in $B\cup F'$ of length
	at most~$n$. Moreover, we can assume
	that~$P$ contains some edge from $E(F')\sm E(B)$. In
	particular, $P$ must contain two vertices
	from $\{x,y,z\}$ and owing to the minimality of~$P$ we infer that
$V(P)\cap\{x,y,z\}=\{x,z\}\,,$
since otherwise we locate a shorter path contained in~$B$
	allowing us to appeal to $\ccB_F\subseteq\binom{B}{F}_n$.
	Consequently, the distance assumption yields the contradiction
$|E(P)|
	>
	\dist_{B-y}(x,\Fs)+\dist_{B-y}(z,\Fs)
	\ge n+1\,.$

	It remains to consider the case when~$u$, $v$ are vertices
	of the added copy $F'$. However, in this case the reasoning
	above applies again: a shortest~$u$-$v$-path of length at
	most~$n$ not contained in~$F'$ either admits a shortcut
	through one of the edges $xy$, $yz\in E(F')$, or it contains
	a segment from~$x$ to~$z$ avoiding~$y$ that passes through a
	vertex $w\in V(B)\sm V(F')$ and, considering the
	copy~$\Fs$ containing~$w$, has length at least
	\[
		\dist_{B-y}(x,\Fs)+\dist_{B-y}(z,\Fs)
		\geq n+1\,,
	\]
	which is absurd. This concludes the proof of
	Lemma~\ref{lem:extend-family}.
\end{proof}

\subsection{Conformity of cycles of copies}
\label{sec:Scycle}
Next we consider special systems resembling the structure of
a cycle of copies of a given graph~$F$.
\begin{dfn}[cycles of copies]\label{def:Scycle}
	For $\l\geq 3$ let $S=(V_S,E_S)$ be a graph with no
	isolated vertices and~$\l$ edges,
	which is either an~$\l$-cycle or a family of pairwise
	vertex disjoint paths.
	Let $E_S=\{e_1,\dots,e_\l\}$ be an enumeration of the edges such that
	only consecutive edges (modulo~$\l$) share a vertex.

	We say an~$F$-system $(Z,\ccZ_F)$ with $Z=\bigcup\ccZ_F$ is
	an \emph{$S$-cycle}
	if there is an enumeration $\ccZ_F=\{F_1,\dots,F_\l\}$ such that
	we have
	\begin{enumerate}[label=\rmlabel]
		\item\label{it:Scyclea}	$E(F_i)\cap E_S=\{e_i,e_{i+1}\}$ and
			$V(F_i)\cap V_S=e_i\cup e_{i+1}$ for all $i\in\ZZ/\l\ZZ$
		\item\label{it:Scycleb} and
			$V(F_i) \cap V(F_{j})=(e_i\cup e_{i+1})\cap (e_j\cup e_{j+1})$
			for all distinct~$i$, $j\in \ZZ/\l\ZZ$.
	\end{enumerate}
	For an~$S$-cycle $(Z,\ccZ_F)$ we refer to $S\subseteq Z$ as
	its \emph{skeleton}.
\end{dfn}
It follows from part~\ref{it:Scycleb} of this definition and
the structural requirement on~$S$ that
for all~$i$, $j\in\ZZ/\l\ZZ$ we have
\[
	V(F_i)\cap V(F_j)\neq\emptyset\quad\Longrightarrow\quad
	j\in\{i-2,i-1,i,i+1,i+2\}\,.
\]
Moreover, for $\l\geq 5$ part~\ref{it:Scyclea} implies for
all distinct~$i$, $j\in\ZZ/\l\ZZ$
\[
	\big|V(F_i)\cap V(F_j)\big|\geq 2
	\ \Longrightarrow\
	j\in\{i-1,i+1\}\tand V(F_i)\cap V(F_j)\in E_S\,.
\]
In particular, for $\l\geq 5$ such~$S$-cycles have clean intersections.
Moreover, for sufficiently large~$\l$ such systems will be~$n$-conform.

\begin{lemma}\label{lem:Scycle-conform}
	For every integer $n\geq 1$ and for every~$S$-cycle
	$(Z,\ccZ_F)$ the following holds.
	If $|E(S)|\geq \max\big\{10,\,\min\{2\og(F)+2,\,4n+4\},\,2n+6\big\}$,
	then $(Z,\ccZ_F)$ is~$n$-conform.
\end{lemma}
\begin{proof}  For the proof we assume that the skeleton~$S$ of the
	given~$S$-cycle $(Z,\ccZ_F)$ is a graph cycle of
	length~$\l=|E(S)|$.
	The proof for the case when~$S$ is a collection of paths is
	almost identical.

	Let $E_S=\{e_1,\dots,e_\l\}$ and $\ccZ_F=\{F_1,\dots,F_\l\}$
	be the enumerations exemplifying that
	$(Z,\ccZ_F)$ is an~$S$-cycle.
	Appealing to $\l-3$ applications of
	Fact~\ref{fact:simple-ext} tells us that the~$F$-system
	$(Z_{--},\ccZ_F^{--})$ given by $\ccZ_F^{--}=\{F_1,\dots,F_{\l-2}\}$
	and $Z_{--}=\bigcup\ccZ_F^{--}$ is~$n$-conform.

	\begin{figure}[H]
		\centering
		\begin{tikzpicture}[
			line width=0.8pt,
			vert/.style={circle, fill, inner sep=0pt, minimum size=4.4pt},
			bandtwoB/.style={gray!50, line width=14.6pt, line cap=round},
			bandtwoF/.style={gray!18, line width=13pt,   line cap=round},
			bandslimmB/.style={gray!50,  line width=11.6pt,  line cap=round},
			bandslimmF/.style={gray!18,  line width=10pt,    line cap=round},
			bandredB/.style={red!55,  line width=8.6pt,  line cap=round},
			bandredF/.style={red!22,  line width=7pt,    line cap=round},
			copylab/.style={font=\footnotesize, gray!40!black},
			elab/.style={font=\footnotesize},
		]
\begin{pgfonlayer}{foreground}
		\node at (-3.55,0) {$\cdots$};
		\node[vert] (u0) at ( -2.80,0) {};
		\node[vert] (u) at ( -1.40,0) {};
		\node[vert, red!60!black] (w) at ( 0.00,0) {}; \node[below=7pt] at (w) {\textcolor{red!60!black}{\footnotesize $a$}};
		\node[vert, red!60!black] (x) at ( 1.40,0) {}; \node[below=4.6pt] at (x) {\textcolor{red!60!black}{\footnotesize $b$}};
		\node[vert, red!60!black] (y) at ( 2.80,0) {}; \node[below=7pt] at (y) {\textcolor{red!60!black}{\footnotesize $y$}};
		\node[vert] (z) at ( 4.20,0) {};
		\node[vert] (v) at ( 5.60,0) {};
		\node at ( 6.35,0) {$\cdots$};
		\end{pgfonlayer}
\begin{scope}[on background layer]
		\draw[bandtwoB] (2.80,0) -- (5.60,0);
		\draw[bandtwoF] (2.80,0) -- (5.60,0);
		\draw[bandtwoB] (-1.40,0) -- (1.40,0);
		\draw[bandtwoF] (-1.40,0) -- (1.40,0);
		\draw[bandslimmB] (-2.80,0) -- (0.00,0);
		\draw[bandslimmF] (-2.80,0) -- (0.00,0);
\draw[bandredB] (0.00,0) -- (2.80,0);
		\draw[bandredF] (0.00,0) -- (2.80,0);
		\end{scope}
\node[copylab] at (4.90,-0.95) {$F_1$};
		\node[copylab, red!60!black] at (1.40,-0.95) {$F_{\l-1}$};
		\node[copylab] at (-0.70,-0.95) {$F_{\l-2}$};
		\node[copylab] at (-2.10,-0.95) {$F_{\l-3}$};
\draw[line width=0.7pt] (u0) -- (u);
		\draw[line width=0.7pt] (u) -- (w);
		\draw[line width=0.7pt, red!60!black] (w) -- (x);
		\draw[line width=0.7pt, gray] (x) -- (y);
		\draw[line width=0.7pt] (y) -- (z);
		\draw[line width=0.7pt] (z) -- (v);
\node[elab] at (-2.10,0.43) {$e_{\l-3}$};
		\node[elab] at (-0.70,0.43) {$e_{\l-2}$};
		\node[elab] at (0.70,0.43) {\textcolor{red!60!black}{$e_{\l-1}$}};
		\node[elab, gray] at (2.10,0.43) {$e_\l$};
		\node[elab] at (3.50,0.43) {$e_1$};
		\node[elab] at (4.90,0.43) {$e_2$};
		\end{tikzpicture}
		\caption{Adding~$F_{\l-1}$ to $(Z_{--},\ccZ_F^{--})$ in the proof of
		Lemma~\ref{lem:Scycle-conform}.}\label{fig:add-Flm}
	\end{figure}

	For the extension of  $\ccZ_F^{--}$ by~$F_{\l-1}$ we shall appeal to
	Lemma~\ref{lem:extend-family-ve} applied for the unique
	vertex~$y$ in $e_1\cap e_\l$
	and the edge $ab=e_{\l-1}$. For that we consider a copy
	$F_i$ with $i\in[\l-2]$ and verify that
	it meets the distance assumption given in
	inequality~\eqref{eq:distance-ve}. Since the vertex set of
	every edge~$e_j$ with $j\in\{2,\dots,i-2\}$
	separates the vertex~$y$ from the vertex set of~$F_i$ in
	$Z_{--}$, every~$y$-$F_i$-path in~$Z_{--}$
	has at least $(i-3)/2$ internal vertices, i.e.,
	\[
		\dist_{Z_{--}}(y,F_i)
		\geq \frac{i-1}{2}\,.
	\]
	Applying the same reasoning for paths connecting~$F_i$ with
	$e_{\l-1}$ in~$Z_{--}$ implies
	\[
		\dist_{Z_{--}}(y,F_i)+\dist_{Z_{--}}(ab,F_i)
		\geq \frac{i-1}{2}+\frac{\l-3-i}{2}
		=\frac{\l}{2}-2
	\]
	and the assumption on $\l=|E(S)|$ yields
	inequality~\eqref{eq:distance-ve}.
	Consequently, the~$F$-system $(Z_-,\ccZ_F^{-})$ given by
	$Z_-=Z_{--}\cup F_{\l-1}$ and
	$\ccZ_F^{-}=\{F_1,\dots,F_{\l-1}\}$ is~$n$-conform.

	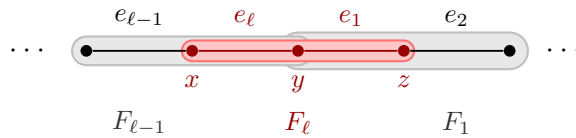
\begin{figure}[H]
		\centering
		\begin{tikzpicture}[
			line width=0.8pt,
			vert/.style={circle, fill, inner sep=0pt, minimum size=4.4pt},
			bandtwoB/.style={gray!50, line width=14.6pt, line cap=round},
			bandtwoF/.style={gray!18, line width=13pt,   line cap=round},
			bandslimmB/.style={gray!50,  line width=11.6pt,  line cap=round},
			bandslimmF/.style={gray!18,  line width=10pt,    line cap=round},
			bandredB/.style={red!55,  line width=8.6pt,  line cap=round},
			bandredF/.style={red!22,  line width=7pt,    line cap=round},
			copylab/.style={font=\footnotesize, gray!40!black},
			elab/.style={font=\footnotesize},
		]
\begin{pgfonlayer}{foreground}
		\node at (-0.75,0) {$\cdots$};
		\node[vert] (w) at ( 0.00,0) {};
		\node[vert, red!60!black] (x) at ( 1.40,0) {}; \node[below=5pt] at (x) {\textcolor{red!60!black}{\footnotesize $x$}};
		\node[vert, red!60!black] (y) at ( 2.80,0) {}; \node[below=5pt] at (y) {\textcolor{red!60!black}{\footnotesize $y$}};
		\node[vert, red!60!black] (z) at ( 4.20,0) {}; \node[below=5pt] at (z) {\textcolor{red!60!black}{\footnotesize $z$}};
		\node[vert] (v) at ( 5.60,0) {};
		\node at ( 6.35,0) {$\cdots$};
		\end{pgfonlayer}
\begin{scope}[on background layer]
		\draw[bandtwoB] (2.80,0) -- (5.60,0);
		\draw[bandtwoF] (2.80,0) -- (5.60,0);
		\draw[bandslimmB] (0.00,0) -- (2.80,0);
		\draw[bandslimmF] (0.00,0) -- (2.80,0);
\draw[bandredB] (1.40,0) -- (4.20,0);
		\draw[bandredF] (1.40,0) -- (4.20,0);
		\end{scope}
\node[copylab] at (4.90,-0.95) {$F_1$};
		\node[copylab] at (0.70,-0.95) {$F_{\l-1}$};
		\node[copylab, red!60!black] at (2.80,-0.95) {$F_\l$};
\draw[line width=0.7pt] (w) -- (x);
		\draw[line width=0.7pt, red!60!black] (x) -- (y);
		\draw[line width=0.7pt, red!60!black] (y) -- (z);
		\draw[line width=0.7pt] (z) -- (v);
\node[elab] at (0.70,0.43) {$e_{\l-1}$};
		\node[elab] at (2.10,0.43) {\textcolor{red!60!black}{$e_\l$}};
		\node[elab] at (3.50,0.43) {\textcolor{red!60!black}{$e_1$}};
		\node[elab] at (4.90,0.43) {$e_2$};
		\end{tikzpicture}
		\caption{Closing the~$S$-cycle in the proof of
		Lemma~\ref{lem:Scycle-conform}. The copy $F_\l$ added
		last is drawn in red.}\label{fig:closing-cycle}
	\end{figure}

	Finally, we add $F_\l$ to $(Z_-,\ccZ_F^{-})$ (see
	Figure~\ref{fig:closing-cycle}) and for that
	we shall appeal to Lemma~\ref{lem:extend-family}
	applied for the edges $e_{\l}=xy$ and $e_1=yz$. In fact,
	the distance assumption~\eqref{eq:distance-P2}
	follows by the same argument, since in $Z_--y$ still the vertex set of
	every edge~$e_j$ with $j\in\{2,\dots,i-2\}$ separates~$z$
	from the vertex set of the copy~$F_i$.
\end{proof}

\subsection{Forcing canonical patterns from homogeneous colourings}
\label{sec:pf-prophom2canon}
For the proof of Proposition~\ref{prop:hom2canon} we shall
consider several constructions---one for each
of the six ordered two-edge graphs.
Roughly speaking, for $T\in\{\separated\,, \overtop\,,
\crossing\,, \ptwo\}$
we will add a family of copies of~$F$
such that if the edges of some copy of~$T$ in a homogeneous colouring
have the same colour, then every copy of~$F$ will be monochromatic.
Moreover, for~$T$ being~$\lcherry$ or~$\rcherry$ it may
result in a common relaxation of
monochromatic or min- or max-coloured copies. However,
this restriction will be strong enough
to guarantee canonical copies in combination with the other
cases (see proof of Proposition~\ref{prop:hom2canon} at the
end of this section).
In the complementary case, when no two edges in~$F$ are monochromatic,
the homogeneous colouring restricted to any copy of~$F$
must be rainbow.

We have to ensure that the final~$F$-system is~$n$-conform. However, the
constructions for different pairs of edges can be carried out
``independently''
and we shall only consider unions of systems that intersect in a
copy of~$F$.
Owing to Lemma~\ref{lem:union-conform} the union of all those
systems will
preserve~$n$-conformity, if we ensure that each single system
is~$n$-conform.

\subsubsection{Forcing equal colours with~$\rseparated$ and~$\rovertop$}
We start with the case $T=\separated$, which pivots on the
following construction.

\begin{dfn}\label{dfn:construct-matching}
	Let~$F_0$ be a copy of an ordered graph~$F$ with
	edges~$e$, $e'=x_1y_1\in E(F_0)$ forming an ordered copy of
	$\separated$, i.e., $\max e<x_1<y_1$.
	For an edge $f\in E(F_0)\sm\{e,e'\}$ and an integer $k\geq 1$
	we define the system \emph{$\ccZ_F^k(e,e';f)$ on~$F_0$} as follows:
	\begin{enumerate}[label=\seplabel]
		\item\label{it:sep1}
			Add new vertices $x_2,\dots,x_{k+1}$ and
			$y_2,\dots,y_{k+1}$ defining the edges $e_i=x_iy_i$
			for $i\in\{2,\dots,k+1\}$ and obeying the ordering
			\[
				\max V(F_0)
				< x_2
				< y_2
				< x_3
				< y_3
				< \dots
				< x_{k+1}
				< y_{k+1}\,.
			\]
		\item\label{it:sep2}
			Let $S=(V_S,E_S)$ be the graph with enumerated edge set~$E_S$ given by
			\[
				E_S=\big\{e_0=f,e_1=e',e_2,e_3,\dots,e_{k+1}\big\}\,.
			\]
		\item\label{it:sep3}
			Let $(Z,\ccZ_F^k(e,e';f))$ be an~$S$-cycle with
			$\ccZ_F^k(e,e';f)=\{F_0,F_1,\dots,F_{k+1}\}$
			(see Figure~\ref{fig:constructB1}),
			where for every $i\neq k+1$ the edges
			$e_i$ and~$e_{i+1}$ are contained in~$F_i$ playing the
			r\^oles of~$e$ and~$e'$ in~$F_0$,
			while for~$F_{k+1}$ the edge $e_0=f$ plays the r\^ole
			of~$e$ and~$e_{k+1}$ plays the r\^ole of $e'$.
			So, with an obvious notation, we ensure
			\[
				(F_0, e, e')
				\cong (F_i, e_i, e_{i+1})
				\cong (F_{k+1}, f, e_{k+1})
			\]
			for every \(i\in [k]\).
	\end{enumerate}
	The ordering of the remaining vertices of the copies of~$F$
	is not essential and any linear extension
	is suitable.

	We also define the union of these systems over all suitable
	choices of~$f$ in the fixed ordered graph~$F_0$
	and set
	\[
		\ccB^k_F(F_0,e,e')=\bigcup_{f\in E(F_0)\sm\{e,e'\}}
		\ccZ^k_F(e,e';f)\,.
	\]
	This is meant to be a free amalgamation over \(F_0\), i.e.,
	for different edges \(f, f'\in E(F_0)\sm\{e,e'\}\) the
	cycles \(\ccZ^k_F(e,e';f)\) and \(\ccZ^k_F(e,e';f')\) are supposed
	to intersect only in the copy \(F_0\).
\end{dfn}

\begin{figure}[H]
	\centering
	\begin{tikzpicture}[scale=1.1,
		line width=0.8pt,
		vert/.style={circle, fill, inner sep=0pt, minimum size=4.4pt},
		bandoneB/.style={gray!75, line width=8.6pt,  line cap=round},
		bandoneF/.style={gray!38, line width=7pt,    line cap=round},
		bandtwoB/.style={gray!50, line width=14.6pt, line cap=round},
		bandtwoF/.style={gray!18, line width=13pt,   line cap=round},
		bandtkB/.style ={gray!60, line width=10.6pt, line cap=round},
		bandtkF/.style ={gray!28, line width=9pt,    line cap=round},
		copylab/.style={font=\footnotesize, gray!40!black},
		elab/.style={font=\footnotesize},
	]
\begin{pgfonlayer}{foreground}
	\node[vert] (fL)  at ( 0.55,0) {};
	\node[vert] (fR)  at ( 1.25,0) {};
	\node[vert] (eL)  at ( 1.95,0) {};
	\node[vert] (eR)  at ( 2.65,0) {};
	\node[vert] (x1)  at ( 3.45,0) {}; \node[below=5pt] at (x1)  {\footnotesize $x_1$};
	\node[vert] (y1)  at ( 4.25,0) {}; \node[below=5pt] at (y1)  {\footnotesize $y_1$};
	\node[vert] (x2)  at ( 5.70,0) {}; \node[below=5pt] at (x2)  {\footnotesize $x_2$};
	\node[vert] (y2)  at ( 6.40,0) {}; \node[below=5pt] at (y2)  {\footnotesize $y_2$};
	\node[vert] (x3)  at ( 7.20,0) {}; \node[below=5pt] at (x3)  {\footnotesize $x_3$};
	\node[vert] (y3)  at ( 7.90,0) {}; \node[below=5pt] at (y3)  {\footnotesize $y_3$};
	\node[vert] (x4)  at ( 8.70,0) {}; \node[below=5pt] at (x4)  {\footnotesize $x_4$};
	\node[vert] (y4)  at ( 9.40,0) {}; \node[below=5pt] at (y4)  {\footnotesize $y_4$};
	\node at (10.15,0) {$\cdots$};
	\node[vert] (xk)  at (10.90,0) {}; \node[below=5pt] at (xk)  {\footnotesize $x_k$};
	\node[vert] (yk)  at (11.60,0) {}; \node[below=5pt] at (yk)  {\footnotesize $y_k$};
	\node[vert] (xk1) at (12.40,0) {}; \node[below=5pt, xshift=-3pt] at (xk1) {\footnotesize $x_{k+1}$};
	\node[vert] (yk1) at (13.10,0) {}; \node[below=5pt, xshift=3pt]  at (yk1) {\footnotesize $y_{k+1}$};
	\end{pgfonlayer}
\begin{scope}[on background layer]
		\fill[gray!12] (2.55,-0.08) ellipse [x radius=2.45, y radius=0.95];
		\node[copylab] at (2.55,-1.45) {$F_0$};
		\draw[gray!50,line width=0.8pt] (2.55,-0.08) ellipse [x radius=2.45, y radius=0.95];
	\end{scope}
\begin{scope}[on background layer]
	\draw[bandtwoB] (5.70,0) -- (7.90,0);
	\draw[bandtwoF] (5.70,0) -- (7.90,0);
	\draw[bandoneB] (3.45,0) -- (6.40,0);
	\draw[bandoneF] (3.45,0) -- (6.40,0);
\draw[bandtkB] (0.55,0) -- (1.25,0);
	\draw[bandtkB] (0.55,0) .. controls (4.2,3.2) and (9.5,3.2) .. (13.10,0);
	\draw[bandtkB] (12.40,0) -- (13.10,0);
	\draw[bandtkF] (0.55,0) -- (1.25,0);
	\draw[bandtkF] (0.55,0) .. controls (4.2,3.2) and (9.5,3.2) .. (13.10,0);
	\draw[bandtkF] (12.40,0) -- (13.10,0);
	\end{scope}
\node[copylab] at ( 6.84, 2.83) {$F_{k+1}\,(\separatedvioletred)$};
	\node[copylab] at ( 4.93,-1.45) {$F_1\,(\separatedgreenred)$};
	\node[copylab] at ( 6.80,-1.45) {$F_2\,(\separatedred)$};
\draw[violet, line width=0.7pt] (fL) -- (fR);
	\node[elab] at (-0.20,0.4) {\textcolor{violet}{$e_0\!=\!f$}};
\draw[green!45!black, line width=0.7pt] (eL) -- (eR);
	\node[elab] at (2.30,0.40) {\textcolor{green!45!black}{$e$}};
	\draw[green!45!black, line width=0.7pt] (x1) -- (y1);
	\node[elab] at (3.85,0.40) {\textcolor{green!45!black}{$e_1=e'$}};
\draw[red, line width=0.7pt] (x2) -- (y2);
	\draw[red, line width=0.7pt] (x3) -- (y3);
	\draw[red, line width=0.7pt] (x4) -- (y4);
	\draw[red, line width=0.7pt] (xk) -- (yk);
	\draw[red, line width=0.7pt] (xk1) -- (yk1);
	\node[elab] at ( 6.05,0.40) {\textcolor{red}{$e_2$}};
	\node[elab] at ( 7.55,0.40) {\textcolor{red}{$e_3$}};
	\node[elab] at ( 9.05,0.40) {\textcolor{red}{$e_4$}};
	\node[elab] at (11.25,0.40) {\textcolor{red}{$e_k$}};
	\node[elab] at (13.45,0.4) {\textcolor{red}{$e_{k+1}$}};
	\end{tikzpicture}
	\caption{Construction of $\ccZ_F^k(e,e';f)$ for~$e$ and~$e'$
	forming a copy of~$\rseparated$
	(see Definition~\ref{dfn:construct-matching}) with the copies
	$F_0$, $F_1$, $F_2$, and~$F_{k+1}$ shaded in grey.
	The coloured copies of~$\rseparated$ next to the labels
	indicate the r\^oles played by the corresponding pairs of
	skeleton edges.}\label{fig:constructB1}
\end{figure}

The relevance of this construction is given by the following lemma.
\begin{lemma}\label{lemma:separated-edges}
	For every ordered graph~$F$ and all integers $k, n\geq 1$ the
	following holds:
	If the edges~$e$ and $e'$  form a monochromatic copy of~$\separated$
	in a $\ccB_F^k(F_0,e,e')$-homogeneous colouring
	$\phi\colon E(B)\to \NN$ for $B=\bigcup \ccB_F^k(F_0,e,e')$,
	then~$\phi$ is monochromatic on every copy of~$F$ in
	$\ccB_F^k(F_0,e,e')$.

	Moreover, if we have $k\geq
	\max\big\{8,\,\min\{2\og(F),\,4n+2\},\,2n+4\big\}$,
	then the~$F$-system $(B,\ccB_F^k(F_0,e,e'))$ is~$n$-conform.
\end{lemma}

\begin{proof}
	Combining the assumption $\phi(e)=\phi(e')$ and invoking
	the $\ccB_F^k(F_0,e,e')$-homogeneity restricted to
	$\ccZ_F^k(e,e';f)$ on
	the skeleton~$S_f$ with edge set
	$\{e^f_{0}=f,e^f_1=e',e_2^f,\dots,e^f_{k+1}\}$ for a
	fixed edge~$f$ yields
	\[
		\phi(e)
		=\phi(e')
		=\phi(e^f_1)
		=\phi(e^f_2)
		=\dots
		=\phi(e^f_{k+1})
		=\phi(e^f_{0})
		=\phi(f)\,.
	\]
	Since $\ccB^k_F(F_0,e,e')=\bigcup_{f\in
	E(F_0)\sm\{e,e'\}} \ccZ^k_F(e,e';f)$, we infer that
	$\phi$ is monochromatic on~$F_0$, and by homogeneity, on
	every copy of~$F$ in $\ccB_F^k(F_0,e,e')$.

	For the moreover-part,
	we note that the skeletons of the cycles $\ccZ^k_F(e,e';f)$
	have size~${k+2}$ and, hence,
	Lemma~\ref{lem:Scycle-conform} implies that all those cycles
	are~$n$-conform. Since any two of these cycles in $\ccB^k_F(F_0,e,e')$
	share only~$F_0$, Lemma~\ref{lem:union-conform} tells us
	that also $\ccB^k_F(F_0,e,e')$ is~$n$-conform.
\end{proof}

We remark that the proof above also shows that~$\phi$ is
monochromatic on all edges of~$B$, i.e.,
every copy of~$F$ in $\ccB^k_F(F_0,e,e')$ is monochromatic with
the same colour, but we make no use of this information.

Next we turn to the ordered matching of type~$\overtop$. In
fact the corresponding construction of
$\ccZ^k_F(e,e';f)$ is almost identical and the only
difference appears in the ordering of the additional vertices
$x_2,y_2,\dots,x_{k+1},y_{k+1}$. More precisely, given a copy
of~$\overtop$ by edges~$e$ and $e'=x_1y_1$
with
\[
	x_1< \min e<\max e <y_1\,,
\]
and another edge~$f$ of~$F_0$,
we replace~\ref{it:sep1} in
Definition~\ref{dfn:construct-matching} by the following:
\begin{enumerate}[label=\otoplabel]
	\item
		Add new vertices $x_2,\dots,x_{k+1}$ and
		$y_2,\dots,y_{k+1}$ defining the edges $e_i=x_iy_i$
		for $i\in\{2,\dots,k+1\}$ and obeying the ordering
		\[
			x_{k+1}
			<\dots
			<x_3
			<x_2
			<\min V(F_0)
			<\max V(F_0)
			< y_2
			< y_3
			< \dots
			< y_{k+1}\,.
		\]
\end{enumerate}
Having defined the cycle of copies $\ccZ^k_F(e,e';f)$ this
way, for~$e$ and $e'$ fixed
we again set
\[
	\ccB^k_F(F_0,e,e')=\bigcup_{f\in E(F_0)\sm\{e,e'\}}
	\ccZ^k_F(e,e';f)\,.
\]
This way we obtain a  version of Lemma~\ref{lemma:separated-edges}
for copies of~$\overtop$ by an identical proof and we omit the details.
\begin{lemma}\label{lemma:overtop-edges}
	For every ordered graph~$F$ and all integers $k, n\geq 1$ the
	following holds:
	If the edges~$e$ and $e'$  form a monochromatic copy of~$\overtop$
	in a $\ccB_F^k(F_0,e,e')$-homogeneous colouring
	$\phi\colon E(B)\to \NN$ for $B=\bigcup \ccB_F^k(F_0,e,e')$,
	then~$\phi$ is monochromatic on every copy of~$F$ in
	$\ccB_F^k(F_0,e,e')$.

	Moreover, if we have $k\geq
	\max\big\{8,\,\min\{2\og(F),\,4n+2\},\,2n+4\big\}$,
	then the~$F$-system $(B,\ccB_F^k(F_0,e,e'))$ is~$n$-conform. \qed
\end{lemma}

\subsubsection{Forcing equal colours with~$\rcrossing$}
Next we focus on the ordered matching of type
$\rcrossing$. As in the previous two matching cases we
construct a cycle of copies.

\begin{dfn}\label{dfn:construct-crossing}
	Let~$F_0$ be a copy of an ordered graph~$F$ whose
	edges~$e$, \mbox{$e'=x_1y_1\in E(F_0)$} form an ordered copy of
	$\rcrossing$, i.e., $\min e<x_1 <\max e<y_1$.
	For an edge \mbox{$f\in E(F_0)\sm\{e,e'\}$} and an integer $k\geq 1$
	we define the system \emph{$\ccZ_F^k(e,e';f)$ on~$F_0$} as follows:
	\begin{enumerate}[label=\crosslabel]
		\item\label{it:cro1}
			Add new vertices $x_2,\dots,x_{k+1}$ and
			$y_2,\dots,y_{k+1}$ defining the edges $e_i=x_iy_i$
			for $i\in\{2,\dots,k+1\}$ and obeying the ordering
			\[
				x_1
				<x_2
				<y_1
				\leq \max V(F_0)
				<x_3
				<y_2
				<x_4
				<y_3
				<\dots
				< x_{k+1}
				<y_{k}
				< y_{k+1}\,.
			\]
			In addition we add the edge $e_{k+2}=x_{k+2}y_{k+2}$ satisfying
			\[
				\min f<x_{k+2}<\max f
				\qqand
				x_{k+1} < y_{k+2} < y_{k+1}\,.
			\]
		\item\label{it:cro2}
			Let $S=(V_S,E_S)$ be the graph with enumerated edge set~$E_S$ given by
			\[
				E_S=\big\{e_0=f,e_1=e',e_2,e_3,\dots,e_{k+1},e_{k+2}\big\}\,.
			\]
		\item\label{it:cro3}
			Let $(Z,\ccZ_F^k(e,e';f))$ be an~$S$-cycle with
			$\ccZ_F^k(e,e';f)=\{F_0,F_1,\dots,F_{k+2}\}$
			(see Figure~\ref{fig:construct-crossing-cycle}),
			where for every $i\in [k]$ the edges
			$e_i$ and~$e_{i+1}$ are contained in~$F_i$ playing the
			r\^oles of~$e$ and~$e'$ in~$F_0$
			while for~$F_{k+1}$
			the edge~$e_{k+2}$ plays the r\^ole of~$e$ and
			$e_{k+1}$ plays the r\^ole of $e'$,
			and for~$F_{k+2}$
			the edge $f=e_0$ plays the r\^ole of~$e$ and~$e_{k+2}$
			plays the r\^ole of $e'$.\footnote{Comparing with the
				constructions for~$\separated$ and~$\overtop$, we
				note that the cycle
				$\ccZ_F^k(e,e';f)$ for~$\crossing$ consists of $k+3$
				copies of~$F$ (instead of $k+2$ copies)
				and the r\^oles of the edges are ``reversed'' twice,
				for the copies~$F_{k+1}$ and~$F_{k+2}$
			(instead of just once).}
			In other words, for every \(i\in [k]\) we ensure
			\[
				(F_0, e, e')
				\cong (F_i, e_i, e_{i+1})
				\cong (F_{k+1}, e_{k+2}, e_{k+1})
				\cong (F_{k+2}, f, e_{k+2})\,.
			\]
	\end{enumerate}
	Moreover, fix some linear extension of the partial order of
	the vertices defined above.
	Finally, we define the union of these systems over
	all suitable choices of~$f$ in the fixed ordered graph~$F_0$ and set
	\[
		\ccB^k_F(F_0,e,e')=\bigcup_{f\in E(F_0)\sm\{e,e'\}}
		\ccZ^k_F(e,e';f)\,.
	\]
\end{dfn}

\begin{figure}[t]
	\centering
	\begin{tikzpicture}[scale=1.1,
		line width=0.8pt,
		declare function={para(\x,\a,\b,\h)=\h*(1-((2*\x-\a-\b)/(\b-\a))^2);},
		vert/.style={circle, fill, inner sep=0pt, minimum size=4.4pt},
		bandoneB/.style={gray!75, line width=8.6pt,  line cap=round},
		bandoneF/.style={gray!38, line width=7pt,    line cap=round},
		bandtwoB/.style={gray!50, line width=14.6pt, line cap=round},
		bandtwoF/.style={gray!18, line width=13pt,   line cap=round},
		bandtkB/.style ={gray!60, line width=10.6pt, line cap=round},
		bandtkF/.style ={gray!28, line width=9pt,    line cap=round},
		copylab/.style={font=\footnotesize, gray!40!black},
		elab/.style={font=\footnotesize},
	]
\begin{pgfonlayer}{foreground}
	\node[vert] (fL)  at ( 0.85,0) {};
	\node[vert] (xk2) at ( 1.50,0) {}; \node[below=5pt] at (xk2) {\footnotesize $x_{k+2}$};
	\node[vert] (fR)  at ( 2.15,0) {};
	\node[vert] (eL)  at ( 3.05,0) {};
	\node[vert] (x1)  at ( 3.70,0) {}; \node[below=5pt] at (x1)  {\footnotesize $x_1$};
	\node[vert] (eR)  at ( 4.35,0) {};
	\node[vert] (x2)  at ( 5.10,0) {}; \node[below=5pt] at (x2)  {\footnotesize $x_2$};
	\node[vert] (y1)  at ( 5.85,0) {}; \node[below=5pt] at (y1)  {\footnotesize $y_1$};
	\node[vert] (x3)  at ( 8.10,0) {}; \node[below=5pt] at (x3)  {\footnotesize $x_3$};
	\node[vert] (y2)  at ( 8.80,0) {}; \node[below=5pt] at (y2)  {\footnotesize $y_2$};
	\node[vert] (x4)  at ( 9.50,0) {}; \node[below=5pt] at (x4)  {\footnotesize $x_4$};
	\node[vert] (y3)  at (10.20,0) {}; \node[below=5pt] at (y3)  {\footnotesize $y_3$};
	\node at (10.95,0) {$\cdots$};
	\node[vert] (xk1) at (11.70,0) {}; \node[below=5pt] at (xk1) {\footnotesize $x_{k+1}$};
	\node[vert] (yk)  at (12.40,0) {}; \node[below=5pt] at (yk)  {\footnotesize $y_k$};
	\node[vert] (yk2) at (13.10,0) {}; \node[below=5pt, xshift=-2pt] at (yk2) {\footnotesize $y_{k+2}$};
	\node[vert] (yk1) at (13.80,0) {}; \node[below=5pt, xshift=2pt] at (yk1) {\footnotesize $y_{k+1}$};
	\end{pgfonlayer}
\begin{scope}[on background layer]
		\fill[gray!12] (3.8,-0.08) ellipse [x radius=3.5, y radius=1.05];
		\node[copylab] at (3.80,-1.45) {$F_0$};
		\draw[gray!50,line width=0.8pt] (3.8,-0.08) ellipse [x radius=3.5, y radius=1.05];
	\end{scope}
\begin{scope}[on background layer]
	\draw[bandtwoB] plot[domain=5.10:8.80,  samples=70] (\x,{para(\x,5.10,8.80,1.50)});
	\draw[bandtwoB] plot[domain=8.10:10.20, samples=60] (\x,{para(\x,8.10,10.20,1.50)});
	\draw[bandtwoF] plot[domain=5.10:8.80,  samples=70] (\x,{para(\x,5.10,8.80,1.50)});
	\draw[bandtwoF] plot[domain=8.10:10.20, samples=60] (\x,{para(\x,8.10,10.20,1.50)});
	\draw[bandoneB] plot[domain=3.70:5.85,  samples=50] (\x,{para(\x,3.70,5.85,0.60)});
	\draw[bandoneB] plot[domain=5.10:8.80,  samples=70] (\x,{para(\x,5.10,8.80,1.50)});
	\draw[bandoneF] plot[domain=3.70:5.85,  samples=50] (\x,{para(\x,3.70,5.85,0.60)});
	\draw[bandoneF] plot[domain=5.10:8.80,  samples=70] (\x,{para(\x,5.10,8.80,1.50)});
	\draw[bandtkB] plot[domain=0.85:2.15, samples=40] (\x,{para(\x,0.85,2.15,0.40)});
	\draw[bandtkB] (1.5,0) .. controls (5.0,3.4) and (9.8,3.4) .. (13.1,0);
	\draw[bandtkF] plot[domain=0.85:2.15, samples=40] (\x,{para(\x,0.85,2.15,0.40)});
	\draw[bandtkF] (1.5,0) .. controls (5.0,3.4) and (9.8,3.4) .. (13.1,0);
\fill[white] (1.50,0) circle (0.21);
	\draw[gray!50, line width=0.8pt] (1.50,0) circle (0.21);
	\fill[white] (5.10,0) circle (0.21);
	\draw[gray!50, line width=0.8pt] (5.10,0) circle (0.21);
	\end{scope}
\node[copylab] at ( 1.50,-1.45) {$F_{k+2}\,(\crossingviolet)$};
	\node[copylab] at ( 5.50,-1.45) {$F_1\,(\crossinggreenred)$};
	\node[copylab] at ( 8.45,-1.45) {$F_2\,(\crossingred)$};
\draw[green!45!black, line width=0.7pt] plot[domain=3.05:4.35,samples=40] (\x,{para(\x,3.05,4.35,0.45)});
	\draw[green!45!black, line width=0.7pt] plot[domain=3.70:5.85,samples=50] (\x,{para(\x,3.70,5.85,0.60)});
	\node[elab] at (3.70,1.115) {\textcolor{green!45!black}{$e$}};
	\node[elab] at (4.78,1.15) {\textcolor{green!45!black}{$e_1\!=\!e'$}};
\draw[violet, line width=0.7pt] plot[domain=0.85:2.15,samples=40] (\x,{para(\x,0.85,2.15,0.40)});
	\draw[violet, line width=0.7pt] (1.5,0) .. controls (5.0,3.4) and (9.8,3.4) .. (13.1,0);
	\node[elab] at (1.45,1.11) {\textcolor{violet}{$e_0\!=\!f$}};
	\node[elab] at (7.15,2.95) {\textcolor{violet}{$e_{k+2}$}};
\draw[red, line width=0.7pt] plot[domain=5.10:8.80,  samples=70] (\x,{para(\x,5.10,8.80,1.50)});
	\draw[red, line width=0.7pt] plot[domain=8.10:10.20, samples=60] (\x,{para(\x,8.10,10.20,1.50)});
	\draw[red, line width=0.7pt, path fading=east] plot[domain=9.50:10.55, samples=40] (\x,{para(\x,9.50,11.60,1.50)});
	\draw[red, line width=0.7pt, path fading=west] plot[domain=11.35:12.40, samples=40] (\x,{para(\x,10.30,12.40,1.50)});
	\draw[red, line width=0.7pt] plot[domain=11.70:13.80, samples=60] (\x,{para(\x,11.70,13.80,1.50)});
	\node[elab] at ( 6.95,1.90) {\textcolor{red}{$e_2$}};
	\node[elab] at ( 9.15,1.90) {\textcolor{red}{$e_3$}};
	\node[elab] at (12.75,1.90) {\textcolor{red}{$e_{k+1}$}};
	\end{tikzpicture}
	\caption{Construction of $\ccZ_F^k(e,e';f)$ for~$e$ and~$e'$
	forming a copy of~$\rcrossing$
	(see Definition~\ref{dfn:construct-crossing}) with $F_0$, $F_1$, $F_2$, and~$F_{k+2}$ shaded in grey.
	The white circles indicate that~$x_2$ and~$x_{k+2}$ are
	not vertices of~$F_0$ and the coloured crossings next to
	the labels indicate the r\^oles played by the skeleton edges.}
\label{fig:construct-crossing-cycle}
\end{figure}
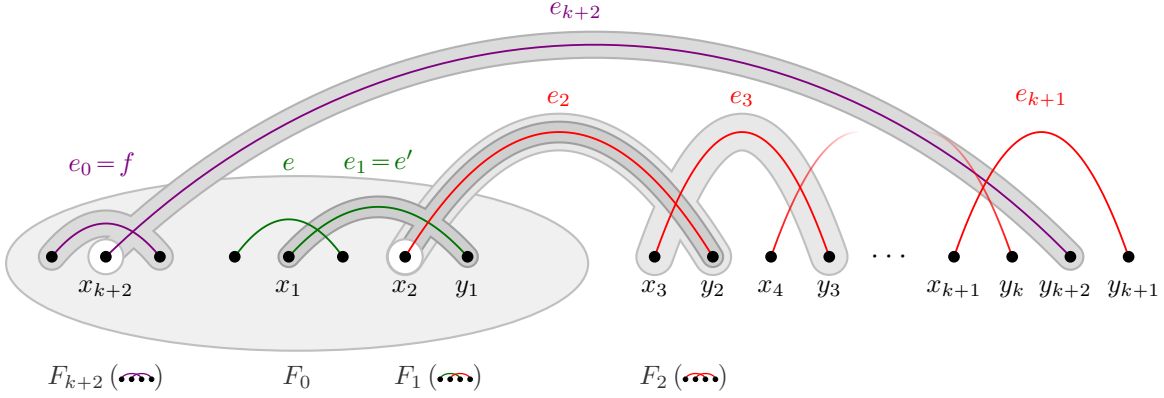

This construction yields an analogous variant of
Lemmata~\ref{lemma:separated-edges}
and~\ref{lemma:overtop-edges} for~$\crossing$.

\begin{lemma}\label{lemma:crossing-edges}
	For every ordered graph~$F$ and all integers $k, n\geq 1$ the
	following holds:
	If the edges~$e$ and $e'$  form a monochromatic copy of~$\crossing$
	in a $\ccB_F^k(F_0,e,e')$-homogeneous colouring
	$\phi\colon E(B)\to \NN$ for $B=\bigcup \ccB_F^k(F_0,e,e')$,
	then~$\phi$ is monochromatic on every copy of~$F$ in
	$\ccB_F^k(F_0,e,e')$.

	Moreover, if we have $k\geq
	\max\big\{8,\,\min\{2\og(F),\,4n+2\},\,2n+4\big\}$,
	then the~$F$-system $(B,\ccB_F^k(F_0,e,e'))$ is~$n$-conform.
\end{lemma}

\begin{proof}
	Owing to the assumption $\phi(e)=\phi(e')$, the
	$\ccB_F^k(F_0,e,e')$-homogeneity of~$\phi$ restricted
	to the skeleton~$S_f$ of $\ccZ_F^k(e,e';f)$ with edge set
	$\{f=e_0^f, e'=e_1^f,\dots,e_{k+2}^f\}$ yields
	\[
		\phi(e)=\phi(e')=\phi(e^f_1)=\phi(e^f_2)=\dots=\phi(e^f_{k+2})=\phi(e^f_{0})=\phi(f)\,.
	\]
	Since $\ccB^k_F(F_0,e,e')=\bigcup_{f\in
	E(F_0)\sm\{e,e'\}} \ccZ^k_F(e,e';f)$, we infer that
	$\phi$ is monochromatic on~$F_0$ and, by homogeneity, it is
	monochromatic
	on every copy of~$F$ in $\ccB_F^k(F_0,e,e')$.

	Since $(B,\ccB_F^k(F_0,e,e'))$ is the union of sufficiently
	long cycles of copies of~$F$
	intersecting in the fixed copy~$F_0$, the conformity follows from
	Lemmata~\ref{lem:union-conform} and~\ref{lem:Scycle-conform}.
\end{proof}

This concludes the discussion for the pairs of edges forming
a matching and it remains to address the situation when the
two edges share a vertex.

\subsubsection{Forcing equal colours with~$\rptwo$}
The definition of the construction in this case is somewhat
similar to that
for~$\crossing$. Roughly speaking, the following construction
arises from
Definition~\ref{dfn:construct-crossing} by identifying
$x_{i+1}$ with~$y_i$ for every $i\in [k]$
and, in addition, $y_k=x_{k+1}=y_{k+2}$ and $\max f=x_{k+2}$.
However, due to these identifications
the resulting system is not a cycle of copies in the sense
of Definition~\ref{def:Scycle} and we therefore
omit the second step in the construction below.
\begin{dfn}\label{dfn:construct-ptwo}
	Let~$F_0$ be a copy of an ordered graph~$F$ with
	edges~$e$, $e'=x_1x_2\in E(F_0)$ forming an ordered copy of
	$\ptwo$, i.e., $\max e=x_1<x_2$.
	For an edge $f\in E(F_0)\sm\{e,e'\}$ and an integer $k\geq 1$
	we define the system \emph{$\ccY_F^k(e,e';f)$ on~$F_0$} as follows:
	\begin{enumerate}[label=\ptwolabel]
		\item
			Add new vertices $x_3,\dots,x_{k+2}$ and edges $e_i=x_ix_{i+1}$
			for $i\in\{2,\dots,k+1\}$ obeying the ordering
			\[
				\max V(F_0) < x_3 < x_4 < \dots < x_{k+1} <x_{k+2}\,.
			\]
			We also add the edge~$e_{k+2}$ joining $\max f$
			with~$x_{k+1}$ and we set $e_0=f$ and $e_1=e'$.
\setcounter{enumi}{2}
		\item\label{it:ptwo3}  Let
			$\ccY_F^k(e,e';f)=\{F_0,F_1,\dots,F_{k+1},F_{k+2}\}$ be
			a system of copies of~$F$
			(see Figure~\ref{fig:construct-ptwo}),
			where for every $i\in[k]$ the edges
			$e_i$ and~$e_{i+1}$ are contained in~$F_i$ playing the
			r\^oles of~$e$ and~$e'$ in~$F_0$,
			while for~$F_{k+1}$
			the edge~$e_{k+2}$ plays the r\^ole of~$e$ and
			$e_{k+1}$ plays the r\^ole of~$e'$,
			and for~$F_{k+2}$
			the edge $f=e_0$ plays the r\^ole of~$e$ and~$e_{k+2}$
			plays the r\^ole of~$e'$. So we guarantee
			\[
				(F_0, e, e')
				\cong (F_i, e_i, e_{i+1})
				\cong (F_{k+1}, e_{k+2}, e_{k+1})
				\cong (F_{k+2}, f, e_{k+2})
			\]
			for all \(i\in [k]\).
			Moreover, besides the necessary intersections given
			above, all those copies are as vertex disjoint as possible.
	\end{enumerate}
	The ordering of the remaining vertices of the copies of~$F$
	is not essential and any linear extension
	is suitable.

	We also define the union of these systems over all suitable
	choices of~$f$ in the fixed ordered graph~$F_0$
	and set
	\[
		\ccB^k_F(F_0,e,e')=\bigcup_{f\in E(F_0)\sm\{e,e'\}}
		\ccY^k_F(e,e';f)\,.
	\]
\end{dfn}

\begin{figure}[H]
	\centering
	\begin{tikzpicture}[
		line width=0.8pt,
		vert/.style={circle, fill, inner sep=0pt, minimum size=4.4pt},
		bandoneB/.style={gray!75, line width=8.6pt,  line cap=round},
		bandoneF/.style={gray!38, line width=7pt,    line cap=round},
		bandtwoB/.style={gray!50, line width=14.6pt, line cap=round},
		bandtwoF/.style={gray!18, line width=13pt,   line cap=round},
		bandtkB/.style ={gray!60, line width=14.6pt, line cap=round},
		bandtkF/.style ={gray!28, line width=13pt,   line cap=round},
		copylab/.style={font=\footnotesize, gray!40!black},
		elab/.style={font=\footnotesize},
	]
\begin{pgfonlayer}{foreground}
	\node[vert] (fL)  at ( 0.55,0) {};
	\node[vert] (fR)  at ( 1.25,0) {};
	\node[vert] (eL)  at ( 2.30,0) {};
	\node[vert] (x1)  at ( 3.00,0) {}; \node[below=5pt] at (x1)  {\footnotesize $x_1$};
	\node[vert] (x2)  at ( 4.10,0) {}; \node[below=5pt] at (x2)  {\footnotesize $x_2$};
	\node[vert] (x3)  at ( 5.65,0) {}; \node[below=5pt] at (x3)  {\footnotesize $x_3$};
	\node[vert] (x4)  at ( 6.85,0) {}; \node[below=5pt] at (x4)  {\footnotesize $x_4$};
	\node[vert] (x5)  at ( 8.05,0) {}; \node[below=5pt] at (x5)  {\footnotesize $x_5$};
	\node at ( 8.95,0) {$\cdots$};
	\node[vert] (xk)  at ( 9.85,0) {}; \node[below=5pt] at (xk)  {\footnotesize $x_k$};
	\node[vert] (xk1) at (11.05,0) {}; \node[below=5pt] at (xk1) {\footnotesize $x_{k+1}$};
	\node[vert] (xk2) at (12.25,0) {}; \node[below=5pt] at (xk2) {\footnotesize $x_{k+2}$};
	\end{pgfonlayer}
\begin{scope}[on background layer]
		\fill[gray!12] (2.55,-0.08) ellipse [x radius=2.45, y radius=0.95];
		\node[copylab] at (2.55,-1.45) {$F_0$};
		\draw[gray!50,line width=0.8pt] (2.55,-0.08) ellipse [x radius=2.45, y radius=0.95];
	\end{scope}
\begin{scope}[on background layer]
	\draw[bandtwoB] (4.10,0) -- (6.85,0);
	\draw[bandtwoF] (4.10,0) -- (6.85,0);
	\draw[bandtwoB] (9.85,0) -- (12.25,0);
	\draw[bandtwoF] (9.85,0) -- (12.25,0);
\draw[bandtkB] (0.55,0) -- (1.25,0);
	\draw[bandtkB] (1.25,0) .. controls (4.5,3.1) and (8.3,3.1) .. (11.05,0);
	\draw[bandtkF] (0.55,0) -- (1.25,0);
	\draw[bandtkF] (1.25,0) .. controls (4.5,3.1) and (8.3,3.1) .. (11.05,0);
	\draw[bandoneB] (3.00,0) -- (5.65,0);
	\draw[bandoneF] (3.00,0) -- (5.65,0);
\draw[bandoneB] (1.25,0) .. controls (4.5,3.1) and (8.3,3.1) .. (11.05,0);
	\draw[bandoneB] (11.05,0) -- (12.25,0);
	\draw[bandoneF] (1.25,0) .. controls (4.5,3.1) and (8.3,3.1) .. (11.05,0);
	\draw[bandoneF] (11.05,0) -- (12.25,0);
	\end{scope}
\node[copylab] at ( 1.50,1.5) {$F_{k+2}\,(\ptwoviolet)$};
	\node[copylab] at ( 3.90,-1.45) {$F_1\,(\ptwogreenred)$};
	\node[copylab] at ( 5.70,-1.45) {$F_2\,(\ptwored)$};
	\node[copylab] at (11.05,-1.45) {$F_k\,(\ptwored)$};
	\node[copylab] at (11.05,1.5) {$F_{k+1}\,(\ptwovioletred)$};
\draw[violet, line width=0.7pt] (fL) -- (fR);
	\node[elab] at (-0.15,0.44) {\textcolor{violet}{$e_0\!=\!f$}};
	\draw[violet, line width=0.7pt] (1.25,0) .. controls (4.5,3.1) and (8.3,3.1) .. (11.05,0);
	\node[elab] at (6.45,2.80) {\textcolor{violet}{$e_{k+2}$}};
\draw[green!45!black, line width=0.7pt] (eL) -- (x1);
	\node[elab] at (2.65,0.40) {\textcolor{green!45!black}{$e_{\phantom{1}}$}};
	\draw[green!45!black, line width=0.7pt] (x1) -- (x2);
	\node[elab] at (3.6,0.45) {\textcolor{green!45!black}{$e_1\!=\!e'$}};
\draw[red, line width=0.7pt] (x2) -- (x3);
	\draw[red, line width=0.7pt] (x3) -- (x4);
	\draw[red, line width=0.7pt] (x4) -- (x5);
	\draw[red, line width=0.7pt] (xk) -- (xk1);
	\draw[red, line width=0.7pt] (xk1) -- (xk2);
	\node[elab] at ( 4.95,0.40) {\textcolor{red}{$e_2$}};
	\node[elab] at ( 6.25,0.40) {\textcolor{red}{$e_3$}};
	\node[elab] at ( 7.45,0.40) {\textcolor{red}{$e_4$}};
	\node[elab] at (10.05,0.40) {\textcolor{red}{$e_k$}};
	\node[elab] at (11.70,0.40) {\textcolor{red}{$e_{k+1}$}};
	\end{tikzpicture}
	\caption{Construction of $\ccY_F^k(e,e';f)$ for~$e$ and~$e'$
	forming a copy of~$\rptwo$
	(see Definition~\ref{dfn:construct-ptwo}) with the copies
	$F_0$, $F_1$, $F_2$, $F_k$, $F_{k+1}$, and~$F_{k+2}$ shaded
	in grey.
	The coloured copies of~$\rptwo$ next to the labels
	indicate the r\^oles played by the corresponding pairs of
	skeleton edges.}\label{fig:construct-ptwo}
\end{figure}

Even though the construction above is not a cycle of copies
of~$F$, this still yields an~$n$-conform~$F$-system
for sufficiently large~\(k\). In fact we shall derive the following
lemma for this construction.

\begin{lemma}{\label{lemma:ptwo}}
	For every ordered graph~$F$ and all integers $k, n\geq 1$ the
	following holds:
	If the edges~$e$ and $e'$  form a monochromatic copy of~$\ptwo$
	in a $\ccB_F^k(F_0,e,e')$-homogeneous colouring
	$\phi\colon E(B)\to \NN$ for $B=\bigcup \ccB_F^k(F_0,e,e')$,
	then~$\phi$ is monochromatic on every copy of~$F$ in
	$\ccB_F^k(F_0,e,e')$.

	Moreover, if we have $k\geq
	\max\big\{9,\,\min\{2\og(F)+1,4n+3\},\,2n+5\big\}$,
	then the~$F$-system $(B,\ccB_F^k(F_0,e,e'))$ is~$n$-conform.
\end{lemma}

\begin{proof}The proof of the first part is identical to that part of
	the proof of Lemma~\ref{lemma:crossing-edges}.
	In fact,  owing to the assumption $\phi(e)=\phi(e')$ the fact that
	$\phi$ is $\ccB_F^k(F_0,e,e')$-homogeneous restricted to the set
	$\{f=e_0^f, e'=e_1^f,\dots,e_{k+2}^f\}$
	in  $\ccY_F^k(e,e';f)$ yields
	\[
		\phi(e)=\phi(e')=\phi(e^f_1)=\phi(e^f_2)=\dots=\phi(e^f_{k+2})=\phi(e^f_{0})=\phi(f)\,.
	\]
	Since $\ccB^k_F(F_0,e,e')=\bigcup_{f\in
	E(F_0)\sm\{e,e'\}} \ccY^k_F(e,e';f)$, we infer that
	$\phi$ is monochromatic on~$F_0$ and, by homogeneity, it is
	monochromatic
	on every copy of~$F$ in $\ccB_F^k(F_0,e,e')$.

	For the proof of the moreover-part we first note that in
	view of Lemma~\ref{lem:union-conform}
	and the definition of $\ccB^k_F(F_0,e,e')$ it suffices to
	establish the conformity of
	$\ccY_F^k(e,e';f)$. Even though $\ccY_F^k(e,e';f)$ is not a
	cycle of copies, below we shall parallel
	the proof of Lemma~\ref{lem:Scycle-conform} to derive the
	conformity with an adjusted choice of~$k$.

	Let $f\in E(F_0)\sm \{e,e'\}$ be fixed. Clearly,
	$\{F_0\}$ is an~$n$-conform
	$F$-system. Owing to the construction in step~\ref{it:ptwo3}
	in Definition~\ref{dfn:construct-ptwo}, we note that by
	adding~$F_i$ sequentially for every $i\in [k]$
	we can invoke Fact~\ref{fact:simple-ext}. Consequently,
	the system $\{F_0,\dots,F_k\}$ is~$n$-conform.

	Next we employ Lemma~\ref{lem:extend-family-ve} to
	extend $\{F_0,\dots,F_k\}$ by~$F_{k+1}$, where $ab=e_{k+1}$
	and $y=\max f$. For that we appeal to the assumed lower
	bound on~$k$, to verify the required distance
	assumption~\eqref{eq:distance-ve} of the lemma in a similar
	way as in the proof of Lemma~\ref{lem:Scycle-conform}.
	We remark that the worst case for the assumption on~$k$
	occurs when the vertex $\max f$ coincides with~$x_2$.

	Finally, we extend $\{F_0,\dots,F_k, F_{k+1}\}$ by
	$F_{k+2}$ in a similar way by appealing to
	Lemma~\ref{lem:extend-family} applied to the edges~$f$ and
	$e_{k+2}$, which share the vertex $\max f$.
\end{proof}

\subsubsection{Forcing non-strict patterns with~$\rlcherry$
and~$\rrcherry$}
In the remaining case we consider monochromatic pairs of
edges isomorphic to~$\lcherry$ or~$\rcherry$
and by symmetry it will be sufficient to address it for~$\lcherry$.
While the general strategy is very similar, there are two
main differences:
firstly, the construction yields a family of copies more
resembling the structure of a \emph{star} instead of a cycle
(see the construction in Definition~\ref{dfn:construct-lcherry} below).
Secondly, homogeneous colourings of the system may result in
non-strict min-coloured copies of~$F$, which is a common relaxation of
monochromatic and min-coloured copies. More precisely we
relax the equivalence in the definition of min-colourings
and we say a colouring $\phi\colon E(F)\to \NN$  is
\emph{non-strictly min-coloured}
if for all edges~$e$, $e'\in E(F)$ we have
\[
	\min e=\min e'\quad\Longrightarrow\quad \phi(e)=\phi(e')\,.
\]
Roughly speaking, in non-strictly min-coloured copies of~$F$
all stars whose centres are smaller than all their leaves
are monochromatic, but some of these stars may use the same colour.
The notion of \emph{non-strictly max-coloured} colourings is
defined analogously (with~$\max$ in place of~$\min$).

\begin{dfn}\label{dfn:construct-lcherry}
	Let~$F_0$ be a copy of an ordered graph~$F$ with
	edges~$e_0=yx_0$ and \mbox{$e_1=yx_1\in E(F_0)$} forming an ordered
	copy of~$\rlcherry$, i.e., $y<x_0<x_1$.
	For another copy of~$\rlcherry$ given by $f_0=ab_0$,
	$f_1=ab_1\in E(F_0)$ with $a<b_0<b_1$
	and an integer $k\geq 1$
	we define the system \emph{$\ccS_F^k(e_0,e_1;f_0,f_1)$ on
	$F_0$} as follows:
	\begin{enumerate}[label=\lchlabel]
		\item\label{it:lche1}
			Add new vertices $x_2,\dots,x_{k+1}$ and edges $e_{i+1}=yx_{i+1}$
			for $i\in[k]$ obeying the ordering
			\[
				\max V(F_0) < x_2 < \dots < x_{k+1}\,.
			\]
			\setcounter{enumi}{2}
		\item\label{it:lche3}  Let
			$\ccS_F^k(e_0,e_1;f_0,f_1)=\{F_0,F_1,\dots,F_{k+1}\}$
			be a system of copies of~$F$ where for every $i\in [k]$ the edges
			$e_i$ and~$e_{i+1}$ are contained in~$F_i$ playing the
			r\^oles of~$e_0$ and~$e_1$ in~$F_0$
			while for~$F_{k+1}$ the edge~$e_0$ plays the r\^ole of
			$f_0$ and~$e_{k+1}$ plays the r\^ole
			of~$f_1$. In other words, we ensure
			\[
				\hspace{1.4cm} (F_0, e_0, e_1)
				\cong (F_i, e_i, e_{i+1}) \text{ for every } i
				\in [k] \text{ and } (F_{k+1}, e_0, e_{k+1})
				\cong (F_0, f_0, f_1)\,.
			\]
			Moreover, besides the necessary intersections given
			above, all those copies are as vertex disjoint as possible.
	\end{enumerate}
	The ordering of the remaining vertices of the copies of~$F$
	is not essential and any linear extension
	is suitable.

	We also define the union of these systems over all choices
	of copies of~$\lcherry$ given by suitable edges
	$f_0=ab_0$ and $f_1=ab_1$ in the fixed ordered graph~$F_0$
	and set
	\[
		\ccB^k_F(F_0,e_0,e_1)=\bigcup_{\substack{ab_0,ab_1\in
		E(F_0)\\a<b_0<b_1}} \ccS^k_F(e_0,e_1;ab_0,ab_1)\,.
	\]
\end{dfn}
As before, the construction is not a cycle of copies of~$F$
in the sense of Definition~\ref{def:Scycle}, but for
sufficiently large~$k$ it still yields an~$n$-conform
$F$-system:

\begin{lemma}\label{lemma:cherries}
	For every ordered graph~$F$ and all integers $k, n\geq 1$ the
	following holds:
	If the edges~$e_0$ and~$e_1$ form a monochromatic copy of~$\lcherry$
	in a $\ccB_F^k(F_0,e_0,e_1)$-homogeneous colouring
	$\phi\colon E(B)\to \NN$ for $B=\bigcup \ccB_F^k(F_0,e_0,e_1)$,
	then~$\phi$ is non-strictly min-coloured on every copy
	of~$F$ in $\ccB_F^k(F_0,e_0,e_1)$.

	Moreover, if we have $k\geq \max\{3,\,n+1\}$, then the
	$F$-system $(B,\ccB_F^k(F_0,e_0,e_1))$ is~$n$-conform.
\end{lemma}

\begin{proof}
	Owing to the homogeneity of~$\phi$ we may assume that the
	monochromatic copy of~$\lcherry$ is contained in~$F_0$, say
	on edges $e_0=yx_0$ and $e_1=yx_1$ with $y<x_0<x_1$. Note
	that being non-strictly min-coloured is equivalent
	to the property that all copies of~$\lcherry$ in~$F_0$ are
	monochromatic.

	For any given pair of
	edges $f_0=ab_0$ and $f_1=ab_1$ in~$F_0$ with $a<b_0<b_1$
	we consider the system $\ccS^k_F(e_0,e_1;f_0,f_1)$.
	It follows from the construction in
	Definition~\ref{dfn:construct-lcherry} combined with the
	homogeneity of~$\phi$
	that all edges $e_0,\dots,e_{k+1}$ in
	$\ccS^k_F(e_0,e_1;f_0,f_1)$ have the same colour. In
	particular, the pair
	of edges playing the r\^oles of~$f_0$ and~$f_1$ in the
	$(k+1)$-st copy of~$F$ in~$\ccS^k_F(e_0,e_1;f_0,f_1)$ have
	the same colour.
	Appealing once more to the homogeneity of~$\phi$, this
	implies $\phi(f_0)=\phi(f_1)$. Therefore, $\phi$ is non-strictly
	min-coloured on~$F_0$ and, by homogeneity, on every copy of
	$F$ in $\ccB_F^k(F_0,e_0,e_1)$.

	For the moreover-part we again exploit
	Lemma~\ref{lem:union-conform}
	and the definition of~$\ccB^k_F(F_0,e_0,e_1)$, which reduces the
	conformity of~$\ccB^k_F(F_0,e_0,e_1)$ to the conformity of
	$\ccS_F^k(e_0,e_1;f_0,f_1)$.
	For that we start with $\{F_0\}$, which clearly is an
	$n$-conform~$F$-system.
	Owing to the construction in step~\ref{it:lche3} in
	Definition~\ref{dfn:construct-lcherry},
	we note that by adding~$F_i$ sequentially for every $i\in [k]$
	we can invoke Fact~\ref{fact:simple-ext}. Consequently,
	the system $\{F_0,\dots,F_k\}$ is~$n$-conform.

	Finally, we extend the system $\{F_0,\dots,F_k\}$ by
	$F_{k+1}$ by applying
	Lemma~\ref{lem:extend-family} over the edges $e_0=yx_0$
	and~$e_{k+1}=yx_{k+1}$, which share the vertex~$y$.
	For the verification of the required distance
	assumption~\eqref{eq:distance-P2} fix two vertices~$u$,
	$v\in V(F_i)$ for
	some $i\in\{0,\dots,k\}$. Observe that every~$u$-$x_0$-path
	avoiding~$y$ must contain
	the vertices $x_0,x_1,\dots,x_i$. Consequently, it has
	length at least~$i$. Similarly, every  $v$-$x_{k+1}$-path
	avoiding~$y$ must pass through the vertices
	$x_{i+1},\dots,x_{k},x_{k+1}$ and has length at least $k-i$.
	Consequently, for $S'=F_0\cup\dots\cup F_k$ the lower bound
	on~$k$ yields
	\[
		\dist_{S'-y}(x_0,F_i)+\dist_{S'-y}(x_{k+1},F_i)
		\geq k
		\geq \max\{3,n+1\}
	\]
	and, hence,  $\ccS_F^k(e_0,e_1;f_0,f_1)$ is~$n$-conform by
	Lemma~\ref{lem:extend-family}.
\end{proof}

We close this section with a final construction dealing with the situation when both
a copy of~$\lcherry$ and a copy of~$\rcherry$ in~$F_0$ are monochromatic. In this situation,
by alternating between both types of cherries, we can transfer the colour to other edges of~$F$, since the centres of the stars are no longer fixed. Consequently, we can enforce monochromatic copies of~$F$ again (see Lemma~\ref{lemma:two-mono-cherries} below).
\begin{dfn}\label{dfn:two-mono-cherries}
Let~$F_0$ be a copy of an ordered graph~$F$ with
	edges~$e_0=x_0x_1$, $e_1=x_0x_2$, $e_0'=x_0'x_1'$, and \mbox{$e_1'=x_0'x_2'\in E(F_0)$} such that $e_0,e_1$ form an ordered copy of~$\rlcherry$, and $e_0',e_1'$ form a copy of~$\rrcherry$, meaning $x_0<x_1<x_2$ and $x_2'<x_1'<x_0'$.
	Additionally, for an edge $f\in E(F_0)\sm\{e_0,e_1,x_1x_2\}$ with $\min f\neq x_0$ and an integer $k\geq 1$
	we define the system \emph{$\ccZ_F^k(e_0,e_1,e_0',e_1';f)$ on
	$F_0$} as follows:
	\begin{enumerate}[label=\bchlabel]
		\item\label{it:speciallche1}
			Fix a vertex~$y_1$ in $\{x_1,x_2\}\sm f$
			and add new vertices $y_2,\dots,y_{2k-1}$
			obeying the ordering
			\begin{align*}
				\quad\qquad y_{2k-2}
				< \dots
				< y_4
				< y_2
				< \min V(F_0)
				< y_1
				\leq \max V(F_0)
				< y_3
				< y_5
				< \dots
				< y_{2k-1}\,.
			\end{align*}
			Moreover, we introduce edges
			$\tilde e_1=x_0y_1$ (which coincides with either~$e_0$ or~$e_1$),
			$\tilde e_{i+1}=y_{i}y_{i+1}$
			for $i\in[2k-2]$,
			$\tilde e_{2k}$ joining $\min f$ with~$y_{2k-1}$, and $\tilde e_{2k+1}=f$.
		\item\label{it:speciallche2}
			Let~$S$ be the graph with enumerated edge set
			$E_S=\big\{\tilde e_1,\tilde e_2,\dots,\tilde e_{2k-1},
				\tilde e_{2k},\tilde e_{2k+1}\big\}$.
		\item\label{it:speciallche3}  Let
			$\{F_1,\dots,F_{2k}\}$
			be a system of copies of~$F$ where for every $i\in [2k-2]$ the edges~$\tilde e_i$ and~$\tilde e_{i+1}$ are contained in~$F_i$ playing the
			r\^oles of
			\begin{itemize}
				\item $e_0'$ and~$e_1'$ in~$F_0$ for odd~$i$
				\item and~$e_0$ and~$e_1$ in~$F_0$ for even~$i$,
			\end{itemize}
			while for~$F_{2k-1}$ the edges~$\tilde e_{2k-1}$ and~$\tilde e_{2k}$ play the r\^oles of $e_1'$ and $e_0'$, and for~$F_{2k}$ we have that~$\tilde e_{2k}$ and~$f$ play the r\^oles of~$e_1$ and~$e_0$. In other words, we ensure
			\[
				(F_0, e_0, e_1)
				\cong (F_{2i}, \tilde e_{2i}, \tilde e_{2i+1})
				\cong (F_{2k}, f, \tilde e_{2k})\,,
			\]
			and
			\[
				(F_0, e_0', e_1')
				\cong (F_{2i-1}, \tilde e_{2i-1}, \tilde e_{2i})
				\cong (F_{2k-1}, \tilde e_{2k}, \tilde e_{2k-1})
			\]
			for all \(i\in [k-1]\).
			Moreover, besides the necessary intersections given
			above, all those copies are as vertex disjoint as possible
			from each other and from~$F_0$. Finally, note that adding~$F_0$ to the family
			$\{F_1,\dots,F_{2k}\}$ defines an~$S$-cycle, which we denote by
			$\ccZ^k_F(e_0,e_1,e_0',e_1';f)$ (see Figure~\ref{fig:zigzag}).
	\end{enumerate}

	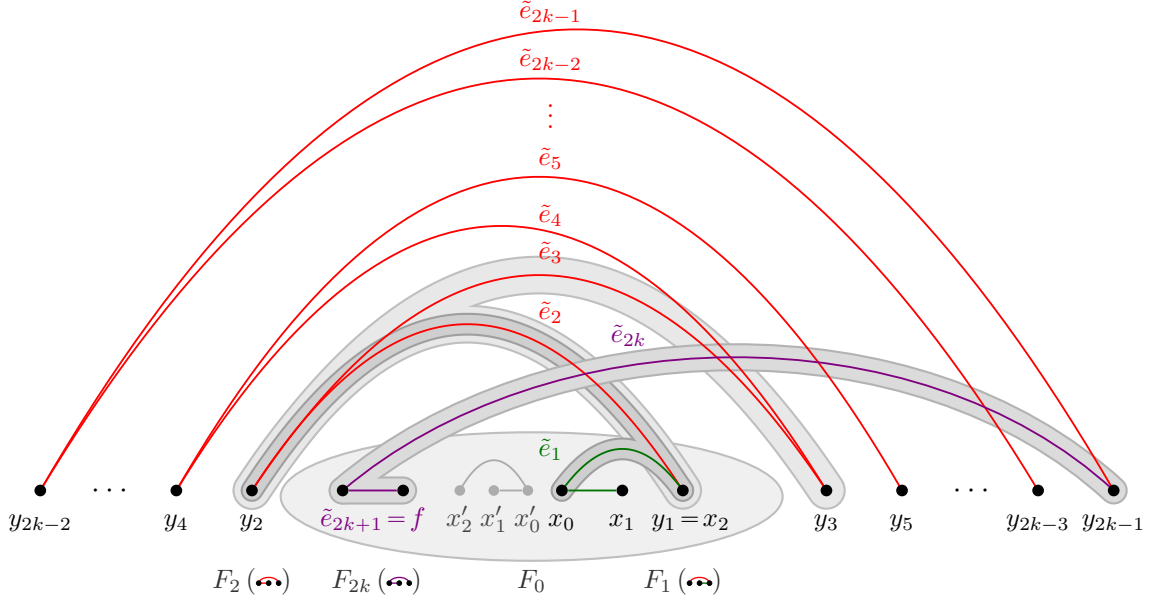
\begin{figure}[t]
	\centering
	\begin{tikzpicture}[
		line width=0.8pt,
		declare function={para(\x,\a,\b,\h)=\h*(1-((2*\x-\a-\b)/(\b-\a))^2);},
		vert/.style={circle, fill, inner sep=0pt, minimum size=4.4pt},
		gvert/.style={circle, fill=gray!65, inner sep=0pt, minimum size=4.0pt},
		bandoneB/.style={gray!75, line width=8.6pt,  line cap=round},
		bandoneF/.style={gray!38, line width=7pt,    line cap=round},
		bandtwoB/.style={gray!50, line width=14.6pt, line cap=round},
		bandtwoF/.style={gray!18, line width=13pt,   line cap=round},
		bandtkB/.style ={gray!60, line width=10.6pt, line cap=round},
		bandtkF/.style ={gray!28, line width=9pt,    line cap=round},
		copylab/.style={font=\footnotesize, gray!40!black},
		elab/.style={font=\footnotesize},
	]
\begin{pgfonlayer}{foreground}
	\node[vert] (y2km2) at ( 0.00,0) {}; \node[below=5pt] at (y2km2) {\footnotesize $y_{2k-2}$};
	\node at ( 0.95,0) {$\cdots$};
	\node[vert] (y4)    at ( 1.80,0) {}; \node[below=5pt] at (y4)    {\footnotesize $y_4$};
	\node[vert] (y2)    at ( 2.80,0) {}; \node[below=5pt] at (y2)    {\footnotesize $y_2$};
	\node[vert] (fL)    at ( 4.00,0) {};
	\node[vert] (fR)    at ( 4.80,0) {};
	\node[below=2.4pt]at (4.40,0) {\footnotesize \textcolor{violet}{$\tilde e_{2k+1}\!=\!f$}};
	\node[vert] (x0)    at ( 6.90,0) {}; \node[below=5pt] at (x0)    {\footnotesize $x_0$};
	\node[vert] (xm)    at ( 7.70,0) {}; \node[below=5pt] at (xm)    {\footnotesize $x_1$};
	\node[vert] (y1)    at ( 8.50,0) {}; \node[below=5pt, xshift=3pt] at (y1)    {\footnotesize $y_1\!=\!x_2$};
	\node[vert] (y3)    at (10.40,0) {}; \node[below=5pt] at (y3)    {\footnotesize $y_3$};
	\node[vert] (y5)    at (11.40,0) {}; \node[below=5pt] at (y5)    {\footnotesize $y_5$};
	\node at (12.35,0) {$\cdots$};
	\node[vert] (y2km3) at (13.20,0) {}; \node[below=5pt] at (y2km3) {\footnotesize $y_{2k-3}$};
	\node[vert] (y2km1) at (14.20,0) {}; \node[below=5pt] at (y2km1) {\footnotesize $y_{2k-1}$};
	\end{pgfonlayer}
\begin{scope}[on background layer]
		\fill[gray!12] (6.5,-0.08) ellipse [x radius=3.32, y radius=0.85];
		\node[copylab] at (6.50,-1.2) {$F_0$};
		\draw[gray!50,line width=0.8pt] (6.5,-0.08) ellipse [x radius=3.32, y radius=0.85];
	\end{scope}
\begin{scope}[on background layer]
	\draw[bandtwoB] plot[domain=2.8:8.5,  samples=70] (\x,{para(\x,2.8, 8.5,2.20)});
	\draw[bandtwoB] plot[domain=2.8:10.4, samples=70] (\x,{para(\x,2.8,10.4,2.85)});
	\draw[bandtwoF] plot[domain=2.8:8.5,  samples=70] (\x,{para(\x,2.8, 8.5,2.20)});
	\draw[bandtwoF] plot[domain=2.8:10.4, samples=70] (\x,{para(\x,2.8,10.4,2.85)});
	\draw[bandoneB] plot[domain=6.9:8.5,  samples=50] (\x,{para(\x,6.9, 8.5,0.55)});
	\draw[bandoneB] plot[domain=2.8:8.5,  samples=70] (\x,{para(\x,2.8, 8.5,2.20)});
	\draw[bandoneF] plot[domain=6.9:8.5,  samples=50] (\x,{para(\x,6.9, 8.5,0.55)});
	\draw[bandoneF] plot[domain=2.8:8.5,  samples=70] (\x,{para(\x,2.8, 8.5,2.20)});
	\draw[bandtkB] (4.0,0) .. controls (6.7,2.3) and (11.6,2.4) .. (14.2,0);
	\draw[bandtkB] (4.0,0) -- (4.8,0);
	\draw[bandtkF] (4.0,0) .. controls (6.7,2.3) and (11.6,2.4) .. (14.2,0);
	\draw[bandtkF] (4.0,0) -- (4.8,0);
	\end{scope}
\node[copylab] at ( 2.80,-1.2) {$F_2\,(\lcherryred)$};
	\node[copylab] at ( 4.45,-1.2) {$F_{2k}\,(\lcherryviolet)$};
	\node[copylab] at ( 8.50,-1.2) {$F_1\,(\rcherrymix)$};
\draw[red, line width=0.7pt] plot[domain=2.8:8.5,  samples=70] (\x,{para(\x, 2.8, 8.5,2.20)});
	\draw[red, line width=0.7pt] plot[domain=2.8:10.4, samples=70] (\x,{para(\x, 2.8,10.4,2.85)});
	\draw[red, line width=0.7pt] plot[domain=1.8:10.4, samples=70] (\x,{para(\x, 1.8,10.4,3.50)});
	\draw[red, line width=0.7pt] plot[domain=1.8:11.4, samples=70] (\x,{para(\x, 1.8,11.4,4.15)});
	\draw[red, line width=0.7pt] plot[domain=0.0:13.2, samples=80] (\x,{para(\x, 0.0,13.2,5.45)});
	\draw[red, line width=0.7pt] plot[domain=0.0:14.2, samples=80] (\x,{para(\x, 0.0,14.2,6.10)});
\node[gvert] (xs2) at (5.55,0) {}; \node[below=1.6pt, gray!55!black] at (xs2) {\footnotesize $x_2'$};
	\node[gvert] (xs1) at (6.00,0) {}; \node[below=1.6pt, gray!55!black] at (xs1) {\footnotesize $x_1'$};
	\node[gvert] (xs0) at (6.45,0) {}; \node[below=1.6pt, gray!55!black] at (xs0) {\footnotesize $x_0'$};
	\draw[gray!65, line width=0.7pt] (xs0) -- (xs1);
	\draw[gray!65, line width=0.7pt] plot[domain=5.55:6.45,samples=40] (\x,{para(\x,5.55,6.45,0.40)});
\draw[green!45!black, line width=0.7pt] (x0) -- (xm);
	\draw[green!45!black, line width=0.7pt] plot[domain=6.9:8.5,samples=50] (\x,{para(\x,6.9,8.5,0.55)});
	\node[elab] at (6.75,0.54) {\textcolor{green!45!black}{$\tilde e_1$}};
\draw[violet, line width=0.7pt] (fL) -- (fR);
	\draw[violet, line width=0.7pt] (4.0,0) .. controls (6.7,2.3) and (11.6,2.4) .. (14.2,0)
		node[elab, black, pos=0.38, above=3pt] {\textcolor{violet}{$\tilde e_{2k}$}};
\node[elab] at (6.75,2.35) {\textcolor{red}{\footnotesize $\tilde e_2$}};
	\node[elab] at (6.75,3.15) {\textcolor{red}{\footnotesize $\tilde e_3$}};
	\node[elab] at (6.75,3.64) {\textcolor{red}{\footnotesize $\tilde e_4$}};
	\node[elab] at (6.75,4.39) {\textcolor{red}{\footnotesize $\tilde e_5$}};
	\node at (6.75,5.06) {\textcolor{red}{\footnotesize $\vdots$}};
	\node[elab] at (6.75,5.69) {\textcolor{red}{\footnotesize $\tilde e_{2k-2}$}};
	\node[elab] at (6.75,6.32) {\textcolor{red}{\footnotesize $\tilde e_{2k-1}$}};
	\end{tikzpicture}
	\caption{Construction of $\ccZ^k_F(e_0,e_1,e_0',e_1';f)$ with
	the copies~$F_0$, $F_1$, $F_2$, and~$F_{2k}$ shaded in grey.
	The coloured cherries next to their labels indicate the
	r\^oles played by the corresponding pairs of consecutive
	skeleton edges and here $y_1=x_2$ is a suitable choice.}\label{fig:zigzag}
\end{figure}

	The ordering of the remaining vertices of the copies of~$F$
	is not essential and any linear extension
	is suitable.

	We also define the union of these systems over all suitable edges
	$f=ab$ in the fixed ordered graph~$F_0$, i.e., $f\notin\{e_0=x_0x_1,e_1=x_0x_2,x_1x_2\}$
	and $\min f\neq x_0$
	and set
	\[
		\ccB^k_F(F_0,e_0,e_1,e'_0,e'_1)=\bigcup_{\substack{ab\in
		E(F_0)\\a<b}} \ccZ^k_F(e_0,e_1,e'_0,e'_1;ab)\,.
	\]
\end{dfn}

\begin{lemma}\label{lemma:two-mono-cherries}
	For every ordered graph~$F$ and all integers $k, n\geq 1$ the
	following holds:
	If the edges~$e_0$ and~$e_1$ form a monochromatic copy of~$\lcherry$,
	and the edges $e_0'$ and $e_1'$ form a monochromatic copy of~$\rcherry$ in a $\ccB_F^k(F_0,e_0,e_1,e_0',e_1')$-homogeneous colouring
	$\phi\colon E(B)\to \NN$ for the graph $B=\bigcup \ccB_F^k(F_0,e_0,e_1,e_0',e_1')$,
	and if~$\phi$ is in addition non-strictly min- and
	non-strictly max-coloured on~$F_0$,
	then~$\phi$ is monochromatic on every copy
	of~$F$ from $\ccB_F^k(F_0,e_0,e_1,e_0',e_1')$.

	Moreover, if $k \geq \max\big\{5,\,\min\{\og(F)+1,\,2n+2\},\,n+3\big\}$, then the
	resulting~$F$-system $(B,\ccB_F^k(F_0,e_0,e_1,e_0',e_1'))$ is~$n$-conform.
\end{lemma}

\begin{proof}
	Owing to the homogeneity of~$\phi$ we may assume that the
	monochromatic copies of~$\lcherry$ and~$\rcherry$ are contained
	in~$F_0$, say on edges $e_0=x_0x_1$ and $e_1=x_0x_2$ with $x_0<x_1<x_2$, and $e_0'=x_0'x_1'$ and $e_1'=x_0'x_2'$ with $x_2'<x_1'<x_0'$.

	For any suitable edge $f=ab$ in~$F_0$ with $a<b$
	we consider the system $\ccZ^k_F(e_0,e_1,e_0',e_1';f)$.
	It follows from the construction in
	Definition~\ref{dfn:two-mono-cherries} combined with the
	homogeneity of~$\phi$
	that all edges $e_0, \tilde e_1,\dots,\tilde e_{2k},\tilde e_{2k+1}=f$
	have the same colour, since any consecutive pair either forms a copy of~$e_0$, $e_1$
	or a copy of $e'_0$, $e'_1$ in some copy of~$F$ from $\ccZ^k_F(e_0,e_1,e_0',e_1';f)$.
	In particular, this implies $\phi(e_0)=\phi(f)$ for every
	suitable edge~$f$.

	For the remaining cases, when $\min f=x_0$ or $f=x_1x_2$ (which are excluded in Definition~\ref{dfn:two-mono-cherries})
	we appeal to the additional assumption: if $\min f=x_0$, then
	$\phi(f)=\phi(e_0)$, since~$\phi$ is non-strictly min-coloured,
	and for $f=x_1x_2$ we have $\max f=x_2=\max e_1$ and the
	non-strictly max-coloured assumption yields
	$\phi(f)=\phi(e_1)=\phi(e_0)$. Consequently, $\phi$ is
	monochromatic on~$F_0$ and, by homogeneity, on
	every copy of~$F$ in $\ccB^k_F(F_0,e_0,e_1,e_0',e_1')$.

	For the moreover-part, note that $\ccZ^k_F(e_0,e_1,e_0',e_1';f)$ are~$S$-cycles and that their skeletons (as described in~\ref{it:speciallche2}) have size $|E(S)|=2k+1$ for any fixed edge~$f$. Hence, Lemma~\ref{lem:Scycle-conform} implies that all those cycles are~$n$-conform. Since any two of these cycles in~$\ccB^k_F(F_0,e_0,e_1,e_0',e_1')$ share only~$F_0$, Lemma~\ref{lem:union-conform} tells us that also $\ccB^k_F(F_0,e_0,e_1,e_0',e_1')$ is~$n$-conform.
\end{proof}

We close this section with the proof of
Proposition~\ref{prop:hom2canon}, which is a simple consequence of
Lemmata~\ref{lemma:separated-edges},
\ref{lemma:overtop-edges}, \ref{lemma:crossing-edges}, \ref{lemma:ptwo}, \ref{lemma:cherries}, and~\ref{lemma:two-mono-cherries}.

\begin{proof}[Proof of Proposition~\ref{prop:hom2canon}]
	If $e(F)\leq 1$, then every colouring of a copy of~$F$ is
	canonical and the~$F$-system $(F,\{F\})$ has all required
	properties. Hence, we may assume $e(F)\geq 2$.

	Given~$F$ and $n\geq 1$, let~$k$ be sufficiently large so that
	we can appeal to the lemmata from this section. In fact, in
	that direction Lemma~\ref{lemma:ptwo} is most restrictive
	and we set
	\[
		k=\max\big\{ 9,\,\min\{2\og(F)+1,4n+3\},\,2n+5 \big\}\,.
	\]
	We fix some copy~$F_0$ of~$F$. If~$F_0$ contains
	both a copy  of~$\lcherry$ on edges $e_0,e_1$ and a copy of~$\rcherry$ on edges
	$e_0',e_1'$, then let
	\[
		\ccB_F(\lcherry\,,\rcherry)=\ccB^k_F(F_0,e_0,e_1,e_0',e_1')\,,
	\]
	which is provided by Definition~\ref{dfn:two-mono-cherries},
	and otherwise $\ccB_F(\lcherry\,,\rcherry)=\emptyset$.
	Next we define the~$F$-system
	\[
		\ccB_F=\bigcup\big\{\ccB^k_F(F_0,e,e')\colon e,e'\in
		E(F_0)\tand e\neq e'\big\}\cup \ccB_F(\lcherry\,,\rcherry)\,,
	\]
	where we appeal to the other constructions for all
	$T\in\{\separated,\, \overtop,\, \crossing,\, \ptwo,\,
	\lcherry,\, \rcherry\}$
	laid out in this section. Since~$F_0$ is fixed, it follows
	from Lemma~\ref{lem:union-conform} combined with
	Lemmata~\ref{lemma:separated-edges}, \ref{lemma:overtop-edges},
	\ref{lemma:crossing-edges}, \ref{lemma:ptwo}, \ref{lemma:cherries}, and~\ref{lemma:two-mono-cherries} that the~$F$-system $(B,\ccB_F)$
	with $B=\bigcup \ccB_F$
	is~$n$-conform.

	It remains to consider a~$\ccB_F$-homogeneous colouring
	$\phi\colon E(B)\to \NN$. Obviously, if no two edges in~$F_0$
	have the same colour, then~$\phi$ is rainbow on~$F_0$ and we are done.
	Hence, we assume there are two edges~$e$ and $e'$ in~$F_0$
	such that $\phi(e)=\phi(e')$.

	In case~$e$ and $e'$ form a copy from $\{ \separated,\,
	\overtop,\, \crossing,\, \ptwo \}$, then
	$\phi$ is monochromatic on~$F_0$ owing to the corresponding lemma among
	Lemmata~\ref{lemma:separated-edges},
	\ref{lemma:overtop-edges}, \ref{lemma:crossing-edges},
	and~\ref{lemma:ptwo}.

	In the remaining case, we may assume that
	$F_0$ contains a monochromatic copy of~$\lcherry$ (the case
	of a monochromatic copy of~$\rcherry$ is symmetric) and in view of
	Lemma~\ref{lemma:cherries} the colouring~$\phi$ is, possibly non-strictly, min-coloured.
	We shall show that either~$F_0$ is in fact (strictly) min-coloured or
	it is monochromatic.

	Suppose~$F_0$ is not min-coloured. This gives rise to two
	edges~$f$ and $f'$ of the same colour, which do not form a copy
	of~$\lcherry$. From the discussion above we can also assume that
	$f$ and~$f'$ do not form a copy from the set $\{ \separated,\,
	\overtop,\, \crossing,\, \ptwo \}$ and, therefore,
	they form a copy of~$\rcherry$. Consequently, $F_0$ contains 
	a monochromatic copy of~$\lcherry$
	and a monochromatic copy
	of~$\rcherry$. Therefore, by
	Lemma~\ref{lemma:cherries} the colouring
	$\phi$ is non-strictly min- and non-strictly max-coloured on~$F_0$.
	Moreover, the system $\ccB_F(\lcherry\,,\rcherry)\neq\emptyset$ and, therefore,
	we conclude with Lemma~\ref{lemma:two-mono-cherries} that~$\phi$ is monochromatic on~$F_0$.
\end{proof}

\section{The Hales--Jewett construction}\label{sec:Hales}
Our proof of Proposition~\ref{prop:final} is based on the
\emph{partite construction
method} introduced by Ne\v{s}et\v{r}il
and R\"{o}dl~\cite{NR-SimpleProofPartCons},
which has found
many applications since then (see, e.g.,~\cites{HK26,
RR-girth} and the references therein).
This method allows
us to decompose the proof of Proposition~\ref{prop:final}
into several easier
problems.
One of these subproblems, treated in the present section,
deals with partite structures.
Therefore all systems and conglomerates appearing in this
section will be ``partite'' in
some sense, while vertex orderings will be irrelevant for
now. In fact, the first moment
they become relevant again is in~\ssign\ref{subsec:Allgemeinbildung}.

The first partite lemma we obtain will be
derived from the Hales--Jewett theorem~\cite{HJ}, which we recall below.
Roughly speaking, it asserts that for every set~$A$ and
every~$r\in\NN$
every~$r$-colouring of a sufficiently high-dimensional
discrete space over~$A$ has a
monochromatic line. It will be convenient to identify for
every set~$A$ and every
dimension~${n\in\NN}$ the Cartesian power~$A^n$ with the set
of all functions from~$[n]$
to~$A$.

\begin{dfn}\label{dfn:comb-emb}
	Given a set~$A$ and a dimension~$n$ let $[n]=C\dcup M$ be a
	partition with $M\ne \vn$
	and let $g\colon C\lra A$ be a function. Defining
	$\ol{g}\colon [n]\times A\lra A$ by
	\[
		\ol{g}(i, a)
		=
		\begin{cases}
			g(i) & \text{ if } i\in C \\
			a      & \text{ if } i\in M
		\end{cases}
	\]
	we call the map
	\[
		\eta_g\colon A\lra A^n\,,
		\quad
		a\longmapsto \bigl(\ol{g}(1, a), \dots, \ol{g}(n, a)\bigr)
	\]
	a {\it combinatorial embedding} and its image $L_g=\eta_g[A]$ the
	associated {\it combinatorial line}.
\end{dfn}

We think of the elements of~$C$ as ``constant coordinates'',
while the coordinates in~$M$
are ``moving''. For instance, if $A=[3]$, $n=4$, and $M=\{2, 4\}$,
then $L_g=\{2{\bf 1}1{\bf 1}, 2{\bf 2}1{\bf 2}, 2{\bf 3}1{\bf 3}\}$
is the line corresponding to the function $g\colon \{1,
3\}\lra [3]$ defined by $g(1)=2$
and $g(3)=1$.

\begin{thm}[Hales \& Jewett]\label{thm:HJ}
	For every finite set~$A$ and every number of colours~$r$
	there exists a dimension~$n$
	such that for every colouring $\phi\colon A^n\lra [r]$
	there exists a monochromatic
	combinatorial line. \qed
\end{thm}

Let us fix any (not necessarily ordered) graph~$F$. By an
{\it $F$-partite graph}
we mean a graph~$L$ together with a distinguished partition
\[
	V(L)=\bigdcup_{x\in V(F)}V_x(L)
\]
of its vertex set into independent sets such that for every
edge $v_xv_y$ of~$L$
with $v_x\in V_x(L)$ and $v_y\in V_y(L)$ the pair $xy$ is an
edge of~$F$. In other words,
the projection from~$L$ to~$F$ sending each vertex class
$V_x(L)$ to~$x$ is demanded to be
a graph homomorphism.

When viewing~$F$ itself as an~$F$-partite graph we always
have the partition $V_x(F)=\{x\}$
in mind. For two~$F$-partite graphs~$L$ and~$N$ the former is
said to be an {\it $F$-partite
subgraph} of the latter if it is a subgraph in the ordinary
sense and, moreover,
$V_x(L)\subseteq V_x(N)$ holds for every $x\in V(F)$. We
write $\binom{N}{L}^\pt$ for the
set of all~$F$-partite copies of~$L$ in~$N$, i.e., for the
set of all induced~$F$-partite
subgraphs of~$N$ which are isomorphic to~$L$ as~$F$-partite
graphs. In particular,
this defines $\binom{N}{F}^\pt$ for every~$F$-partite graph~$N$.

By a {\it partite~$F$-system} we mean a pair $(L, \ccL_F)$
consisting of an~$F$-partite
graph~$L$ and a subset $\ccL_F\subseteq \binom{L}{F}^\pt$. In
other words, an~$F$-system
$(L, \ccL_F)$ is called a partite~$F$-system if both~$L$ and
$\ccL_F$ come with
an~$F$-partite structure. As expected, $(L, \ccL_F)$ is
called an {\it induced
partite~$F$-subsystem} of another partite~$F$-system $(N, \ccN_F)$ if
\begin{itemize}
	\item[$\bullet$] $L$ is an induced~$F$-partite subgraph of~$N$
	\item[$\bullet$] and $\ccL_F=\ccN_F\cap\binom{L}{F}^\pt$.
\end{itemize}
Furthermore we write $\binom{(N, \ccN_F)}{(L, \ccL_F)}^\pt$
for the collection of the induced partite~$F$-subsystems
of~$(N, \ccN_F)$ which are---in an obvious sense---isomorphic
to $(L, \ccL_F)$.
Finally, a {\it partite $(L,\ccL_F)$-conglomerate} is an
$(L,\ccL_F)$-conglomerate
$(N, \ccN_L, \ccN_F)$ such that $\ccN_L\subseteq \binom{(N,
\ccN_F)}{(L, \ccL_F)}^\pt$.

The {\it Hales--Jewett construction}, which we describe in
steps~\ref{it:HJ1}--\ref{it:HJ5} below,
associates to every
partite~$F$-system $(L, \ccL_F)$ and every number of colours~$r$ a
partite $(L,\ccL_F)$-conglomerate $\HJ_r(L, \ccL_F)=(N, \ccN_L, \ccN_F)$
such that $\ccN_L\lra (L, \ccL_F)^F_r$.

In the degenerate case $|\ccL_F|\le 1$ we simply set $(N,
\ccN_F)=(L, \ccL_F)$
and $\ccN_L=\{(L, \ccL_F)\}$. Assuming~$|\ccL_F|\ge 2$ from
now on, we intend to
satisfy the desired partition relation, roughly speaking, by
taking $(N, \ccN_F)$ to
be a sufficiently high-dimensional power of $(L, \ccL_F)$ and
letting~$\ccN_L$
correspond to the combinatorial lines in the associated
discrete space over~$\ccL_F$.
\begin{enumerate}[label=\nlabel]
	\item\label{it:HJ1} Let~$n$ denote the dimension delivered
		by Theorem~\ref{thm:HJ}
		applied to~${A=\ccL_F}$ and~$r$. In other words, $n$ is
		chosen so large that
		for every~$r$-colouring of~$\ccL_F^n$ there exists a
		monochromatic combinatorial
		line.
	\item\label{it:HJ2} The vertex classes of the new graph~$N$
		are the $n^{\mathrm{th}}$
		Cartesian powers of the corresponding vertex classes of~$L$, i.e.,
		we set $V_x(N)=V_x(L)^n$ for every $x\in V(F)$.
	\item\label{it:HJ3} Similarly, the edges of~$N$ are defined
		such that the entire
		graph~$N$ is, in the following sense, a Cartesian power of~$L$.
		Two vertices $(u_1, \dots, u_n)$
		and $(v_1, \dots, v_n)$ of~$N$ are declared to be
		adjacent if $u_iv_i\in E(L)$
		holds for every $i\in [n]$. In particular, $N$ is indeed an~$F$-partite
		graph.
	\item\label{it:HJ4} The set~$\ccN_F$ is constructed
		together with a {\it canonical
		bijection}~${\lambda\colon \ccL_F^n\lra\ccN_F}$.
		Given an~$n$-tuple $\vF=(F_1, \dots, F_n)\in \ccL_F^n$
		we need to specify a corresponding
		copy~${\lambda(\vF)\in \binom{N}{F}^\pt}$,
		which we intend to put into~$\ccN_F$. To this end we denote for
		all pairs $(i, x)\in [n]\times V(F)$ the unique vertex in
		$V(F_i)\cap V_x(L)$
		by~$u_{ix}$. This defines the vertex $(u_{1x}, \dots,
		u_{nx})\in V_x(L)^n$
		and we define~${\lambda(\vF)}$ to have the vertex
		set~$\{(u_{1x}, \dots, u_{nx})\colon x\in V(F)\}\subseteq
		V(N)$. The definition of $E(L)$
		in Step~\ref{it:HJ3} immediately implies that this set
		induces a partite copy
		of~$F$ in~$N$ and justifies the definition
		$\ccN_F=\{\lambda(\vF)\colon \vF\in \ccL_F^n\}$.
	\item\label{it:HJ5} Finally, we implement the idea that the
		copies in~$\ccN_L$ should
		correspond to combinatorial lines. Resuming the notation of
		Definition~\ref{dfn:comb-emb} we suppose that
		$\eta\colon\ccL_F\lra\ccL_F^n$
		denotes the combinatorial embedding corresponding to the
		partition $[n]=C\dcup M$
		with $M\ne\vn$ and the function $g\colon C\lra\ccL_F$.
		For each $i\in C$
		the graph $g(i)$ is a partite copy~$F_i$ of~$F$ in~$L$
		and thus its vertex set can be
		written in the form $V(g(i))=\{u_{ix}\colon x\in V(F)\}$
		with $u_{ix}\in V_x(L)$
		for every $x\in V(F)$. These vertices can be used to
		define combinatorial embeddings
		$\eta_{x}\colon V_x(L)\lra V_x(N)$ associated again to
		the partition $[n]=C\dcup M$
		and to the functions
		\[
			g_x\colon C\lra V_x(L)\,, \quad i\longmapsto u_{ix}\,.
		\]
		In more explicit terms this means that we set
		\[
			\eta_{x}(v_x)=\bigl(\ol{g}_x(1, v_x), \dots,
			\ol{g}_x(n, v_x)\bigr)
		\]
		for every $v_x\in V_x(L)$, where
		\[
			\ol{g}_x(i, v_x)
			=
			\begin{cases}
				u_{ix}, & \text{ if } i\in C, \\
				v_x,      & \text{ if } i\in M.
			\end{cases}
		\]
		As we shall verify in the claim below, for every
		combinatorial embedding \(\eta\) the
		union~$\bigdcup_{x\in V(F)} \eta_{x}[V_x(L)]$ induces
		an~$F$-partite copy of the
		system $(L, \ccL_F)$ in $(N, \ccN_F)$. We write~$\ccN_L$
		for the collection of all
		such copies as~$\eta$ varies over the combinatorial
		embeddings from~$\ccL_F$
		to~$\ccN_F$.
\end{enumerate}

Having thus completed our description of the construction
$\HJ$, it remains to
address the assertion in Step~\ref{it:HJ5} on induced
subsystems. The following
claim provides more information than would be required for this purpose.

\begin{clm}\label{clm:HJlines}
	In the notation of Step~\ref{it:HJ5}, the vertex
	set $\bigdcup_{x\in V(F)} \eta_x[V_x(L)]$
	induces a copy $(L^\eta, \ccL^\eta_F)\in \binom{(N,
	\ccN_F)}{(L, \ccL_F)}^\pt$
	such that $L^\eta$ is distance preserving in~$N$.
	Moreover, for every $F_\star\in\ccN_F$ there is some
	$F_\star^\eta\in \ccL^\eta_F$
	such that $V(L^\eta)\cap V(F_\star)\subseteq
	V(F^\eta_\star)$.\end{clm}

\begin{proof}
	{\bf Stage A:} First we note that, since the set of moving
	coordinates is non-empty,
	the mapping $\eta_\bullet=\bigdcup_{x\in V(F)} \eta_x$ is
	an injection from~$V(L)$ into~$V(N)$.
	Next we convince ourselves that $\eta_\bullet$
	sends edges of~$L$ to edges of~$N$. So given any two
	adjacent vertices $v_x\in V_x(L)$
	and $v_y\in V_y(L)$ we need to show $\eta_x(v_x)\eta_y(v_y)\in E(N)$,
	which due to the definition of~$N$ in Step~\ref{it:HJ3} means
	that $\ol{g}_x(i, v_x)\ol{g}_y(i, v_y)\in E(L)$ holds for
	every $i\in [n]$.
	Indeed for $i\in M$ this is equivalent to $v_xv_y\in E(L)$
	and for $i\in C$
	we have
	\[
		\ol{g}_x(i, v_x)\ol{g}_y(i, v_y)
		=u_{ix}u_{iy}
		\in E(F_i)
		\subseteq E(L)\,,
	\]
	where $F_i\in \ccL_F$ is given by $F_i=g(i)$.
	Thus $\eta_\bullet$ maps~$L$ to an isomorphic subgraph
	$L^\eta$ of~$N$, but
	we have not checked yet that this graph $L^\eta$ is induced in~$N$.

	{\bf Stage B:} Our next step is to show $L^\eta\in \binom
	NL_\dip$, which implies,
	in particular, that~$L^\eta$ is an induced subgraph of~$N$.
	To this end we pick
	an arbitrary moving coordinate $i\in M$, let $\pi_i\colon
	V(N)\lra V(L)$ be the
	projection to the $i^{\mathrm{th}}$ coordinate, and define the
	map $\rho\colon V(N)\lra V(L^\eta)$ by
	\[
		\rho(v_x)=(\eta_x\circ\pi_i)(v_x)
	\]
	whenever $v_x\in V_x(N)$ and $x\in V(F)$. This is the
	composition of the graph
	homomorphisms
	\begin{center}
		\begin{tikzcd}[column sep=1.3cm, row sep=.8cm]
			{N}\arrow{r}{\pi_i} & {L}\arrow{r}{\eta_\bullet} & {L^\eta}\,.
\end{tikzcd}
	\end{center}
	Therefore, $\rho$ maps paths in~$N$ to walks in $L^\eta$.
	In particular, any shortest path in~$N$
	connecting two vertices  of $L^\eta\subseteq N$ is mapped
	to a walk within $L^\eta$. Combined with
	the observation that~$\rho$ is the identity on $L^\eta$,
	this implies that $L^\eta$
	is distance preserving in~$N$.

	{\bf Stage C:}
	Next, for every copy $F'\in \ccL_F$ the copy
	$(\lambda\circ\eta)(F')\in \ccN_F$
	(see Step~\ref{it:HJ4}) has the set of vertices
	\begin{equation}\label{eq:41}
		V\bl(\lambda\circ\eta)(F')\br
		=
		\bigl\{\eta_x(v_x)\colon x\in V(F) \text{ and } v_x\in
		V(F')\cap V_x(L)\bigr\}
		=
		\eta_\bullet[V(F')]\,.
	\end{equation}
	Setting $\ccL^\eta_F=(\lambda\circ\eta)[\ccL_F]$ this tells us
	that $(L^\eta, \ccL^\eta_F)$ is a partite~$F$-subsystem of
	$(N, \ccN_F)$,
	which is isomorphic to $(L, \ccL_F)$, but we have not
	checked yet that it
	is induced, i.e., that
	$\ccL^\eta_F\supseteq\binom{L^\eta}{F}^\pt\cap \ccN_F$
	holds as well. As this follows from the moreover-part of
	our claim by applying
	it to any $F_\star\in \binom{L^\eta}{F}^\pt\cap \ccN_F$, it
	suffices to
	establish the latter statement.

	{\bf Stage D:}
	So let any copy $F_\star\in\ccN_F$ be given. Owing to
	Step~\ref{it:HJ4}
	there exists an~$n$-tuple $\vF=(F_1, \dots,
	F_n)\in\ccL_F^n$ such that
	$F_\star=\lambda(\vF)$. Picking an arbitrary index $i\in M$
	we contend that $F^\eta_\star=(\lambda\circ\eta)(F_i)$ has the
	desired property $V(L^\eta)\cap V(F_\star)\subseteq
	V(F^\eta_\star)$.
	To see this, we consider any vertex in $V(L^\eta)\cap
	V(F_\star)$ and write
	it in the form $\eta_x(v_x)$ with $x\in V(F)$ and $v_x\in V_x(L)$.
	In the $i^{\mathrm{th}}$ coordinate of $\eta_x(v_x)\in V(F_\star)$
	we see the statement $v_x\in V(F_i)$. By the first equality
	in~\eqref{eq:41}
	applied to~$F_i$ in place of~$F'$ this leads to
	$\eta_x(v_x)\in V(F^\eta_\star)$.
\end{proof}

The next proposition summarises all properties of the Hales--Jewett
construction we need
in the sequel.

\begin{prop}\label{lem:HJconstr}
	For every partite~$F$-system $(L, \ccL_F)$ and every
	number of colours~$r$ the partite conglomerate $\HJ_r(L,
	\ccL_F)=(N, \ccN_L, \ccN_F)$
	satisfies
	\begin{enumerate}[label=\alabel]
		\item\label{it:HJa} $\ccN_L\lra (L, \ccL_F)^F_r$,
		\item\label{it:HJb} $\ccN_L\subseteq \binom NL_\dip$,
		\item\label{it:HJc} $\omega(N)=\omega(L)$,
		\item\label{it:HJd} $\og(N)=\og(L)$,
		\item\label{it:HJe} and if for some $m\in\NN$ we have
			$\ccL_F\subseteq\binom LF_m$,
			then $\ccN_F\subseteq \binom NF_m$ follows.
	\end{enumerate}
	Furthermore, if the copies in~$\ccL_F$ have clean intersections, then
	\begin{enumerate}[label=\alabel, resume]
		\item\label{it:HJf} so do the copies in~$\ccN_F$;
		\item\label{it:HJg} and for every $(L^\eta, \ccL^\eta_F)\in \ccN_L$
			and every $F_\star\in \ccN_F\sm \ccL^\eta_F$
			the intersection of~$F_\star$ and~$L^\eta$ is clean.
	\end{enumerate}
\end{prop}

Strictly speaking, there is some abuse of language in
clause~\ref{it:HJb}, because
in reality~$\ccN_L$ consists of copies of the form $(L^\eta,
\ccL^\eta_F)$, and what
we claim is that their underlying graphs~$L^\eta$ are
distance preserving in~$N$.

\begin{proof}
	In the degenerate case $|\ccL_F|\le 1$ we have $(N,
	\ccN_F)=(L, \ccL_F)$
	and $\ccN_L=\{(L, \ccL_F)\}$, so that all seven statements are clear.
	It remains to discuss the case $|\ccL_F|\ge 2$, where the conglomerate
	$(N, \ccN_L, \ccN_F)$ is defined according to the five
	Steps~\ref{it:HJ1}\,--\,\ref{it:HJ5} described above.

	The partition relation~\ref{it:HJa} follows from the fact
	that every colouring
	$\phi\colon\ccN_F\lra [r]$ induces, via~$\lambda^{-1}$, an
	$r$-colouring of the discrete
	space~$\ccL_F^n$. By our choice of~$n$ in Step~\ref{it:HJ1}
	this colouring contains
	a monochromatic combinatorial line, which translates back to a
	copy $(L^\eta, \ccL^\eta_F)\in\ccN_L$ such that
	$\ccL_F^\eta$ is monochromatic with
	respect to~$\phi$.

	Part~\ref{it:HJb} was already obtained in Claim~\ref{clm:HJlines}. The
	statements~\ref{it:HJc} and~\ref{it:HJd} follow from the
	fact, implicit in
	Stage~B of the proof of Claim~\ref{clm:HJlines}, that there
	exists a graph
	homomorphism from~$N$ to~$L$ combined with $\ccN_L\neq\emptyset$.

	Proceeding with~\ref{it:HJe} we assume $\ccL_F\subseteq
	\binom LF_m$ for
	some $m\in\NN$ and consider any copy ${F_\star=\lambda(F_1,
	\dots, F_n)\in\ccN_F}$.
	Let $P\subseteq N$ be a path which connects two
	vertices~${u, w\in V(F_\star)}$, which is
	not completely contained in $F_\star$, and which has length
	at most~$m$.
	Take an arbitrary vertex ${v\in V(P)\sm V(F_\star)}$, write it as
	an~$n$-tuple $v=(v_1, \dots, v_n)$, and notice that due to
	$v\not\in V(F_\star)$
	there is a coordinate direction $i\in [n]$ such that
	$v_i\not\in V(F_i)$.

	Now for each $j\in [n]$ the projection $\pi_j\colon V(N)\lra V(L)$
	to the $j^{\mathrm{th}}$ coordinate is a graph homomorphism
	from~$N$ to~$L$.
	In particular, $\pi_i[P]\subseteq L$ is a walk connecting
	the vertices $\pi_i(u)$
	and $\pi_i(w)$ of~$F_i$ and passing through~$v_i$. Since
	$F_i\in \binom LF_m$,
	this leads to a~$\pi_i(u)$-$\pi_i(w)$-path $Q_i\subseteq
	F_i$, which is strictly
	shorter than~$P$.

	Next, the projection $\psi\colon V(L)\lra V(F)$ mapping
	each vertex class $V_x(L)$
	to~$x$ is a graph homomorphism as well. As $\psi|_{V(F_i)}$
	is a graph isomorphism
	between~$F_i$ and~$F$, it sends~$Q_i$ to
	a $(\psi\circ\pi_i)(u)$-$(\psi\circ\pi_i)(w)$-path $Q\subseteq F$.
	For each $j\in [n]$ the graph isomorphism~$\psi|_{V(F_j)}$
	between~$F_j$ and~$F$
	pulls~$Q$ back to a $\pi_j(u)$-$\pi_j(w)$-path $Q_j\subseteq F_j$.
	In the same way as the~$n$-tuple $(F_1, \dots, F_n)$
	corresponds to the
	copy $F_\star\in \binom NF^\pt$, the~$n$-tuple $(Q_1,
	\dots, Q_n)$ corresponds
	to a~$u$-$w$-path $Q_\star\subseteq F_\star$ of the same
	length as~$Q$.
	This completes the proof of~$F_\star\in\binom NF_m$ and,
	hence, of clause~\ref{it:HJe}.

	Throughout the remainder of the argument we suppose that
	the copies in~$\ccL_F$
	have clean intersections.
	For Part~\ref{it:HJf} we consider two distinct~$n$-tuples
	$(F_1, \dots, F_n)$
	and $(F'_1, \dots, F'_n)$ in~$\ccL_F^n$ such that the copies
	$F_\star=\lambda(F_1, \dots, F_n)$
	and $F_\sstar=\lambda(F'_1, \dots, F'_n)$ in~$\ccN_F$ have
	a set~$e$ of
	at least two vertices in common. We need to show that~$e$
	is a common edge of these
	two copies.

	To this end we pick an index $i\in [n]$ such that $F_i\ne
	F'_i$. Every $x\in V(F)$
	with $e\cap V_x(N)\ne\vn$ needs to satisfy $V(F_i)\cap
	V(F'_i)\cap V_x(L)\ne \vn$.
	This can hold for at most two vertices $x\in V(F)$, because by
	our hypothesis on $(L, \ccL_F)$ the intersection of~$F_i$
	and~$F'_i$ is clean.
	Together with $|e|\ge 2$ this shows that there are two vertices
	of~$F$, say~$x$ and~$y$, such that $V_x(N)$ and $V_y(N)$
	are the only vertex classes
	of~$N$ intersected by~$e$.

	Using again that the intersection of~$F_i$ and~$F'_i$
	is clean, we see that the vertices in $V(F_i)\cap V(F'_i)\cap V_x(L)$
	and $V(F_i)\cap V(F'_i)\cap V_y(L)$ are adjacent. In
	particular, $xy$ is an edge
	of~$F$. Exploiting the product structure of $E(N)$
	(cf.\ Step~\ref{it:HJ3})
	it is now easy to confirm that~$e$ is an edge from~$V_x(N)$
	to $V_y(N)$
	belonging to $E(F_\star)\cap E(F_{\star\star})$.

	Concluding with~\ref{it:HJg} we refer to the moreover-part
	of Claim~\ref{clm:HJlines}, which yields a
	copy~$F_\star^\eta$ in $\ccL^\eta_F$
	such that $V(L^\eta)\cap V(F_\star)\subseteq V(F^\eta_\star)$.
	Due to $F_\star\not\in \ccL^\eta_F$ we have $F_\star\ne
	F^\eta_\star$, so according
	to~\ref{it:HJf} the intersection of $F_\star$ and
	$F^\eta_\star$ is clean.
	It follows that the intersection of~$F_\star$ and $L^\eta$
	is clean as well.
\end{proof}

\section{Clean Ramsey conglomerates}

Our goal in this section is to establish a version of
Proposition~\ref{prop:final}
not addressing distances and odd girth yet, but preserving
clean intersections and the clique number.
For that we define an adequate notion of clean intersections
for conglomerates.

\begin{dfn}\label{dfn:54}
	Given a graph~$F$ and an~$F$-system $(B, \ccB_F)$ with
	clean intersections,
	a $(B, \ccB_F)$-conglomerate $(H, \ccH_B, \ccH_F)$ is said
	to be {\it clean}
	provided that
	\begin{enumerate}[label=\rmlabel]
		\item\label{it:cli} the copies in~$\ccH_F$ have clean intersections;
		\item\label{it:clii} for every $(B^\star,
			\ccB^\star_F)\in \ccH_B$ and
			every $F_{\star\star}\in\ccH_F\sm \ccB^\star_F$ the intersection
			of $B^\star$ and~$F_{\star\star}$ is clean;
		\item\label{it:cliii} and for all distinct $(B^\star, \ccB^\star_F),
			(B^{\star\star}, \ccB^{\star\star}_F)\in \ccH_B$ the intersection
			of~$B^\star$ and~$B^{\star\star}$ is either clean or a copy of~$F$
			belonging to $\ccB^\star_F\cap \ccB^{\star\star}_F$.
	\end{enumerate}
\end{dfn}

The following proposition yields clean conglomerates that
preserve the clique number and
enjoy the Ramsey property for colouring copies of~$F$.

\begin{prop}\label{prop:warmup}
	For every ordered~$F$-system $(B, \ccB_F)$ with clean intersections
	and every number of colours~$r$ there exists a clean
	ordered $(B, \ccB_F)$-conglomerate\footnote{The superscript
	in $\Ups^{(1)}_r$ indicates it is the first member of
	a family of constructions $\Ups^{(n)}_r$ defined in
	\ssign\ref{sec:conform-cong}.}
	\[
		\Ups^{(1)}_r(B, \ccB_F)=(H, \ccH_B, \ccH_F)
	\]
	such that $\ccH_B\lra (B, \ccB_F)^F_r$ and $\omega(H)=\omega(B)$.
\end{prop}

The proof involves two applications of the partite construction method.
In a first major step we obtain the so-called \emph{clean
partite lemma}---an
improved version of the Hales--Jewett construction which yields clean
partite conglomerates (see Proposition~\ref{prop:CPL}). Armed with this
result and utilising the well-known fact that ordered graphs and induced
embeddings form a Ramsey category (quoted as Theorem~\ref{thm:41} below)
we can then derive Proposition~\ref{prop:warmup} by means of a second
partite construction.

\subsection{Pictures}

Before starting a partite construction, one needs to define
an appropriate
notion of pictures and clarify how they amalgamate. For
reasons that will become
apparent later, we need to work with a slightly relaxed
version of conglomerates
to this end.

For a fixed graph~$F$, an~$F$-system $(B, \ccB_F)$ is said to
be a {\it semi-induced
subsystem} of another~$F$-system $(H, \ccH_F)$ if~$B$ is an
induced subgraph of~$H$
and~${\ccB_F\subseteq \ccH_F}$. So in contrast to induced subsystems the
reverse inclusion $\ccB_F\supseteq \ccH_F\cap \binom BF$ is
not necessarily satisfied.
Triples $(H, \ccH_B, \ccH_F)$ such that $(H, \ccH_F)$ is an
$F$-system and~$\ccH_B$
is a collection of semi-induced copies of $(B, \ccB_F)$ in
$(H, \ccH_F)$ will be referred
to as {\it weak $(B, \ccB_F)$-conglomerates}.

\begin{dfn}\label{dfn:picture}
	Suppose that~$F$ is a graph, $(B, \ccB_F)$ is an
	$F$-system, and $(G, \ccG_B, \ccG_F)$
	is a weak $(B, \ccB_F)$-conglomerate. A {\it picture} over
	$(G, \ccG_B, \ccG_F)$
	is a triple $(\Pi, \ccP_B, \ccP_F)$ consisting of a
	$G$-partite graph~$\Pi$
	and subsets $\ccP_F\subseteq \binom{\Pi}{F}$,
	$\ccP_B\subseteq \binom{(\Pi, \ccP_F)}{(B, \ccB_F)}$ such
	that the projection
	$\psi\colon V(\Pi)\lra V(G)$ sending each vertex class
	$V_x(\Pi)$ to~$x$ maps
	each copy in~$\ccP_F$ to a copy in~$\ccG_F$ and each copy
	in~$\ccP_B$ to a copy
	in~$\ccG_B$.
\end{dfn}

We would like to emphasise that the map~$\psi$, called the
{\it canonical projection}
of~$\Pi$, is uniquely determined by the~$G$-partite structure of~$\Pi$.
The vertex classes $V_x(\Pi)$ of a picture are also called
its {\it music lines}.
In figures, they are usually drawn horizontally.
For each copy $\Fbu\in \ccG_F$ we define the {\it
constituent} of $(\Pi, \ccP_B, \ccP_F)$
over~$\Fbu$ to be the following partite~$\Fbu$-system
$(\Pi^\Fbu, \ccP^\Fbu)$.
\begin{itemize}
	\item The underlying graph of $\Pi^\Fbu$ is the subgraph of
		$\Pi$ induced
		by $\bigdcup_{x\in V(\Fbu)}V_x(\Pi)$ and the vertex
		classes determining
		its~$\Fbu$-partite structure are the music lines $V_x(\Pi)$.
	\item The set $\ccP^\Fbu$ consists of all members of
		$\ccP_F$ projected
		by~$\psi$ onto~$\Fbu$. Equivalently, we set
		$\ccP^\Fbu=\{F_\circ\in\ccP_F\colon V(F_\circ)\subseteq
		V(\Pi^\Fbu)\}$.
\end{itemize}

Notice that each copy in~$\ccP_F$ belongs to a unique
constituent, while vertices
and edges of~$\Pi$ can belong to many constituents.

The picture every partite construction begins with, called
{\it picture zero}, can be
defined over every weak $(B, \ccB_F)$-conglomerate $(G,
\ccG_B, \ccG_F)$.
This picture $(\Pi_0, \ccP_{0, B}, \ccP_{0, F})$ is obtained
by taking vertex disjoint
copies of $(B, \ccB_F)$, one for each member of~$\ccG_B$, and
placing them onto
the music lines in such a way that the canonical projection
$\psi_0$ of~$\Pi_0$
sets up a bijection between the set~$\ccP_{0, B}$ of these
copies and~$\ccG_B$.
As an~$F$-system $(\Pi_0, \ccP_{0, F})$ is the disjoint union
of the~$F$-systems
constituting~$\ccP_{0, B}$. For clarity we point out that
picture zero satisfies
\begin{align*}
	|V(\Pi_0)|&=|V(B)|\cdot |\ccG_B|\,, &
	|E(\Pi_0)|&=|E(B)|\cdot |\ccG_B|\,, \\
	\intertext{and}
	|\ccP_{0, B}|&=|\ccG_B|\,, &
	|\ccP_{0, F}|&=|\ccB_F|\cdot |\ccG_B|\,.
\end{align*}

Given a picture $(\Pi, \ccP_B, \ccP_F)$ over a weak $(B,
\ccB_F)$-conglomerate
$(G, \ccG_B, \ccG_F)$, a copy $\Fbu\in\ccG_F$ and a partite
$(\Pi^\Fbu, \ccP^\Fbu)$-conglomerate $(H, \ccH_{\Pi^\Fbu},
\ccH_\Fbu)$ we define
the picture
\[
	(\Sigma, \ccQ_B, \ccQ_F)
	=
	(\Pi, \ccP_B, \ccP_F)\conc (H, \ccH_{\Pi^\Fbu}, \ccH_\Fbu)
\]
to be the result of the following amalgamation procedure.
\begin{itemize}
	\item Recall that each copy
		$(\Pi^\Fbu_\circ, \ccP^\Fbu_\circ)\in \ccH_{\Pi^\Fbu}$ is a
		partite~$\Fbu$-system isomorphic to the constituent
		$(\Pi^\Fbu, \ccP^\Fbu)$ of $(\Pi, \ccP_B, \ccP_F)$
		over~$\Fbu$. Thus each of these copies can be extended to a copy
		$(\Pi_\circ, \ccP_B^\circ, \ccP_F^\circ)$ of this
		picture. We suppose that these
		extensions are constructed as disjoint as possible, so that we have
		\[
			V(\Pi_\circ)\cap V(\Pi_\ccirc)
			=
			V(\Pi^\Fbu_\circ)\cap V(\Pi^\Fbu_\ccirc)
		\]
		for all distinct
		$(\Pi^\Fbu_\circ, \ccP^\Fbu_\circ), (\Pi^\Fbu_\ccirc,
		\ccP^\Fbu_\ccirc)
		\in \ccH_{\Pi^\Fbu}$.
	\item As a graph, $\Sigma$ is the union of~$H$ with all
		graphs $\Pi_\circ$ obtained
		in this way. Moreover, the~$G$-partite structure of
		$\Sigma$ is defined by setting
		\[
			V_x(\Sigma)=
			\begin{cases}
				\bigdcup\bigl\{V_x(\Pi_\circ)\colon (\Pi^\Fbu_\circ,
					\ccP^\Fbu_\circ)
				\in \ccH_{\Pi^\Fbu}\bigr\}\,, & \text{ if } x\not\in V(\Fbu) \cr
				V_x(H)\,, & \text{ if } x\in V(\Fbu)\,.
			\end{cases}
		\]
	\item Put
		\begin{align*}
			\ccQ_F
			&=
			\bigcup\bigl\{\ccP^\circ_F\colon (\Pi^\Fbu_\circ, \ccP^\Fbu_\circ)
			\in \ccH_{\Pi^\Fbu}\bigr\}\cup \ccH_\Fbu
			\intertext{and}
			\ccQ_B
			&=
			\bigcup\bigl\{\ccP^\circ_B\colon (\Pi^\Fbu_\circ, \ccP^\Fbu_\circ)
			\in \ccH_{\Pi^\Fbu}\bigr\}\,.
		\end{align*}
\end{itemize}

It is straightforward to check that $(\Sigma, \ccQ_B,
\ccQ_F)$ is indeed a picture
over~$(G, \ccG_B, \ccG_F)$. The copies $(\Pi_\circ,
\ccP_B^\circ, \ccP_F^\circ)$
of the given picture $(\Pi, \ccP_B, \ccP_F)$ constructed in
the first bullet are
called the {\it standard copies} of~${(\Pi, \ccP_B, \ccP_F)}$
in $(\Sigma, \ccQ_B, \ccQ_F)$.

\subsection{The clean partite lemma}
The partite conglomerates produced by the Hales--Jewett construction have
properties~\ref{it:cli} and~\ref{it:clii} of
Definition~\ref{dfn:54}, but
they are capable of violating~\ref{it:cliii}, so that they
are usually not
clean. Our first application of the partite construction method leads to
another partite lemma, which gives better control over the intersections
of two copies.

Let us point out that when we have a picture
$(\Pi, \ccP_B, \ccP_F)$ over a weak $(B,
\ccB_F)$-conglomerate $(G, \ccG_B, \ccG_F)$,
then by ``forgetting'' the~$G$-partite structure of~$\Pi$ we
can view this picture
as a $(B, \ccB_F)$-conglomerate (which, however, by
Definition~\ref{dfn:picture} cannot be ``only a weak one'').
Accordingly, we can also speak of clean pictures.
For instance, picture zero is clean as soon as the copies in
$\ccB_F$ have clean
intersections.
Next we show that, under suitable assumptions,
which are satisfied
when $(G, \ccG_B, \ccG_F)$ is obtained by means of the
Hales--Jewett construction,
being clean is preserved by partite amalgamations. This
observation can be traced back
to~\cite{BNRR18}*{Lemma~2.12} (see also~\cite{RR-girth}*{Lemma~4.5}).

\begin{lemma}\label{lem:pict-cl}
	Let~$F$ be a graph, $(B, \ccB_F)$ an~$F$-system with clean
	intersections,
	and $(G, \ccG_B, \ccG_F)$ a $(B, \ccB_F)$-conglomerate such that
	\begin{enumerate}[label=\alabel]
		\item\label{it:pict-cl-a} the copies in~$\ccG_F$ have
			clean intersections;
		\item\label{it:pict-cl-b} and for every $(B', \ccB'_F)\in
			\ccG_B$ and
			every $F_\star\in \ccG_F\sm \ccB'_F$
			the intersection of~$B'$ and~$F_\star$ is clean.
	\end{enumerate}
	If we have
	\[
		(\Sigma, \ccQ_B, \ccQ_F)
		=
		(\Pi, \ccP_B, \ccP_F)\conc (H, \ccH_{\Pi^{\Fbu}},
		\ccH_{\Fbu})
	\]
	for two pictures $(\Pi, \ccP_B, \ccP_F)$, $(\Sigma, \ccQ_B,
	\ccQ_F)$ over
	$(G, \ccG_B, \ccG_F)$, the copies in~$\ccH_{\Fbu}$
	have clean intersections,
	and the picture $(\Pi, \ccP_B, \ccP_F)$ is clean, then so
	is $(\Sigma, \ccQ_B, \ccQ_F)$.
\end{lemma}

\begin{proof}
	{\bf Stage A:}
	Let $\psi\colon V(\Sigma)\lra V(G)$ be the canonical projection
	of $(\Sigma, \ccQ_B, \ccQ_F)$. The following projection argument
	will often be used implicitly.

	\begin{clm}\label{clm:1503}
		Let~$Z$ be a subgraph of~$\Sigma$ that either belongs
		to~$\ccQ_F$ or underlies a member of~$\ccQ_B$.
		If for some copy $F'\in\ccG_F$ we have
		$V(F')\not\subseteq \psi[V(Z)]$,
		then the constituent of~$\Sigma$ over~$F'$ has a clean
		intersection with~$Z$.
	\end{clm}

	\begin{proof}
		By our assumptions~\ref{it:pict-cl-a}
		and~\ref{it:pict-cl-b}, the intersection
		of the graphs $\psi[Z]$ and $F'$ is clean, whence
		$\psi[V(Z)]\cap V(F')$ contains
		at most two vertices, and if there are two, then they
		form an edge of $\psi[Z]$.
		Since $\psi|_{V(Z)}$ is a graph isomorphism, this entails
		our claim.
	\end{proof}

	We need to show that the new picture $(\Sigma, \ccQ_B,
	\ccQ_F)$ satisfies
	clauses~\ref{it:cli}\,--\,\ref{it:cliii} of Definition~\ref{dfn:54}.

	\smallskip

	{\bf Stage B:}
	Beginning with~\ref{it:cli}, we consider two distinct
	copies $F_\star, F_\sstar\in \ccQ_F$. If their
	projections $\psi[F_\star], \psi[F_\sstar]\in \ccG_F$ are distinct,
	they have a clean intersection by
	assumption~\ref{it:pict-cl-a} and, hence,
	the intersection of $F_\star$, $F_\sstar$ is clean. So we may assume
	$F'=\psi[F_\star]=\psi[F_\sstar]$ for some~${F'\in \ccG_F}$ from
	now on. If $F'=\Fbu$ is the copy over which the
	amalgamation occurs,
	we appeal to the assumption that the copies in
	$\ccH_{\Fbu}$ have clean
	intersections. Turning now to the case $F'\ne \Fbu$ we
	let $(\Pi_\circ, \ccP_B^\circ, \ccP_F^\circ)$,
	$(\Pi_\ccirc, \ccP_B^\ccirc, \ccP_F^\ccirc)$
	be the standard copies of $(\Pi, \ccP_B, \ccP_F)$
	satisfying $F_\star\in \ccP_F^\circ$ and $F_\sstar\in \ccP_F^\ccirc$,
	respectively. If $\Pi_\circ=\Pi_\ccirc$ our claim follows
	from~$\ccP_F$
	having clean intersections, so the interesting case is
	$\Pi_\circ\ne\Pi_\ccirc$.
	Now we have
	\[
		V(F_\star)\cap V(F_\sstar)
		\subseteq V(\Pi_\circ)\cap V(\Pi_\ccirc)
		\subseteq V(H)\,,
	\]
	whence $\psi[V(F_\star)\cap V(F_\sstar)]\subseteq V(F')\cap V(\Fbu)$.
	The clean intersections of~$\ccG_F$ imply that
	this intersection has at most two elements. Moreover,
	Claim~\ref{clm:1503}
	applied to $Z=F_\star, F_\sstar$ and~$\Fbu$ here in place
	of $F'$ implies that $V(F_\star)\cap V(F_\sstar)\in
	E(F_\star)\cap E(F_\sstar)$,
	if $|V(F_\star)\cap V(F_\sstar)|=2$.

	\smallskip

	{\bf Stage C:}
	Continuing with~\ref{it:clii} we need to show that for
	every~$F$-system
	$(B^\star, \ccB^\star_F)\in \ccQ_B$ and every copy
	$F_\sstar\in \ccQ_F\sm \ccB^\star_F$
	the intersection of $B^\star$ and $F_\sstar$ is clean.
	Let~${(B', \ccB'_F)\in \ccG_B}$
	and $F'\in \ccG_F$ denote the projections of $(B^\star,
	\ccB^\star_F)$ and $F_\sstar$,
	respectively. If $F'\not\in \ccB'_F$ we just need to invoke
	assumption~\ref{it:pict-cl-b} and Claim~\ref{clm:1503}, so
	suppose $F'\in \ccB'_F$
	from now on.

	\smallskip

	{\hspace{2em} \it First case: $F'=\Fbu$}

	\smallskip

	This is equivalent to $F_\sstar\in \ccH_{\Fbu}$. Now $\Fbu=F'\in
	\ccB'_F$ leads to a
	copy $F_\star\in \ccB^\star_F\cap \ccH_{\Fbu}$;
	due to $F_\sstar\not\in \ccB^\star_F$ this copy is distinct
	from $F_\sstar$ and,
	therefore, the intersection of $F_\star, F_\sstar\in \ccH_F$ is clean.
	Together with $V(B^\star)\cap V(F_\sstar)=V(F_\star)\cap V(F_\sstar)$
	this shows that the intersection of~$B^\star$ and $F_\sstar$ is clean.

	\smallskip

	{\hspace{2em} \it Second case: $F'\ne \Fbu$}

	\smallskip

	Now there are standard copies $(\Pi_\circ, \ccP_B^\circ,
	\ccP_F^\circ)$ and
	$(\Pi_\ccirc, \ccP_B^\ccirc, \ccP_F^\ccirc)$
	such that $F_\sstar\in  \ccP_F^\ccirc$ and $(B^\star,
	\ccB_F^\star)\in \ccP_B^\circ$.
	In the special case $\Pi_\circ=\Pi_\ccirc$ our claim
	follows from the fact
	that $(\Pi, \ccP_B, \ccP_F)$ is a clean picture. On the other hand, if
	$\Pi_\circ\ne\Pi_\ccirc$, then
	\[
		\psi[V(B^\star)\cap V(F_\sstar)]\subseteq V(F')\cap V(\Fbu)
	\]
	has at most two elements and by another application of
	Claim~\ref{clm:1503}
	the intersection of $B^\star$ and~$F_\sstar$ is
	indeed clean. This proves assertion~\ref{it:clii} of
	Definition~\ref{dfn:54}.

	\smallskip

	{\bf Stage D:}
	Proceeding with part~\ref{it:cliii} of
	Definition~\ref{dfn:54} we consider two
	$F$-systems
	$(B^\star, \ccB^\star_F)$ and $(B^\sstar, \ccB^\sstar_F)$ in~$\ccQ_B$
	and standard copies
	$(\Pi_\circ, \ccP_B^\circ, \ccP_F^\circ)$,
	$(\Pi_\ccirc, \ccP_B^\ccirc, \ccP_F^\ccirc)$
such that
	$(B^\star, \ccB^\star_F)\in \ccP_B^\circ$
	and
	$(B^\sstar, \ccB^\sstar_F)\in \ccP_B^\ccirc$.
	If these standard copies coincide, we just need to exploit
	that $(\Pi, \ccP_B, \ccP_F)$ is clean.
	So suppose $\Pi_\circ\ne \Pi_\ccirc$ from now on, which yields
	\[
		V(B^\star)\cap V(B^\sstar)\subseteq V(H)\,.
	\]

	Let~$(B', \ccB'_F)$ and~$(B'', \ccB''_F)$ be the projections
	of $(B^\star, \ccB^\star_F)$ and $(B^\sstar,
	\ccB^\sstar_F)$ in~$\ccG_B$.
	In the special case $\Fbu\not\in \ccB'_F$
	hypothesis~\ref{it:pict-cl-b}
	shows that $B^\star$ meets~$H$ in at most two vertices and
	thus the intersection
	of $B^\star$ and $B^\sstar$ is clean, by another application
	of Claim~\ref{clm:1503}.

	So we can henceforth suppose $\Fbu\in \ccB'_F$ and, for the same
	reason, $\Fbu\in \ccB''_F$.
	This means that there are copies $F_\star\in \ccB^\star_F$
	and $F_\sstar\in \ccB^\sstar_F$ projecting to~$\Fbu$. Moreover, due
	to $\Pi_\circ\ne \Pi_\ccirc$ the graphs $B^\star$,
	$B^\sstar$ have the same
	intersection as $F_\star$, $F_\sstar$. So either $F_\star\ne F_\sstar$
	and this intersection is clean, or the intersection in
	question is $F_\star=F_\sstar$,
	and thus a member of~${\ccB^\star_F\cap \ccB^\sstar_F}$.
	This proves~\ref{it:cliii} and, hence, the lemma.
\end{proof}

Suppose now that a graph~$F$, a partite~$F$-system $(B,
\ccB_F)$ with clean intersections,
and a number of colours~$r$ are given. The partite
construction we have in mind produces a clean partite $(B,
\ccB_F)$-conglomerate $(H, \ccH_B, \ccH_F)$ with
$\ccH_B\lra (B, \ccB_F)^F_r$. This conglomerate will be
denoted by~${\CPL_r(B, \ccB_F)}$,
where~$\CPL$ abbreviates ``clean partite lemma''; it is
constructed as follows.
\begin{itemize}
	\item Set $(G, \ccG_B, \ccG_F)=\HJ_r(B, \ccB_F)$ and
		enumerate~$\ccG_F$
		arbitrarily as
		\[
			\ccG_F=\{F(1), \dots, F(N)\}\,.
		\]
	\item Starting with picture zero we recursively define a
		sequence $\bl\Pi_\alpha, \ccP_{\alpha, B}, \ccP_{\alpha,
		F}\br_{\alpha\le N}$
		of pictures over $(G, \ccG_B, \ccG_F)$. If for some
		$\alpha\in [N]$ the
		picture $\bl\Pi_{\alpha-1}, \ccP_{\alpha-1, B},
		\ccP_{\alpha-1, F}\br$ has just
		been constructed, we subject its constituent over $F(\alpha)$ to the
		construction~${\HJ_r(\cdot)}$, thereby obtaining a conglomerate
		$\bl H_\alpha, \ccH_{\Pi_{\alpha-1}^{F(\alpha)}},
		\ccH_{F(\alpha)}\br$,
		and then we set
		\begin{equation}\label{eq:0136}
			\bl\Pi_\alpha, \ccP_{\alpha, B}, \ccP_{\alpha, F}\br
			=
			\bl\Pi_{\alpha-1}, \ccP_{\alpha-1, B}, \ccP_{\alpha-1, F}\br
			\conc
			\bl H_\alpha, \ccH_{\Pi_{\alpha-1}^{F(\alpha)}},
			\ccH_{F(\alpha)}\br\,,
		\end{equation}
		where the amalgamation occurs over $F(\alpha)$.
	\item This recursive construction ends with a terminal
		picture $\bl\Pi_N, \ccP_{N, B}, \ccP_{N, F}\br$.
		As~$\Pi_N$ is a~$G$-partite graph, while~$G$ itself is an
		$F$-partite graph,
		we can define an~$F$-partite structure on~$\Pi_N$ by setting
		\[
			V_x(\Pi_N)=\bigdcup\bigl\{V_y(\Pi_N)\colon y\in V_x(G)\bigr\}
		\]
		for every $x\in V(F)$.
		Owing to the fact that $(G, \ccG_F)$ is a partite~$F$-system, we can
		view~${(\Pi_N, \ccP_{N, F})}$ as a partite~$F$-system as well. For a
		similar reason we can regard $(\Pi_N, \ccP_{N, B}, \ccP_{N, F})$ as
		a partite $(B, \ccB_F)$-conglomerate. This partite conglomerate
		will henceforth be denoted by $\CPL_r(B, \ccB_F)$.
\end{itemize}

\begin{prop}\label{prop:CPL}
	For every graph~$F$, every partite~$F$-system $(B, \ccB_F)$
	with clean intersections,
	and every number of colours~$r$ the partite $(B, \ccB_F)$-conglomerate
	\[
		\CPL_r(B, \ccB_F)=(H, \ccH_B, \ccH_F)
	\]
	is clean and satisfies
	\begin{enumerate}[label=\alabel]
		\item\label{it:CPLa} $\ccH_B\lra (B, \ccB_F)^F_r$,
		\item\label{it:CPLb} $\ccH_B\subseteq \binom HB_\dip$,
		\item\label{it:CPLc} $\omega(H)=\omega(B)$,
		\item\label{it:CPLd} as well as $\og(H)=\og(B)$.
	\end{enumerate}
\end{prop}

\begin{proof}
	Clearly, picture zero is clean. Repeated applications of
	Lemma~\ref{lem:pict-cl}
	show that all pictures $\bl\Pi_\alpha, \ccP_{\alpha, B},
	\ccP_{\alpha, F}\br$
	defined in the course of the partite construction are
	clean. Indeed, the
	vertical system $(G, \ccG_B, \ccG_F)$ fulfils the
	assumptions~\ref{it:pict-cl-a}
	and~\ref{it:pict-cl-b} of Lemma~\ref{lem:pict-cl} due to
	Proposition~\ref{lem:HJconstr}.
	Similarly, for every $\alpha\in [N]$ the copies in the sets
	$\ccH_{F(\alpha)}$
	appearing horizontally have clean intersections, provided
	that in the previous picture
	the copies in~$\ccP_{\alpha-1, F}$ have clean
	intersections. In particular, the final
	picture $\bl\Pi_N, \ccP_{N, B}, \ccP_{N, F}\br$ and thus the
	conglomerate $(H, \ccH_B, \ccH_F)$ is clean.

	The partition relation $\ccP_{N, B}\lra (B, \ccB_F)^F_r$
	follows from a standard
	argument in the area, so we shall be brief about it. The
	main point is that an
	easy induction on $\alpha\in\{0, 1, \dots, N\}$ discloses that
	\begin{quotation}
		\it
		for every colouring $\phi\colon \ccP_{\alpha, F}\lra [r]$
		there is a copy
		$\bl\wt{\Pi}_0, \wt{\ccP}_{0, B}, \wt{\ccP}_{0, F}\br$ of
		picture zero
		in $\bl\Pi_\alpha, \ccP_{\alpha, B}, \ccP_{\alpha, F}\br$
such that for every~${\beta\in [\alpha]}$ the constituent
		of $\bl\wt{\Pi}_0, \wt{\ccP}_{0, B}, \wt{\ccP}_{0, F}\br$
		over $F(\beta)$ is monochromatic\footnote{More precisely,
			if $\bl\wt{\Pi}^{F(\beta)}_0, \wt{\ccP}_{0, F}^{F(\beta)}\br$
			denotes that constituent, we mean that all elements
		of $\wt{\ccP}_{0, F}^{F(\beta)}$ have the same colour.}.
	\end{quotation}

	This statement is vacuously true for $\alpha=0$ and in the
	induction step
	from $\alpha-1$ to~$\alpha$ one exploits that the
	constituent of picture~$\alpha$
	over $F(\alpha)$ is obtained by means of the Hales--Jewett construction
	(cf.~\eqref{eq:0136} and Proposition~\ref{lem:HJconstr}\ref{it:HJa}).

	In particular, for $\alpha=N$ we learn that for every colouring
	$\phi\colon \ccP_{N, F}\lra [r]$ there is a
	copy $\bl\wt{\Pi}_0, \wt{\ccP}_{0, B}, \wt{\ccP}_{0, F}\br$
	of picture zero
	all of whose constituents are monochromatic. This gives rise to an
	auxiliary colouring $\chi\colon [N]\lra [r]$ such that for
	every $\beta\in [N]$
	all copies of~$F$ belonging to the constituent
	of $\bl\wt{\Pi}_0, \wt{\ccP}_{0, B}, \wt{\ccP}_{0, F}\br$
	over $F(\beta)$
	have the colour~$\chi(\beta)$. Now we apply the partition
	relation $\ccG_B\lra (B, \ccB_F)^F_r$, which follows from
	the construction
	of~${(G, \ccG_B, \ccG_F)}$ in the first bullet and from
	Proposition~\ref{lem:HJconstr}\ref{it:HJa},
	to the~$r$-colouring $F(\beta)\longmapsto \chi(\beta)$ of~$\ccG_F$.
	Thereby we find a monochromatic copy $(B', \ccB'_F)\in \ccG_B$.
	Now, if $(B^\star, \ccB^\star_F)$ denotes the copy
	in $\wt{\ccP}_{0, B}\subseteq \ccP_{N, B}$
	projected to $(B', \ccB'_F)$, then $\ccB^\star_F$ is
	monochromatic with respect
	to~$\phi$, which completes the proof of~\ref{it:CPLa}.

	For the remaining three statements we need the canonical
	projection~$\psi$
	from~$\Pi_N$ to~$G$.
	Since~$\psi$ is a graph homomorphism,~\ref{it:CPLc}
	and~\ref{it:CPLd} follow from the
	analogous properties of the Hales--Jewett construction established in
	Proposition~\ref{lem:HJconstr}.

	It remains to address~\ref{it:CPLb} and to this end we consider any
	copy $(B^\star, \ccB_F^\star)\in \ccP_{N, B}$ together with
	a path~$P$ in~$\Pi_N$
	connecting two vertices~$x$ and~$y$ of $B^\star$.
	By Proposition~\ref{lem:HJconstr}\ref{it:HJb} $\psi$
	sends~$B^\star$ to
	some $B'\in \binom GB_\dip$. Thus there is a
	$\psi(x)$-$\psi(y)$-path $Q\subseteq B'$
	whose length is at most the length of~$P$. The graph
	isomorphism $\psi|_{V(B^\star)}$
	between~$B^\star$ and $B'$ allows us to pull~$Q$ back to $B^\star$.
\end{proof}

\begin{rem}\label{rem:57}
	It seems less obvious that
the clean partite lemma also satisfies property~\ref{it:HJe}
	of Proposition~\ref{lem:HJconstr}, meaning that if
	$\ccB_F\subseteq \binom BF_m$ holds
	for some natural number~$m$, then $\ccH_F\subseteq \binom
	HF_m$ follows (provided that
	$\ccB_F$ has clean intersections). We shall return to this
	statement, which is central
	to our ability to preserve odd girth, in
	Corollary~\ref{cor:CPLe} below.
\end{rem}

\subsection{The second partite
construction}\label{subsec:Allgemeinbildung}

The following strengthening of the induced Ramsey theorem was
proved independently
by Abramson and Harrington~\cite{AH78},
and by Ne\v{s}et\v{r}il and R\"{o}dl~\cite{NR-PartitionsFiniteSystems}.

\begin{thm}\label{thm:41}
	For all ordered graphs~$F$ and~$B$ and every number of colours~$r$
	there exists an ordered graph~$G$ satisfying
	\[
		G\lra (B)^F_r\,,
	\]
	which means that for every colouring $\phi\colon \binom
	GF\lra [r]$ there exists
	a copy $B^\star\in \binom GB$ such that~$\binom{B^\star}F$
	is monochromatic.\qed
\end{thm}

This result is sometimes expressed by saying that ordered
graphs form a Ramsey class
or a Ramsey category. We do not introduce the corresponding
abstract language here
and refer the interested reader to the preface of~\cite{RR-girth}.

By choosing~$G$ according to Theorem~\ref{thm:41} and setting
$\ccG_F=\binom GF$
one sees that for every ordered~$F$-system $(B, \ccB_F)$ and
every number of colours~$r$
there exists an ordered~$F$-system~${(G, \ccG_F)}$ such that
for every colouring
$\phi\colon \ccG_F\lra [r]$ there is a semi-induced copy
$(B^\star, \ccB^\star_F)$
of $(B, \ccB_F)$ such that $\ccB^\star_F$ is monochromatic.
In several respects this is
weaker than Proposition~\ref{prop:warmup} and thus we run
another partite construction,
this time with the clean partite lemma.
The advantage of horizontal clean intersections is most
clearly spelled out
in~\cite{RR-girth}*{Lemma~4.7} and~\cite{RR-girth}*{Lemma~4.14}.
Here is a straightforward adaptation of the basic ideas to
our current situation.

\begin{lemma}\label{lem:1738}
	Suppose that $(G, \ccG_B, \ccG_F)$ denotes a weak $(B,
	\ccB_F)$-conglomerate
	for some~$F$-system $(B, \ccB_F)$ and that we have
	\[
		(\Sigma, \ccQ_B, \ccQ_F)
		=
		(\Pi, \ccP_B, \ccP_F)\conc (H, \ccH_{\Pi^\Fbu}, \ccH_\Fbu)
	\]
	for two pictures $(\Pi, \ccP_B, \ccP_F)$, $(\Sigma, \ccQ_B,
	\ccQ_F)$ over
	$(G, \ccG_B, \ccG_F)$ and a partite $(\Pi^\Fbu,
	\ccP^\Fbu)$-conglomerate
	$(H, \ccH_{\Pi^\Fbu}, \ccH_{\Fbu})$.
	If $(\Pi, \ccP_B, \ccP_F)$ and $(H, \ccH_{\Pi^\Fbu},
	\ccH_{\Fbu})$ are clean,
	then so is~$(\Sigma, \ccQ_B, \ccQ_F)$.
\end{lemma}

\begin{proof}
	Again we need to check that $(\Sigma, \ccQ_B, \ccQ_F)$
	satisfies the three
	clauses~\ref{it:cli}\,--\,\ref{it:cliii} of Definition~\ref{dfn:54}.
	As usual, we denote the canonical projection from~$\Sigma$
	to~$G$ by~$\psi$.

	{\bf Stage A:}
	Starting with~\ref{it:cli} we consider two distinct copies
	$F_\star, F_\sstar\in \ccQ_F$.
	If $F_\star, F_\sstar\in \ccH_{\Fbu}$ we appeal to the
	assumption that~$\ccH_{\Fbu}$ has clean
	intersections. So we can assume $F_\star\not\in \ccH_{\Fbu}$
	from now on, so that there is
	a unique standard copy  $(\Pi_\circ, \ccP^\circ_B, \ccP^\circ_F)$
	of $(\Pi, \ccP_B, \ccP_F)$ with~${F_\star\in\ccP^\circ_F}$.

	\smallskip

	{\hspace{2em} \it First case: $F_\sstar\not\in\ccH_{\Fbu}$}

	\smallskip

	This means that there also exists a standard copy
	$(\Pi_\ccirc, \ccP^\ccirc_B, \ccP^\ccirc_F)$ with
	$F_\sstar\in\ccP^\ccirc_F$.
	If $\Pi_\circ=\Pi_\ccirc$ we use that $\ccP^\circ_F$ has
	clean intersections;
	so we can assume $\Pi_\circ\ne\Pi_\ccirc$ from now on,
	which implies that the
	constituents $\Pi_\circ^\Fbu$ and $\Pi_\ccirc^\Fbu$
	intersect each other either
	cleanly or in some copy of $\wt{F}\in \ccH_{\Fbu}$. In
	the former case
	\begin{equation}\label{eq:1144}
		V(F_\star)\cap V(F_\sstar)
		\subseteq
		V(\Pi_\circ)\cap V(\Pi_\ccirc)
		=
		V(\Pi^\Fbu_\circ)\cap V(\Pi^\Fbu_\ccirc)
	\end{equation}
	immediately yields $|V(F_\star)\cap V(F_\sstar)|\le 2$ and
	we are done.
	If, on the other hand, the aforementioned copy~$\wt{F}$ exists,
	then~\eqref{eq:1144} tells us $V(F_\star)\cap
	V(F_\sstar)\subseteq V(\wt{F})$
	and our assertion follows from the fact that~$\wt{F}$ has a
	clean intersection
	with each of $F_\star$, $F_\sstar$.

	\smallskip

	{\hspace{2em} \it Second case: $F_\sstar\in\ccH_{\Fbu}$}

	\smallskip

	If $F_\sstar\in\ccP^\circ_F$ we exploit that the copies in
	$\ccP^\circ_F$
	have clean intersections and otherwise property~\ref{it:clii}
	of $(H, \ccH_{\Pi^\Fbu}, \ccH_{\Fbu})$ being clean yields
	\[
		|V(F_\star)\cap V(F_\sstar)|
		\le |V(\Pi_\circ^\Fbu)\cap V(F_\sstar)|
		\le 2\,.
	\]

	\smallskip

	{\bf Stage B:}
	Moving on to the verification of~\ref{it:clii} we consider two
	copies $(B^\star, \ccB_F^\star)\in \ccQ_B$ and
	$F_\sstar\in\ccQ_F\sm\ccB^\star_F$,
	and need to show that the intersection of $B^\star$ and
	$F_\sstar$ is clean.
	Let $(\Pi_\circ, \ccP^\circ_B, \ccP^\circ_F)$ be a standard copy
	of $(\Pi, \ccP_B, \ccP_F)$ with $(B^\star,
	\ccB_F^\star)\in\ccP^\circ_B$.
	Since $(\Pi, \ccP_B, \ccP_F)$ is clean, only the case
	$F_\sstar\not\in \ccP^\circ_F$
	is interesting. If $F_\sstar\in\ccH_{\Fbu}\sm \ccP^\circ_F$ we
	know that $F_\sstar$
	intersects the constituent $\Pi_\circ^\Fbu$ cleanly and our
	claim follows again.
	So it remains to discuss the case that some standard
	copy $(\Pi_\ccirc, \ccP^\ccirc_B, \ccP^\ccirc_F)$ satisfies
	$\Pi_\circ\ne \Pi_\ccirc$
	and $F_\sstar\in\ccP^\ccirc_F\sm \ccH_{\Fbu}$.
	Now we have
	\[
		V(B^\star)\cap V(F_\sstar)
		\subseteq
		V(\Pi^\Fbu_\circ)\cap V(\Pi^\Fbu_\ccirc)\,.
	\]
	So if the graphs $\Pi^\Fbu_\circ$, $\Pi^\Fbu_\ccirc$ have a
	clean intersection
	we are done immediately and if they intersect in
	some copy $\wt{F}\in \ccP^\circ_F\cap\ccP^\ccirc_F$ we can exploit
	that~$\wt{F}$, $F_\sstar$ have a clean intersection due to
	the already established
	statement~\ref{it:cli}.

	\smallskip

	{\bf Stage C:}
	Finally addressing~\ref{it:cliii} we consider two distinct copies
	$(B^\star, \ccB_F^\star)$ and $(B^\sstar, \ccB_F^\sstar)$
	in~$\ccQ_B$ and
	take two standard copies $(\Pi_\circ, \ccP^\circ_B, \ccP^\circ_F)$,
	$(\Pi_\ccirc, \ccP^\ccirc_B, \ccP^\ccirc_F)$ with
	${(B^\star, \ccB_F^\star)\in \ccP^\circ_B}$
	and $(B^\sstar, \ccB_F^\sstar)\in \ccP^\ccirc_B$,
	respectively. As usual, it suffices
	to treat the case $\Pi_\circ\ne\Pi_\ccirc$, so that we have
	\[
		V(B^\star)\cap V(B^\sstar)
		\subseteq
		V(\Pi^\Fbu_\circ)\cap V(\Pi^\Fbu_\ccirc)\,.
	\]
	If the intersection of $\Pi^\Fbu_\circ$ and
	$\Pi^\Fbu_\ccirc$ is clean,
	everything is clear, and we can assume from now on that
	this intersection
	is some copy $\wt{F}\in \ccP^\circ_F\cap \ccP^\ccirc_F$. As long as
	$\wt{F}\not\in\ccB_F^\star\cap \ccB_F^\sstar$ it follows
	from the already established statement~\ref{it:clii}
	that~$\wt{F}$ intersects one of $B^\star$ and $B^\sstar$
	cleanly, which causes
	the intersection of $B^\star$ and $B^\sstar$ to be clean as well.
	In the remaining case, $\wt{F}\in\ccB_F^\star\cap \ccB_F^\sstar$, the
	second outcome of~\ref{it:cliii} occurs.
\end{proof}

We are now ready to describe the construction~$\Ups^{(1)}$
promised at the beginning of this section.
As mentioned earlier, the proof is rendered
by a partite construction and parallels the proof of the clean partite
lemma---Proposition~\ref{prop:CPL}. In fact, we shall use the
clean partite lemma
here. The other difference concerns the initial conglomerate
$(G,\ccG_B,\ccG_F)$
underlying the pictures in the construction. In the context
of Proposition~\ref{prop:warmup}
the conglomerate $(G,\ccG_B,\ccG_F)$ will be provided by
Theorem~\ref{thm:41}
on Ramsey classes. As a result we have less control over the
structural properties
of $(G,\ccG_B,\ccG_F)$ and this is reflected in the
conclusion of Proposition~\ref{prop:warmup}.
In particular, the delivered conglomerate $(H,\ccH_B,\ccH_F)$
falls short of preserving distances and odd girth.
These shortcomings will be addressed in
\ssign\ref{sec:conform-cong} by another
application of the partite construction method.

\begin{proof}[Proof of Proposition~\ref{prop:warmup}]
	Given an ordered~$F$-system $(B, \ccB_F)$ with clean
	intersections and a
	number of colours~$r$ we proceed as follows.

	\begin{itemize}
		\item By Theorem~\ref{thm:41} there exists an ordered graph~$G$
			satisfying $G\lra (B)^F_r$.
Set
			$\ccG_F=\binom GF$, let
			$\ccG_B$ be the collection of all semi-induced ordered
			copies of $(B, \ccB_F)$
			in the ordered~$F$-system $(G, \ccG_F)$, and observe
			that $(G, \ccG_B, \ccG_F)$
			is an ordered weak $(B, \ccB_F)$-conglomerate.
			Choose an arbitrary enumeration $\ccG_F=\{F(1), \dots,
			F(N)\}$ of~$\ccG_F$.
		\item Starting with picture zero we recursively define a
			sequence $\bl\Pi_\alpha, \ccP_{\alpha, B},
			\ccP_{\alpha, F}\br_{\alpha\le N}$
			of clean pictures over $(G, \ccG_B, \ccG_F)$. If for
			some $\alpha\in [N]$
			we have just constructed the
			clean picture $\bl\Pi_{\alpha-1}, \ccP_{\alpha-1, B},
			\ccP_{\alpha-1, F}\br$,
			then we know that the copies of~$F$ belonging
			to its constituent over $F(\alpha)$ have clean intersections.
			By Proposition~\ref{prop:CPL} the clean partite lemma
			$\CPL_r(\cdot)$
			applied to this constituent yields a clean partite conglomerate
			$\bl H_\alpha, \ccH_{\Pi_{\alpha-1}^{F(\alpha)}},
			\ccH_{F(\alpha)}\br$.
			By Lemma~\ref{lem:1738} the picture
			\[\bl\Pi_\alpha, \ccP_{\alpha, B}, \ccP_{\alpha, F}\br
				=
				\bl\Pi_{\alpha-1}, \ccP_{\alpha-1, B}, \ccP_{\alpha-1, F}\br
				\conc
				\bl H_\alpha, \ccH_{\Pi_{\alpha-1}^{F(\alpha)}},
				\ccH_{F(\alpha)}\br\,,
			\]where the amalgamation occurs over $F(\alpha)$, is
			clean again and, therefore,
			our construction continues.
		\item This recursive process culminates with a final
			picture $\bl\Pi_N, \ccP_{N, B}, \ccP_{N, F}\br$.
			We order the vertices of~$\Pi_N$ in such a manner that
			each music line
			becomes an interval and the ordering of these intervals
			reflects the
			vertex ordering of the ordered graph~$G$. Since the copies
			in $\ccP_{N, B}\cup\ccP_{N, F}$ intersect each music
			line at most once,
			this ensures that these copies are ordered correctly in~$\Pi_N$.
			So we have converted $\bl\Pi_N, \ccP_{N, B}, \ccP_{N,
			F}\br$ into an
			ordered $(B, \ccB_F)$-conglomerate.
	\end{itemize}

	The partition relation $\ccP_{N, B}\lra (B, \ccB_F)^F_r$
	follows from the
	same standard argument we have already presented in the proof of
	Proposition~\ref{prop:CPL} and
	we omit the details.

	It remains to address the preservation of the clique
	number. To this end,
	we prove $\omega(\Pi_\alpha)=\omega(B)$ by induction on $\alpha\le N$.
	The base case $\alpha=0$ is trivial; in the induction step we
	assume $\omega(\Pi_{\alpha-1})=\omega(B)$ for some
	$\alpha\in [N]$ and consider
	a clique~$K$ in $\Pi_\alpha$. If~$K$ is completely
	contained in $H_\alpha$, then
	its order is at most
	\[
		\omega(H_\alpha)
		= \omega\bl\Pi_{\alpha-1}^{F(\alpha)}\br
		\le \omega(\Pi_{\alpha-1})
		= \omega(B)
	\]
	by Proposition~\ref{prop:CPL}\ref{it:CPLc}.

	Otherwise there
	are a standard
	copy $\Pi^\circ_{\alpha-1}$ of~$\Pi_{\alpha-1}$ and a
	vertex~$v$ in~$K$ such that  $v\in
	V(\Pi^\circ_{\alpha-1})\sm V(H_\alpha)$.
	Since~$v$ is
	adjacent to all
	other vertices of~$K$, we have $K\subseteq \Pi^\circ_{\alpha-1}$ and,
	consequently, the order of~$K$ is at most
	$\omega(\Pi_{\alpha-1})=\omega(B)$.
	This proves $\omega(\Pi_\alpha)\le \omega(B)$ and the
	reverse inequality is obvious.
\end{proof}

\section{Conform Ramsey conglomerates}
\label{sec:conform-cong}

For the proof of Proposition~\ref{prop:final} we recall
that~$\Ups_r^{(1)}$ from Proposition~\ref{prop:warmup} already provides
a clean conglomerate $(H,\ccH_B,\ccH_F)$ preserving the
clique number. In
particular, this conglomerate
already satisfies
property~\ref{it:final-a} of Proposition~\ref{prop:final}
and parts~\ref{def:conf-clean} and~\ref{def:conf-omega}
from Definition~\ref{def:conf} required by property~\ref{it:final-b}.
Consequently,
only parts~\ref{def:conf-oddg} and~\ref{def:conf-induced} of
Definition~\ref{def:conf}, which concern the odd girth and
$n$-inducedness, are missing. For $n=1$ these two missing properties
are trivial and, in fact,~$\Ups_r^{(1)}$ provides the
base case for an inductive proof.
For the inductive step, we run a partite construction, where we employ
$\Ups^{(n-1)}_r$ vertically and the clean partite lemma~$\CPL_r$ from
Proposition~\ref{prop:CPL} horizontally. Since the standard
construction of picture zero
consists of disjoint copies, it clearly preserves odd girth
and~$n$-inducedness.

Roughly speaking, the following picturesque statement rendered
in Lemma~\ref{lem:61} shows that these properties are kept during
the partite construction described above
(see statements~\ref{it:61a} and~\ref{it:61b} of the lemma).
However, we remark that in Lemma~\ref{lem:61}
we also require \(n\)-induced copies of \(F\) provided by the
partite lemma, which
is not stated to be a property of the clean partite lemma in
Proposition~\ref{prop:CPL}; but we shall address this issue in
Corollary~\ref{cor:CPLe}.

\begin{lemma}\label{lem:61}
	Given a graph~$F$, let $(B, \ccB_F)$ be an~$F$-system and
	let $(G, \ccG_B, \ccG_F)$
	be a weak $(B, \ccB_F)$-conglomerate. Suppose further that we have
	\[
		(\Sigma, \ccQ_B, \ccQ_F)
		=
		(\Pi, \ccP_B, \ccP_F)\conc (H, \ccH_{\Pi^\Fbu}, \ccH_\Fbu)
	\]
	for two pictures $(\Pi, \ccP_B, \ccP_F)$, $(\Sigma, \ccQ_B,
	\ccQ_F)$ over
	$(G, \ccG_B, \ccG_F)$ and an~$\Fbu$-partite $(\Pi^\Fbu,
	\ccP^\Fbu)$-conglomerate
	$(H, \ccH_{\Pi^\Fbu}, \ccH_{\Fbu})$
	with $\ccH_{\Pi^\Fbu}\subseteq \binom H{\Pi^\Fbu}_\dip$.
	Finally let~${n\ge 2}$ be an integer such that
	$\ccG_F\subseteq \binom GF_{n-1}$.
	\begin{enumerate}[label=\alabel]
		\item\label{it:61a}
			We have $\ccQ_F\subseteq \binom{\Sigma}{F}_n$ provided that
			\begin{enumerate}[label=\nlabel]
				\item\label{it:61a1}
					$\ccP_F\subseteq \binom{\Pi}{F}_n$ and
					$\ccH_\Fbu\subseteq \binom HF_n$,
				\item\label{it:61a2}
					the copies in~$\ccG_F$ have clean intersections,
				\item\label{it:61a3}
					and for all copies $(\Pi_\circ^\Fbu, \ccP^\Fbu_\circ)\in
					\ccH_{\Pi^\Fbu}$
					and $F_\star\in \ccH_\Fbu\sm \ccP^\Fbu_\circ$ the intersection
					of~$\Pi_\circ^\Fbu$ and~$F_\star$ is clean.
\end{enumerate}
		\item\label{it:61b} If $\og(G)\ge 2n-1$, $\og(\Pi)\ge 2n+1$,
			and $\og(H)\ge 2n+1$, then $\og(\Sigma)\ge 2n+1$.
	\end{enumerate}
\end{lemma}

\begin{proof}
	Starting with property~\ref{it:61a} we consider a copy $F_\star\in
	\ccQ_F$ and a
	path $P\not\subseteq F_\star$ of length at most~$n$ that
	connects two vertices~$x$
	and~$y$ of $F_\star$. We need to find some~$x$-$y$-path
	$R\subseteq F_\star$ which is
	shorter than~$P$. The canonical projection~$\psi$ of
	$\Sigma$ sends $F_\star$ to
	some copy $F'$ in the system~$\ccG_F$.

	\smallskip

	{\hspace{2em} \it First case: $F'=\Fbu$}

	\smallskip

	Suppose first that no inner vertex of~$P$ belongs to~$H$.
	Now the entire path~$P$
	is contained in a single standard copy $(\Pi_\circ,
	\ccP^\circ_B, \ccP^\circ_F)$
	of $(\Pi, \ccP_B, \ccP_F)$. So if $F_\star\in \ccP^\circ_F$
	we just need to appeal
	to the hypothesis $\ccP_F\subseteq \binom{\Pi}{F}_n$. On
	the other hand,
	if $F_\star\not\in \ccP^\circ_F$, then by~\ref{it:61a3} the
	intersection
	of $\Pi_\circ$ and $F_\star$ is clean, so that $xy$ is an
	edge of $F_\star$ and
	can play the r\^{o}le of the desired path~$R$.

	It remains to deal with the case that~$P$ is of the form
	$P=xP'zP''y$, where~$z$
	is an inner vertex of~$P$ belonging to $V(H)$ connected to
	$x$ and~$y$ by the
	paths~$P'$ and $P''$. Let the music line through~$z$
	intersect the copy $F_\star$
	in $z_\star$.

	We contend that there is an~$x$-$z_\star$-path $Q'\subseteq
	F_\star$ whose
	length is at most the length of~$P'$. To see this, we look at
	the $\psi(x)$-$\psi(z)$-path $\psi[P']$ in~$G$. If it is
	contained in~$\Fbu$,
	we just need to pull this path back to $F_\star$.
	Otherwise, $\Fbu\in \ccG_F\subseteq \binom GF_{n-1}$ yields
	a $\psi(x)$-$\psi(z)$-path $S'\subseteq \Fbu$ that is even
	shorter than $P'$ and
	the existence of $Q'$ follows again by pulling this path back.

	The same argument also yields a $z_\star$-$y$-path
	$Q''\subseteq F_\star$
	whose length is at most the length of~$P''$. The
	$x$-$y$-walk $xQ'z_\star Q''y$
	is contained in $F_\star$ and has at most the length of~$P$.
	If this walk fails
	to contain the desired path, then $P'$, $Q'$ have the same
	length, and so
	do $P''$, $Q''$. But then the entire path~$P$ is contained
	in~$H$ and the
	assumption $\ccH_\Fbu\subseteq \binom HF_n$ from~\ref{it:61a1}
	leads to the desired path.

	\smallskip

	{\hspace{2em} \it Second case: $F'\ne\Fbu$}

	\smallskip

	Let $\Pi_\circ$ denote the unique standard copy of~$\Pi$
	containing $F_\star$.
	If~$P$ crosses~$H$ in at most one vertex, then we have
	$P\subseteq \Pi_\circ$
	and $F_\star\in \binom{\Pi_\circ}F_n$ leads to the desired path.

	As we traverse~$P$ from~$x$ to~$y$, we can therefore
	assume that there are a
	first vertex~$u$ and a last vertex~$v$ belonging to~$H$,
	and that these vertices
	are distinct. We write $P=xP'uP''vP'''y$, where the paths
	$P', P'''\subseteq \Pi_\circ$
	might degenerate to single vertices (in case $x=u$ or
	$v=y$), while~$P''$ has at
	least one edge.

	Suppose first that $P''\subseteq H$. Since
	$\Pi_\circ^{\Fbu}$ is distance preserving in~$H$,
	there exists a~$u$-$v$-path $Q''\subseteq \Pi_\circ^{\Fbu}$ of at
	most the same length as $P''$.
	Now $xP'uQ''vP'''y$ is an~$x$-$y$-walk in $\Pi_\circ$ of at
	most the same length as~$P$
	and $F_\star\in\binom{\Pi_\circ}F_n$ leads again to the
	desired path. So we can assume
	$P''\not\subseteq H$ in the sequel.

	Suppose next that $u=x$ and $v=y$. By~\ref{it:61a2} the intersection of~$F'$ and~$\Fbu$ is
	clean; as~$\psi(x)$, $\psi(y)$
	belong to both graphs, we have $\psi(x)\psi(y)\in E(G)$.
	Consequently, $xy$
	is an edge of $F_\star$ and our claim is again clear.
So we can assume from now on that $P''$ has length at most $n-1$.

	By $\Fbu\in\binom GF_{n-1}$ applied to the
	$\psi(u)$-$\psi(v)$-walk $\psi[P'']$,
	which is not contained in~$\Fbu$, there is a $\psi(u)$-$\psi(v)$-path
	$R''\subseteq \Fbu$, which is shorter than $P''$. Now
	$\psi[xP'u]R''\psi[vP'''y]$
	is a $\psi(x)$-$\psi(y)$-walk in~$G$ and shorter than~$P$.
	Due to $F'\in\ccG_F\subseteq\binom GF_{n-1}$
	this entails that some $\psi(x)$-$\psi(y)$-path $S\subseteq
	F'$ is shorter
	than~$P$ as well, wherefore the desired~$x$-$y$-path
	$R\subseteq F_\star$ exists again.

	Having thereby proved~\ref{it:61a} we proceed with~\ref{it:61b}.
	Assume towards a contradiction that~$\Sigma$ contains an odd
	cycle~$C$ whose
	length is at most $2n-1$. The projection~$\psi$ sends~$C$
	to an odd closed walk~$D$
	in~$G$ of the same length as~$C$. Due to $\og(G)\ge 2n-1$
	this implies that both~$C$
	and~$D$ are cycles of length $2n-1$. If~$C$ intersects~$H$
	in at most one vertex,
	then~$C$ has to be contained in a single standard copy of
	$\Pi$, contrary
	to $\og(\Pi)\ge 2n+1$. Thus we can assume that $s=|V(C)\cap
	V(H)|$ is at least~$2$.
	Consequently, $C$ is of the form
	\begin{equation}\label{eq:2023}
		C=x_1P_1\dots x_sP_s\,,
	\end{equation}
	where $V(C)\cap V(H)=\{x_1, \dots, x_s\}$
	and $P_\rho$ is an $x_\rho$-$x_{\rho+1}$-path for every
	$\rho\in\ZZ/s\ZZ$.

	\begin{clm}
		If for some $\rho\in\ZZ/s\ZZ$ the length of $P_\rho$ is
		at most $n-1$,
		then this path is contained in~$H$.
	\end{clm}

	\begin{proof}
		Assume contrariwise that the length of $P_\rho$ is at
		most $n-1$, while
		$P_\rho\not\subseteq H$. Since~$\psi$ is injective on
		$V(C)$, the image
		$\psi[P_\rho]$ is a
		$\psi(x_\rho)$-$\psi(x_{\rho+1})$-path in~$G$. Together
		with $\psi[P_\rho]\not\subseteq \Fbu$ and $\Fbu\in \binom
		GF_{n-1}$ this shows
		that there exists a
		$\psi(x_\rho)$-$\psi(x_{\rho+1})$-path $Q\subseteq \Fbu$
		whose length is strictly smaller than the length of $P_\rho$. Both
		\[
			Q\psi[x_{\rho+1}P_{\rho+1}\dots x_\rho]
			\quad \text{ and } \quad
			\psi[x_\rho P_\rho x_{\rho+1}]Q^{-1}\,,
		\]
		where~$Q^{-1}$ denotes~$Q$ traversed in reverse order,
		are closed walks in~$G$
		whose lengths are at most $2n-2$. Moreover, the parities
		of their lengths are different
		and, therefore, one of them is odd. Thereby we reach a
		contradiction to $\og(G)\ge 2n-1$.
	\end{proof}

	Since $\og(H)\ge 2n+1$ implies $C\not\subseteq H$, it
	follows that one of the paths
	$P_1, \dots, P_s$ has length at least~$n$. This path needs
	to be unique and the other
	paths are contained in~$H$. For these reasons we can
	rewrite~\eqref{eq:2023}
	as $C=xPyQ$, where $x, y\in V(C)\cap V(H)$, the
	$x$-$y$-path~$P$ has no inner vertices
	in~$H$, and the~$y$-$x$-path $Q\subseteq H$ has length at
	most~${n-1}$.
	Clearly,~$P$ is completely contained in a single standard
	copy $(\Pi_\circ, \ccP^\circ_B, \ccP^\circ_F)$. As the
	constituent of this picture
	over~$\Fbu$ is distance preserving in~$H$, it contains an
	$x$-$y$-path~$R$ whose length
	is at most the length of~$Q$. Now
	\[
		xRyQ
		\quad \text{ and } \quad
		xPyR^{-1}
	\]
	are closed walks in~$H$ and $\Pi_\circ$, respectively, whose
	lengths are at most $2n-1$.
	Moreover, their parities are different, so that one of them
	is odd. Thus we get a
	contradiction either to $\og(H)\ge 2n+1$ or to
	$\og(\Pi_\circ)\ge 2n+1$.
\end{proof}

Keeping a promise made in Remark~\ref{rem:57} we can now
easily deduce a further
property of the clean partite lemma.

\begin{cor}\label{cor:CPLe}
	Let $(B, \ccB_F)$ be a partite~$F$-system with clean
	intersections. If for some
	natural number~$n$ we have $\ccB_F\subseteq \binom BF_n$, then for
	every number of colours~$r$ the clean
	conglomerate $\CPL_r(B, \ccB_F)=(H, \ccH_B, \ccH_F)$
	satisfies $\ccH_F\subseteq \binom HF_n$.
\end{cor}

\begin{proof}
	This is clear for $n=1$, so we can assume $n\ge 2$ from now
	on. Let us recall that
	$(H, \ccH_B, \ccH_F)$ is obtained by means of a partite
	construction. We keep using
	the notation from the explanation of this construction
	given immediately before
	Proposition~\ref{prop:CPL}.

	Owing to the items~\ref{it:HJe} and~\ref{it:HJf} of
	Proposition~\ref{lem:HJconstr}
	we know $\ccG_F\subseteq \binom GF_{n-1}$ and that the
	copies in~$\ccG_F$ have
	clean intersections. We want to establish
	$\ccP_{\alpha, F}\subseteq \binom{\Pi_\alpha}{F}_n$ for
	every non-negative integer
	$\alpha\le N$ by induction on~$\alpha$. The case $\alpha=N$
	of this assertion
	will then prove our corollary.

	The base case $\alpha=0$ is clear, because picture zero is
	just a disjoint union
	of copies of $(B, \ccB_F)$ and $\ccB_F\subseteq \binom
	BF_n$ is assumed.
	Suppose now that we already know
	$\ccP_{\alpha-1, F}\subseteq \binom{\Pi_{\alpha-1}}{F}_n$
	for some $\alpha\in [N]$. Utilising
	Proposition~\ref{lem:HJconstr} one easily
	checks that all assumptions of
	Lemma~\ref{lem:61}\ref{it:61a} are satisfied
	for the partite amalgamation
	\[
		\bl\Pi_\alpha, \ccP_{\alpha, B}, \ccP_{\alpha, F}\br
		=
		\bl\Pi_{\alpha-1}, \ccP_{\alpha-1, B}, \ccP_{\alpha-1, F}\br
		\conc
		\bl H_\alpha, \ccH_{\Pi_{\alpha-1}^{F(\alpha)}}, \ccH_{F(\alpha)}\br
	\]
	and thus we have indeed $\ccP_{\alpha, F}\subseteq
	\binom{\Pi_\alpha}{F}_n$.
\end{proof}

We conclude by establishing the following strengthening of
Proposition~\ref{prop:final}.

\begin{prop}\label{prop:Ups}
	For every ordered graph~$F$ and every integer $n\geq 1$
	there exists a construction~$\Ups^{(n)}$
	which, given an ordered~$F$-system $(B, \ccB_F)$ with clean
	intersections
	satisfying $\ccB_F\subseteq\binom BF_n$ and a number of
	colours~$r$, yields
	an ordered $(B, \ccB_F)$-conglomerate $\Ups^{(n)}_r(B,
	\ccB_F)=(H, \ccH_B, \ccH_F)$
	such that
	\begin{enumerate}[label=\alabel]
		\item\label{it:46a} $\ccH_B\lra (B, \ccB_F)^F_r$,
		\item\label{it:46b}
			\begin{enumerate}[label=\nlabel]
				\item\label{it:46bi} $\ccH_F\subseteq \binom HF_n$,
				\item\label{it:46bii} $\omega(H)=\omega(B)$,
				\item\label{it:46biii} $\og(H)\ge \min\{2n+1, \og(B)\}$,
			\end{enumerate}
		\item\label{it:46c} and the $(B, \ccB_F)$-conglomerate $(H,
			\ccH_B, \ccH_F)$ is clean.
	\end{enumerate}
\end{prop}

Before proving Proposition~\ref{prop:Ups} we note that it implies
Proposition~\ref{prop:final}. First we observe that the assumptions
of Proposition~\ref{prop:Ups} are less restrictive. Indeed,
given an~$n$-conform ordered
$F$-system $(B, \ccB_F)$, as assumed in
Proposition~\ref{prop:final}, conditions~\ref{def:conf-clean}
and~\ref{def:conf-induced} of Definition~\ref{def:conf} ensure that
$(B, \ccB_F)$ has clean intersections and
$\ccB_F\subseteq\binom{B}{F}_n$.

Concerning the conclusion of Proposition~\ref{prop:Ups}, we note that
the ordered $(B, \ccB_F)$-conglomerate
$\Ups^{(n)}_r(B, \ccB_F)=(H, \ccH_B, \ccH_F)$ clearly
satisfies the Ramsey property~\ref{it:final-a}
of Proposition~\ref{prop:final}.
Moreover, part~\ref{it:46c} of Proposition~\ref{prop:Ups}
implies that~$\ccH_F$ has clean
intersections. The other three requirements needed for the
$n$-conformity
of $(H,\ccH_F)$, asserted in part~\ref{it:final-b}
of Proposition~\ref{prop:final}, follow directly from the
three properties in part~\ref{it:46b} of Proposition~\ref{prop:Ups}
combined with the assumption that $(B,\ccB_F)$ is~$n$-conform itself.

\begin{proof}
	We argue by induction on~$n$, using the construction
	$\Ups^{(1)}$ provided by
	Proposition~\ref{prop:warmup} as a base case. Assuming now
	that for some $n\ge 2$
	we have the construction~$\Ups^{(n-1)}$ at our disposal and that an
	ordered~$F$-system $(B, \ccB_F)$ with clean intersections
	such that $\ccB_F\subseteq \binom BF_n$
	as well as a number of colours~$r$ are given, we perform
	the same partite construction
	as in the proof of Proposition~\ref{prop:warmup}, the only
	difference being that
	vertically we use the conglomerate $(G, \ccG_B,
	\ccG_F)=\Ups^{(n-1)}_r(B, \ccB_F)$.
	Again this yields a sequence $\bl\Pi_\alpha, \ccP_{\alpha, B},
	\ccP_{\alpha, F}\br_{\alpha\le N}$
	of clean pictures and the last of them can be converted into a
	clean ordered $(B, \ccB_F)$-conglomerate, which we shall denote
	by $\Ups^{(n)}_r(B, \ccB_F)=(H, \ccH_B, \ccH_F)$ from now on.
	In particular, $(H, \ccH_B, \ccH_F)$ satisfies
	property~\ref{it:46c} of Proposition~\ref{prop:Ups}.

	The partition relation~\ref{it:46a} holds for the usual
	reason. Assertion~\ref{it:46bii} of part~\ref{it:46b}
	follows from arguments already presented in the proof of
	Proposition~\ref{prop:warmup}, which we do not repeat.

	Next, a straightforward induction on~$\alpha$ based on
	Lemma~\ref{lem:61}\ref{it:61a},
	Proposition~\ref{prop:CPL}, Corollary~\ref{cor:CPLe}, and
	our induction hypothesis
	on~$\Ups^{(n-1)}$ reveals $\ccP_{\alpha, F}\subseteq
	\binom{\Pi_\alpha}F_n$
	for every $\alpha\le N$.
	The case $\alpha=N$ of this assertion proves~\ref{it:46bi}
	of part~\ref{it:46b}.

	Finally,~\ref{it:46biii} only requires attention if $\og(B)\ge 5$. Now
	$\ol{n}=\min\bigl\{n, \frac12(\og(B)-1)\bigr\}$ is at least~$2$ and
	we can apply Lemma~\ref{lem:61}\ref{it:61b} with~$\ol{n}$
	instead of~$n$.
	This allows us to show $\og(\Pi_\alpha)\ge 2\ol{n}+1$ for
	every $\alpha\le N$
	by induction on~$\alpha$, and thus we have indeed
	$\og(H)=\og(\Pi_N)\ge 2\ol{n}+1$.
\end{proof}

\section{Concluding remarks}

\subsection{Generalisations of the main result}
The first generalisation concerns the girth of the canonical Ramsey graph~$H$ in
Theorem~\ref{thm:main-simple}. It seems likely that, as long as \(F\) is not a forest,
Theorem~\ref{thm:main} remains valid when we add the demand that \(H\) and \(F\)
should have the same girth. In the non-canonical setting, such a result was obtained
recently by Reiher and R\"odl. In fact,
the proof of their result provides
a system~$\ccH_F$ of copies of~$F$ in~$H$, which itself contains
no short cycles (see~\cite{RR-girth}*{\ssign5}), and we believe that those results
extend to the canonical setting.

The second generalisation concerns hypergraphs. The proof of Theorem~\ref{thm:main-simple}
presented here can be partly extended to linear hypergraphs. For hypergraphs
the clique number is replaced by the size of the largest subset~$X$
of vertices such that any two vertices in~$X$ are contained in a hyperedge.
However, it is perhaps less clear what the analogous concepts of odd girth,
$n$-inducedness, and distance preservation for such hypergraphs are.
It seems plausible that versions of
Theorems~\ref{thm:main-simple} and~\ref{thm:main} hold for general hypergraphs.

The last generalisation we briefly mention concerns a common generalisation of
our result and Theorem~\ref{thm:41} in the context of Ramsey classes, when we colour
arbitrary subgraphs. In that context an induced canonical Ramsey theorem
with maintained clique number was provided
by Pr\"omel and Voigt~\cite{PV85} and it would be of some interest to establish
a result subsuming both their work and Theorem~\ref{thm:main}.

\subsection{Classical versus canonical Ramsey theory}
As it turned out, many results in Ramsey theory can be
paralleled in the canonical context.
For example, it follows from the work of
Erd\H os, Hajnal, and Rado~\cite{EHR65}*{\ssign16.4} (see
also reference~\cite{EH89}*{(4.2)})
and of Shelah~\cite{Sh96} that in terms of the number of exponentiations
the canonical Ramsey number and the non-canonical
Ramsey number (for at least four colours) for cliques in
hypergraphs display the same behaviour.
Another example, in the somewhat different context of random graphs,
concerns the threshold for the corresponding Ramsey
properties. For graph
cliques with at least four vertices those thresholds coincide,
which follows from the work of
R\"odl and Ruci\'nski~\cites{RR93,RR95}
and of Kam\v cev and Schacht~\cite{KS25}.
A far-reaching generalisation of these parallel developments
would be an affirmative answer to the question of whether having the Ramsey property
for sufficiently many colours yields the canonical Ramsey property as well.
For graphs this na\"{\i}ve question reads as follows.

\begin{nquest}
	Given an ordered graph~$F$, does there exist an integer~$r$ such that
	for every ordered graph~$H$ that satisfies the relation
	\(H \lra (F)_r\)
	we also have
	\(H \clra (F)\)?
\end{nquest}
Note that for graphs~$F$ with at most two edges this question is trivial, since
any colouring of such small graphs is canonical.

More interestingly, by applying the machinery developed in~\cite{RR-girth} we
can answer this question in the negative when we replace~$H$
by a system~$\ccH_F$ of copies of an ordered graph~$F$ with at least three edges.
For such graphs there are at least five equivalence relations on~\(E(F)\) and, hence,
at least one of them corresponds to a non-canonical colouring. We fix such
a non-canonical equivalence relation,
say \(\equiv^F\), and let \(r\in\NN\) be given.

With the notation developed in~\cite{RR-girth} the pair \((F, \equiv^F)\) is a pretrain
and~\cite{RR-girth}*{Proposition~9.1} applied to \((F, \equiv^F)\) and~$r$ yields
a pretrain system \((H, \equiv^H, \ccH_F)\) with the Ramsey property
$\ccH_F\lra (F)_r$. Moreover, the restriction of~$\equiv^H$ to copies of~$F$ in~$\ccH_F$ is identical to~$\equiv^F$. So, we can colour $E(H)$
in such a way that no copy of~$F$ from~$\ccH_F$ is canonically coloured.

Consequently, the version of the question for the system~$\ccH_F$
has the expected negative answer.
Moreover, if~\(F=K_k\) for some \(k\ge 3\) (and for certain other graphs~$F$), the
constructions involved in the proof of~\cite{RR-girth}*{Proposition~9.1}
yield \(\ccH_{K_k}=\binom{H}{K_k}\). This answers the question posed above
negatively for such instances.

Nevertheless, in light of the discussion above one may wonder whether these parallel
developments in canonical and non-canonical Ramsey theory are just anecdotal
or whether there is some profound connection.

\end{document}